\documentclass[letterpaper,10pt,reqno,onefignum,onetabnum]{amsart}
\usepackage[english]{babel}
\usepackage{amsmath}
\usepackage{amsthm}
\usepackage{verbatim}
\usepackage[foot]{amsaddr}
\usepackage{delimset}
\usepackage[left]{lineno}
\usepackage[dvipsnames]{xcolor}
\usepackage{hyperref}
\hypersetup{
	colorlinks = true,
	linkcolor = OliveGreen,
	anchorcolor = OliveGreen,
	citecolor = OliveGreen,
	filecolor = OliveGreen,
	urlcolor = OliveGreen
}
\usepackage{algorithm}% http://ctan.org/pkg/algorithms
\usepackage{algorithmic}% http://ctan.org/pkg/algorithms
\usepackage{float}
\usepackage{lipsum}
\usepackage{amsfonts}
\usepackage{amssymb}
\usepackage{graphicx}
\usepackage{epstopdf}
\usepackage{cancel}
\usepackage{cases}
\usepackage{multirow}
\ifpdf
\DeclareGraphicsExtensions{.eps,.pdf,.png,.jpg}
\else
\DeclareGraphicsExtensions{.eps}
\fi
\def\urlprefix{}

\newtheorem{teo}{Theorem}[section]
\newtheorem{prop}[teo]{Proposition}
\newtheorem{lem}[teo]{Lemma}
\newtheorem{cor}[teo]{Corollary}
\newtheorem{pro}[teo]{Problem}
\newtheorem{algo}[teo]{Algorithm}
\newtheorem{asume}[teo]{Assumption}
\newtheorem{notation}[teo]{Notation}
\newtheorem{rem}[teo]{Remark}

\usepackage{algorithmic}
\usepackage{tikz}
\usepackage{booktabs}%
\usetikzlibrary{matrix}
\usetikzlibrary{arrows}
\usepackage{mathrsfs}
\usepackage{bm}
\usepackage{color}
\usepackage{xcolor}
\usepackage{subcaption}

\newcommand{\N}{\mathbb N}

\newcommand{\R}{\mathbb R}

\renewcommand{\H}{\mathcal{H}}
\newcommand{\G}{\mathcal G}
\newcommand{\K}{\mathcal K}

\newcommand{\HH}{{\bm{\mathcal{H}}}}

\newcommand{\Id}{{\bf Id}}
\newcommand{\id}{\textnormal{Id}}
\newcommand{\x}{\bm x}
\newcommand{\y}{\bm y}
\newcommand{\bu}{\bm u}
\newcommand{\p}{\bm p}

\newcommand{\weak}{\rightharpoonup}
\newcommand{\ran}{\textnormal{ran}\,}
\newcommand{\dom}{\textnormal{dom}\,}

\newcommand{\zer}{\textnormal{zer}}
\newcommand{\fix}{\textnormal{Fix}\,}
\newcommand{\gra}{\textnormal{gra}\,}

\newcommand{\scal}[2]{{\left\langle{{#1}\mid{#2}}\right\rangle}}

\newcommand{\menge}[2]{\big\{{#1}~\big |~{#2}\big\}}

\newcommand{\RR}{\ensuremath{\mathbb{R}}}
\newcommand{\RP}{\ensuremath{\left[0,+\infty\right[}}

\newcommand{\RPP}{\ensuremath{\left]0,+\infty\right[}}

\newcommand{\RX}{\ensuremath{\left]-\infty,+\infty\right]}}

\newcommand{\weakly}{\ensuremath{\:\rightharpoonup\:}}
\newcommand{\minimize}[2]{\ensuremath{\underset{\substack{{#1}}}%
		{\mathrm{minimize}}\;\;#2 }}

\usepackage{geometry}
\numberwithin{equation}{section}

\numberwithin{equation}{section}

\DeclareFontEncoding{FMS}{}{}
\DeclareFontSubstitution{FMS}{futm}{m}{n}
\DeclareFontEncoding{FMX}{}{}
\DeclareFontSubstitution{FMX}{futm}{m}{n}
\DeclareSymbolFont{fouriersymbols}{FMS}{futm}{m}{n}
\DeclareSymbolFont{fourierlargesymbols}{FMX}{futm}{m}{n}
\DeclareMathDelimiter{\nr}{\mathord}{fouriersymbols}{152}{fourierlargesymbols}{147}

\DeclareMathOperator*{\argmin}{arg\,min}
\DeclareMathDelimiter{\nr}{\mathord}{fouriersymbols}{152}{fourierlargesymbols}{147}
\DeclareMathAlphabet{\mathpzc}{OT1}{pzc}{m}{it}

\title[NFB for solving non-monotone inclusions]{Nonlinear Forward-Backward Algorithm for Solving Non-monotone+Lipschitz Inclusions with Applications to Adjoint Mismatch Problems}

\author{Jean-Christophe Pesquet$^{1}$}
\author{Fernando Rold\'an$^{2}$}
\address{$^{1}$Universit\'e Paris-Saclay, CentraleSup\'elec, CVN, Inria, Gif-sur-Yvette 91190, France.}
\address{$^{2}$ Departamento de Ingeniería Matemática and CI$^2$MA, Universidad de Concepción, Concepción, Chile.}
\email{jean-christophe.pesquet@centralesupelec.fr}
\email{fernandoroldan@udec.cl (corresponding author)}
\begin{document}
	%\linenumbers
	\begin{abstract}
		This article presents a novel algorithmic framework for solving a broad class of
nonlinear inclusion problems involving non-monotone and Lipschitz continuous operators.
Our primary motivation originates from the study of adjoint mismatch problems,
which frequently arise in inverse problems and data science
when the adjoint of a linear measurement operator is replaced by an approximation.
Because this approximation may inherently destroy classical monotonicity properties,
standard convergence guarantees for splitting methods may no longer apply. To address this,
we consider a Nonlinear Forward-Backward algorithm
and rigorously establish its convergence without assuming monotonicity.
By leveraging a warped resolvent formulation and semimonotonicity  assumptions,
we derive explicit conditions ensuring both weak convergence and, under stronger assumptions, R-linear convergence.
Furthermore, our theoretical analysis 
yields new convergence guarantees for several popular
methods, including the Condat–V\~u and Forward-Half-Reflected-Backward algorithms.
Finally, we illustrate our theoretical findings and demonstrate the applicability
%validity 
of our approach through numerical experiments in signal recovery.
		\par
		\bigskip
		
		\noindent \textbf{Keywords.} {\it Splitting algorithms, convergence analysis, fixed point theory, convex optimization, adjoint mismatch}
		\par
		\bigskip \noindent
		2020 {\it Mathematics Subject Classification.} {47H05, 47H10, 65K05, 90C25.}
		%62H35, 94A08,
		
	\end{abstract}
	
	\maketitle

	\section{Introduction}
	The main objective of this paper is to develop numerical methods for solving the following problem.
	\begin{pro}\label{pro:main}
		Let $\H$ and $\G$
		be real Hilbert spaces. Let $L\colon \H \to \G$ be 
		a bounded linear operator with adjoint $L^*\colon \G \to \H$.
		Let $A\colon \H \to 2^\H$ and $B\colon\G \to 
		2^\G$ 
		be set-valued operators, 
		let 
		$C\colon\H \to \H$ be a 
		$\beta$-cocoercive operator for $\beta \in \RPP$, and let $D\colon\H 
		\to \H$ be a $\vartheta$-Lipschitzian operator for $\vartheta \in 
		\RP$. The problem is to
        \begin{equation}\label{eq:mainPD} 
			\text{find }  (x,u) \in \H \times \G \text{ such that } 
			\begin{cases}
				0 &\in (A+C+
                D)x+L^*u\\
				u &\in B Lx ,
			\end{cases} 
		\end{equation}
		under the assumption that its solution set, denoted by 
		$\bm{Z}$,  is nonempty.
	\end{pro}
    This problem encompasses several scenarios in convex programming \cite{Combettes2018MP}, game theory \cite{Nash13}, data science \cite{CombettesPesquet2021strategies}, and image recovery \cite{BotHendrich2014TV,Briceno2011ImRe}, among others.
	Note that, given $(\hat{x},\hat{u}) \in \bm{Z}$, $\hat{x}$ is a 	solution to the primal inclusion
	\begin{equation}\label{eq:primalinclu}
		\text{find }  x \in \H \text{ such that } 
		0 \in (A+L^*BL+C+D)x,
	\end{equation}
	and $\hat{u}$ is a solution to the dual inclusion
	\begin{equation}\label{eq:dualinclu}
		\text{find }  u \in \G \text{ such that } 
		0 \in  B^{-1}u - L(A+C+D)^{-1}(-L^*u).
	\end{equation}
	An instance of Problem~\ref{pro:main} is the following optimization  problem.
	\begin{pro}\label{prob:problemopti}
		%In the context of Problem~\ref{pro:main}, 
		Let $(\mu_f,\mu_g)\in \R^2$ and $(\beta,\vartheta_d)\in \RPP\times\RP$.
		Let $\H$, $\G$, and   $\K$
		be real Hilbert spaces.  
		Let $L\colon \H \to \G$ and $T \colon \H \to \K$ be bounded lnear operators.
		Let $f\colon \H \to ]-\infty,+\infty]$ be a proper, 
		lower-semicontinuous, 
		$\mu_f$-convex  function, let $g\colon \G \to ]-\infty,+\infty]$ be a proper,  
		lower-semicontinuous, 
		$\mu_g$-convex function, let
		$h\colon \H \to \R$ be a differentiable convex function
		with a $1/\beta$-Lipschitzian gradient, and let $d\colon \K \to \R$ be a differentiable 
		convex function with a $\vartheta_d$-Lipschitzian 
		gradient. The problem is to 
		\begin{equation}\label{eq:problemopti}
			\minimize{x\in \H}f(x)+g(Lx)+h(x)
			+ d(Tx),
		\end{equation}
		under the assumption that the set of solutions is nonempty.
	\end{pro}
	Indeed, under standard qualification conditions, Problem~\ref{prob:problemopti} can be written as \eqref{eq:primalinclu} by considering  
    $A=\partial f$, $B=\partial g$, $C=\nabla h$, $D=T^*\circ\nabla d\circ T$, and $\vartheta = \|T\|^2\vartheta_d$, where $\partial f$ (resp. $\partial g$) denotes the Fréchet subdifferential of $f$ (resp. $g$). 
    In turn, replacing $T^*$ in the definition of $D$ by a bounded linear operator $K\colon\K\to\H$ amounts to introducing an \emph{adjoint mismatch}. In this case, $D$ is no longer guaranteed to be a cocoercive operator, but it is $\vartheta$-Lipschitzian with $\vartheta =\|K\|\|T\| \vartheta_d$. Even though $A$, $B$, and $C$ are monotone operators, the monotonicity of the overall inclusion is not guaranteed.

	If the operators $A$, $B$, 
	and $D$ are monotone, then \eqref{eq:mainPD} defines a monotone inclusion, which can be solved by the Primal-Dual Method with Block Triangular Resolvent (PDBTR) proposed in \cite{MorinBanertGiselsson2022}. This method is derived from the Nonlinear Forward-Backward algorithm under suitable metrics and Lipschitzian operators (see also the related warped-resolvent framework in \cite{BuiCombettesWarped2020,Combettes2024AN,Giselsson2021NFBS,TangMM20251,TangMM20252}). In particular, when $D=0$, PDBTR reduces to the Condat--V\~u (CV) algorithm \cite{Condat13,Vu13} and to Chambolle--Pock (CP) if $C=0$ \cite{ChambollePock2011}. On the other hand, when $L=0$ and $B=0$, it reduces to the Forward-Half-Reflected-Backward (FHRB) algorithm proposed in \cite{Malitsky2020SIAMJO}. Moreover, if $L=0$ and $D=0$, the Forward-Backward (FB) algorithm is recovered \cite{passty1979JMAA}. If $D=0$, $\H=\G$, and $L=\id$, Problem~\ref{pro:main} can be solved by the Three-Operator Splitting (TOS) proposed in \cite{DavisYin2017}, which further reduces to Douglas--Rachford (DR) splitting \cite{DR1956,Eckstein1992,Lions1979SIAM} if $C=0$. On the other hand, if $B=0$ and $L=0$, Problem~\ref{pro:main} can be solved by the Forward-Backward-Half-Forward (FBHF) algorithm proposed in \cite{BricenoDavis2018} which reduces to FB when $D=0$ and to Forward-Backward-Forward (FBF) when $C=0$
	\cite{Tseng2000SIAM}. Other methods for solving this problem are available in \cite{AttouchBricenoCombettes2010,CombettesMinh2022,Comb13,CombettesEckstein2018MP,Roldan20254op}. 

    The adjoint mismatch problem arises when the adjoint $T^*$ of a linear measurement operator $T$ must be approximated because of physical or computational limitations, or efficiency requirements. It is particularly common in inverse problems such as computed tomography, where its importance was first highlighted in \cite{ZengGullberg2000}, and has since been widely studied \cite{ChouzenouxMismatch2023,ChouzenouxMismatch2021,ChouzenouxPesquetRoldan2023,DongHansen2019,ElfvingHansen2018,LorenzRose2018,LorenzMMCP2023,Savanier2022}. Algebraic analyses of mismatched quadratic problems corresponding to \eqref{eq:problemopti} with $f=g=h=L=0$ and $d=\widetilde{d}:=\alpha\|\cdot\|^2/2$, $\alpha\in\RP$, were provided in \cite{DongHansen2019,ElfvingHansen2018}, while a mismatched Kaczmarz method was studied in \cite{LorenzRose2018}. Mismatched and preconditioned versions of FB were proposed in \cite{ChouzenouxMismatch2021,Savanier2022}. For $h=0$ and $d=\widetilde{d}$, mismatched versions of the CV, Loris--Verhoeven \cite{LorisVerhoeven2011}, and Combettes--Pesquet \cite{CombettesPesquet2011PD} algorithms were developed in \cite{ChouzenouxMismatch2023}, allowing iteration-dependent approximations $K_n\colon\K\to\H$ of $K$. A mismatched CP method was introduced in \cite{LorenzMMCP2023}, whereas mismatched FBHF and Forward--Douglas--Rachford–Forward methods \cite{RyuVu2020} were proposed in \cite{ChouzenouxPesquetRoldan2023} for Problem~\ref{pro:main} with $B=L=0$.

    To address scenarios beyond those covered by the preceding methods, we consider inclusion problems outside the monotone setting. Over the years, several authors have investigated algorithms that relax the monotonicity assumption on the operators. For instance, in \cite{CombettesPennanen2004,IusemPennanenSvaiter2003,Kohlenbach2022}, the proximal point algorithm (PPA) \cite{martinet1970} was studied for $\mu$-comonotone operators, including the cohypomonotone case $\mu<0$. 
	Moreover, the convergence of the PPA was analyzed under metric subregularity and submonotonicity in \cite{LukeTam2025}, for (strongly) quasiconvex functions in \cite{IusemLara2022,Lara2022}, and for prox-convex functions in \cite{GradLara2022,GradLaraMarcavilla2024}. Extragradient-type algorithms for $\rho$-comonotone operators were studied in \cite{Bohm2022,LeeKIM2021,TranDinh2024} and in \cite{DiakonikolasDaskalakisJordan2020} under a weak Minty condition. The convergence of the Douglas--Rachford algorithm and the Chambolle--Pock algorithm under a weak Minty condition was established in \cite{EvensLatafatPatrinos2025CP,EvensPasLatafatPatrinos2025DR}. In \cite{AlacaogluKimWright2024}, the authors studied the convergence of Halpern and Krasnosel'ski\u{\i}--Mann iterations for solving minimax problems under a weak Minty condition. Warped-resolvent algorithms involving $\mu$-comonotone operators have been studied in \cite{PapadimitriouVu2023,TrangNguyen2026}. Note that in these works the notions of comonotonicity and Lipschitz continuity are defined with respect to the operator inducing the warped resolvent. Finally, in \cite{BauschkeMoursiXianfu2021}, the authors investigated properties of $\mu$-monotone and $\rho$-comonotone operators and their resolvents.
	
	In this article, we propose a primal-dual method for solving Problem~\ref{pro:main}. Its convergence analysis is based on a nonlinear forward-backward algorithm for comonotone inclusions. In particular, we extend the convergence analysis of the Nonlinear Forward-Backward algorithm with momentum from \cite{MorinBanertGiselsson2022} to comonotone operators. We analyze conditions on operators $A$, $B$, $C$, and $D$ to ensure that the primal-dual operators associated with our method are comonotone, 
    thereby recovering, as a special case, the conditions established in \cite{EvensLatafatPatrinos2025CP} for CP beyond the monotone setting. Moreover, our conditions apply to CV, FHRB, and FB in the comonotone setting.	
	
	The structure of the paper is as follows.
	In Section~\ref{se:notation}, we present the
	notation and necessary mathematical background. In Section~\ref{sec:MNFB},
	we introduce an auxiliary inclusion, present the Nonlinear Forward-Backward algorithm, and study its convergence for comonotone operators. In Section~\ref{sec:MR}, we
	provide conditions guaranteeing the comonotonicity of the associated primal-dual operators and establish the convergence of the proposed methods for solving Problem~\ref{pro:main}. Finally, in Section~\ref{se:numexp},
	we present an application to adjoint mismatch problems and numerical experiments in signal recovery.
	
	\section{Notation and Preliminaries}\label{se:notation}
	Throughout this paper, $\H$ and $\G$ are real Hilbert spaces with inner 
	product $\scal{\cdot}{\cdot}$ and associated norm $\|\cdot 
	\|$. The symbols $\weakly$ and $\to$ denote the weak and strong 
	convergence, respectively. The identity operator is denoted by $\id$ and the space of bounded linear operators from $\H$ to $\G$ is denoted by $\mathcal{B}(\H,\G)$. Given a linear operator $M \in \mathcal{B}(\H,\G)$, we denote its adjoint 
	by $M^* \in \mathcal{B}(\G,\H)$ and, when $\G=\H$ define $\scal{\cdot}{\cdot}_{M}=\scal{\cdot}{M\cdot}$. If  
    $S \in \mathcal{B}(\H,\H)$ is a self-adjoint strongly monotone operator, i.e, there exists $\mu\in \RPP$ such that, for every $x \in \H$, $\scal{Sx}{x}\geq \mu \|x\|^2$, then, $\scal{\cdot}{\cdot}_{S}$ is an inner product on $\H$. The norm of $x\in \HH$ induced by $S$ is denoted by $\|x\|_S=\sqrt{\scal{x}{x}_{S}}$. 
    Let $\mathcal{E}\subset\H$ and $R\colon \mathcal{E} \to \H$ and let $\beta \in \left]0,+\infty\right[$. The operator $R$ is 
	$\beta$-cocoercive with respect to $S$ if 
	\begin{equation} \label{def:coco}
		(\forall x \in \mathcal{E}) (\forall y \in \mathcal{E})\quad \langle x-y \mid Rx-Ry 
		\rangle 
		\geq \beta \|Rx - Ry 
		\|^2_{S^{-1}}.
	\end{equation}
	Note that this definition is equivalent to the cocoercivity of ${S}^{-1}\circ R$ in the Hilbert space $(\H,\scal{\cdot}{\cdot}_{{S}})$ (see \cite[Section~1.2]{MorinBanertGiselsson2022}).
	Similarly, 
	$R$ is said to be $1/\beta$-Lipschitzian with respect to $S$ if 
	\begin{equation} \label{def:lips}
		(\forall x \in \mathcal{E}) (\forall y \in \mathcal{E})\quad \beta\|Rx-Ry\|_{S^{-1}} \leq  \|x 
		- y 
		\|_{S}.
	\end{equation}	
	Obviously, \eqref{def:coco} implies
	\eqref{def:lips}. We say that $R$ is nonexpansive with respect to $S$ if it is $1$-Lipschitzian with respect to $S$.
	The set of fixed points of $R$ is $\fix R = \menge{x \in \mathcal{E}}{x=Rx}$.
	Let $2^\H$ be the power set of $\H$ and let $A\colon\H \rightarrow 2^{\H}$ be a set-valued operator.
	The domain, range, set of zeros, and graph of $A$ 
	are 
	$\dom\, A = \menge{x \in \H}{Ax \neq  \varnothing}$,
	$\ran\, A = \menge{u \in 
		\H}{(\exists x \in \H)\,\, u \in Ax}$,  $\zer A = 
	\menge{x \in \H}{0 \in Ax}$,
	and $\gra A = \menge{(x,u) \in \H \times \H}{u \in Ax}$, respectively. Moreover, the inverse of $A$ is the operator
    $A^{-1}\colon\H\to 2^\H
\colon u\mapsto\menge{x\in\H}{u\in Ax}$ and its resolvent is $J_A = (\id+A)^{-1}$.
	Let $M\in \mathcal{B}(\H,\H)$ and $N\in \mathcal{B}(\H,\H)$ be self-adjoint (not necessarily strongly monotone) operators. The following extension of monotonicity was considered in \cite{EvensLatafatPatrinos2025CP}. The operator $A$ is $(M,N)$-semimonotone on $\mathcal{S}\subset \H^2$ if
	\begin{equation}\label{eq:defrhomon}
		(\forall (y,v)\in\mathcal{S} )( \forall (x,u) \in \gra A) \quad \scal{x-y}{u-v} 
		\geq \scal{x-y}{M(x-y)}+\scal{u-v}{N(u-v)}.
	\end{equation}
    This is also equivalent to the fact that $A^{-1}$ is $(N,M)$-semimonotone on 
    $\check{\mathcal{S}}= \{(u,x)\in\HH^2\mid (x,u)\in \mathcal{S}\}$.
	If $\mathcal{S} = \gra A$, the operator $A$ is called $(M,N)$-semimonotone. When $N=0$ (resp. $M=0$), we refer to 
	$M$-monotonicity (resp. $N$-comonotonicity). When $M=0$ and
    $N= \beta S^{-1}$, this reduces to $\beta$-cocoercivity with respect to $S$.
	Additionally, $A$ is maximally $(M,N)$-semimonotone if it is $(M,N)$-semimonotone and its graph is 
	maximal with respect to inclusion among the graphs of $(M,N)$-semimonotone operators. 
	Given $(\mu,\rho) \in \R^2$, if $M=\mu \id$ and $N=\rho \id$, for convenience, 
	we simply say that $A$  is $(\mu,\rho)$-semimonotone on $\mathcal{S}$.  The following table summarizes various notions of monotonicity which are particular cases of $(M,N)$-semimonotonicity.
	\begin{table}[h!]
		\centering
		\setlength{\tabcolsep}{2.5pt}
		\begin{tabular}{l|l|l|c}
			$M$ & $N$ & Property & Reference\\ \hline 
			$\mu \id$ & $\rho \id$ & $(\mu,\rho)$-semimonotone & \cite{EvensPasLatafatPatrinos2025DR}\\
			$\mu M^2$ & $0$ & $(\mu M)$-hypomonotone & \cite{PapadimitriouVu2023,TrangNguyen2026}\\
			$\mu \id$ & $0$ & $\mu$-monotone & \cite{BauschkeMoursiXianfu2021} \\
			$0$ & $\rho \id$& $\rho$-comonotone & \cite{BauschkeMoursiXianfu2021} \\
			$-|\mu|\id$ & $0$ & $|\mu|$-hypomonotone & \cite{CombettesPennanen2004} \\
			$0$ & $-|\rho|\id$ & $|\rho|$-cohypomonotone & \cite{CombettesPennanen2004} \\
			$|\mu|\id$ & $0$ & $|\mu|$-strongly monotone &  \cite{bauschkebook2017} \\
			$0$ & $|\rho|\id$ & $|\rho|$-cocoercive &  \cite{bauschkebook2017} \\
			$0$ & $0$ & monotone & \cite{bauschkebook2017} \\
		\end{tabular}
		\caption{Special cases of $(M,N)$-semimonotonicity. Here, $(\mu,\rho)\in \R^2$.}\label{Tab:NofM}
	\end{table} 
	Note that \cite{PapadimitriouVu2023} also introduces the notion of $(\rho M,W)$-hypomonotonicity where $W \colon \H \to \H$ is a single-valued operator (not necessarily linear). Let $A$ be a maximally $(M,N)$-semimonotone operator.
    %is defined by $J_A:=(\id+A)^{-1}$. 
    If $M=\mu\id$, $\mu > -1$, $N=\rho\id$, and $\rho\geq 0$, $J_A$ is everywhere defined and single-valued on $\H$. 
    Since $J_A = \id - J_{A^{-1}}$, if $M=\mu\id$, $\mu \geq  0$, $N=\rho\id$, and $\rho > -1$, $J_A$ is also everywhere defined and single-valued on $\H$   \cite{BauschkeMoursiXianfu2021}.
    %Note that, if $A$ is $\rho$-monotone, then, for every $\gamma \in \RPP$, $\gamma A$ is $\gamma\rho$-monotone.
	We say that $\gra A$ is sequentially weak-strong closed if, for every sequence $(x_n,u_n)_{n \in \N}$ in $\gra A$ such that $x_n \weakly x$ and $u_n \to u$, we have $u \in Ax$. An operator $T\colon\mathcal{E}\subset \H\to\H$ is demiclosed at $y\in\H$ if, for every $x\in\mathcal{E}$
    and every
sequence $(x_n)_{n\in\N}$ in $\mathcal{E}$ such that
$x_n\weakly x$ and $Tx_n\to y$, we have $Tx=y$.
The operator $T$ is demiclosed if it is demiclosed at every point of $\H$.
 	
	We denote by $\Gamma_0(\H)$ the class of proper, lower 
	semicontinuous, convex functions $f\colon\H\to\RX$. Let 
	$f\in\Gamma_0(\H)$.
	The Fenchel conjugate of $f$ is 
	defined by $f^*\colon u\mapsto \sup_{x\in\H}(\scal{x}{u}-f(x))$ and we have
	$f^*\in \Gamma_0(\H)$. The Moreau subdifferential of $f$ is the maximally monotone operator
	$\partial f\colon x\mapsto \menge{u\in\H}{(\forall y\in\H)\:\: 
		f(x)+\scal{y-x}{u}\le f(y)}$. In addition,
	$(\partial f)^{-1}=\partial f^*$ 
	and the set of 
	minimizers of $f$ is $\zer\,\partial f = \argmin_{x\in \H}f$. 
	
	Examples of semimonotone operators can be found in \cite[Section~6]{BauschkeMoursiXianfu2021}, where properties and examples of hypoconvex functions, as well as the semimonotonicity properties of their \emph{generalized subdifferentials} (Clarke--Rockafellar, Mordukhovich, and Fréchet), are presented. In addition, several examples of semimonotone problems can be found in \cite{EvensLatafatPatrinos2025CP,EvensPasLatafatPatrinos2025DR,PapadimitriouVu2023}.

	For further properties of monotone operators,
	nonexpansive mappings, and convex analysis, the 
	reader is referred to \cite{bauschkebook2017}. 
	
	We conclude this section by presenting two lemmas, which will be useful for dealing with sums of semimonotone operators.
	\begin{lem}\label{lem:des}
		Let $M\in\mathcal{B}(\H,\H)$ be a self-adjoint monotone operator and let $(\rho,\beta) \in \R^2$ be such that $\rho+\beta>0$. The following inequality holds:
		\begin{equation}
			(\forall (x,y) \in \H^2) \quad \quad \rho \|x\|^2_M+\beta\|y\|^2_M \geq \frac{\beta\rho}{\rho+\beta}\|x+y\|^2_M.
		\end{equation}
	\end{lem}
	\begin{proof}
		%It follows from \cite[Corollary~2.15]{bauschkebook2017} that,
		We have
		\begin{align*}
			\left\|\frac{\rho x}{\rho+\beta}+\frac{\beta (-y)}{\rho+\beta}\right\|^2_M + \frac{\rho\beta}{(\rho+\beta)^2}\|x+y\|^2_M = \frac{\rho}{\rho+\beta}\|x\|^2_M+\frac{\beta}{\rho+\beta}\|y\|^2_M.
		\end{align*}
		Multiplying this identity by $\rho+\beta>0$ yields the result.
	\end{proof}
    \begin{lem}
    \label{le:sumsemimon3} 
    Let $S\in\mathcal{B}(\H,\H)$ be self-adjoint and strongly monotone, let $A\colon\H\to2^{\H}$, and let $C\colon\H\to 2^{\H}$. Let $(\nu_1,\nu_2,\rho_1,\rho_2) \in \R^4$ be such that $\rho_1+\rho_2>0$.
    Suppose that $A$ is $(\nu_1 S,\rho_1 S^{-1})$-semimonotone and $C$ is $(\nu_2 S,\rho_2 S^{-1})$-semimonotone. 
    Then, $A+C$ is $\big((\nu_1+\nu_2)S,\frac{\rho_1\rho_2}{\rho_1+\rho_2}S^{-1}\big)$-semimonotone.
    \end{lem}
        \begin{proof}
    For every $\big((x,u_1),(y,v_1)\big) \in (\gra A)^2$ and every $\big((x,u_2),(y,v_2)\big) \in (\gra C)^2$, we have
              \begin{align*}
                \scal{u_1+u_2-(v_1+v_2)}{x-y}&=\scal{u_1-v_1}{x-y}+\scal{u_2-v_2}{x-y}\\
                &\geq \nu_1\|x-y\|_{S}^2 +\rho_1\|u_1-v_1\|_{S^{-1}}^2+\nu_2\|x-y\|_{S}^2 +\rho_2\|u_2-v_2\|_{S^{-1}}^2\\
                &\geq (\nu_1+\nu_2)\|x-y\|_{S}^2 +\frac{\rho_1\rho_2}{\rho_1+\rho_2}\|(u_1+u_2)-(v_1+v_2)\|_{S^{-1}}^2,
            \end{align*}
            where the last inequality stems from Lemma~\ref{lem:des}.
    \end{proof}
    
    \begin{lem}
    \label{le:graclosedsemimon} 
    Let $S\in\mathcal{B}(\H,\H)$ be self-adjoint and strongly monotone, let $A\colon\H\to\H$, and let $C\colon\H\to \H$. Let $(\rho_1,\rho_2) \in \RPP\times \R$ be such that $\rho_1+\rho_2>0$.
    Suppose that $A$ is $\rho_1$-cocoercive with respect to S, and $C$ is $\rho_2 S^{-1}$-comonotone and weakly continuous.
    Then, $\gra(A+C)$ is sequentially weak-strong closed.
    \end{lem}
    
    \begin{proof}
    Let $(x,u)\in \H^2$ and let $(x_n)_{n\in\N}$ be a sequence of $\H$ such that $x_n\weakly x$ and $(A+C)x_n\to u$. For every $(n,m)\in\N^2$,
              \begin{align}
                &\scal{(A+C)x_n-(A+C)x_m}{x_n-x_m}\nonumber\\
                &=\scal{Ax_n-Ax_m}{x_n-x_m}+\scal{Cx_n-C x_m}{x_n-x_m}\nonumber\\
                &\geq \rho_1\|Ax_n-Ax_m\|_{S^{-1}}^2+ \rho_2\|C x_n-C x_m\|_{S^{-1}}^2\nonumber\\
               &=\rho_1\big(\|(A+C)x_n-(A+C)x_m\|_{S^{-1}}^2-
               2\scal{S^{-1}\big((A+C)x_n-(A+C)x_m\big)}{Cx_n-Cx_m}\big)\nonumber\\
               &\quad+(\rho_1+\rho_2)\|C x_n- C x_m\|_{S^{-1}}^2.
               \label{e:weakstronglemspec}
            \end{align}
            The weak convergence of
        $(x_n)_{n\in\N}$ implies that this sequence is bounded, whereas
the strong convergence of $((A+C)x_n)_{n\in\N}$ implies that 
            $(A+C)x_n-(A+C)x_m\to 0$, as $n\to+\infty$ and $m\to +\infty$.
            Consequently,
            \begin{equation*}
            \scal{(A+C)x_n-(A+C)x_m}{x_n-x_m}\to 0\quad  \text{and}\quad \|(A+C)x_n-(A+C)x_m\|_{S^{-1}}\to 0.
            \end{equation*}
            In addition,
            since $C$ is weakly continuous, 
            $Cx_n\weakly Cx$ and
            \begin{multline*}
   |\scal{S^{-1}\big((A+C)x_n-(A+C)x_m\big)}{Cx_n-Cx_m}| \\\le \|S^{-1}\|\|(A+C)x_n-(A+C)x_m\|\|Cx_n-Cx_m\|\to 0.
  \end{multline*}
            Using the positivity of $\rho_1+\rho_2$, it follows from \eqref{e:weakstronglemspec}
            that $\|C x_n- C x_m\|_{S^{-1}}\to 0$.
            Because the norm induced by $S^{-1}$ is equivalent to the original norm, $(Cx_n)_{n\in\N}$
            is a Cauchy sequence. Therefore, it converges strongly and since $Cx_n\weakly Cx$, its strong limit is 
            $Cx$.
            It follows that $Ax_n\to u-Cx$. By using now the cocoercivity of $A$,
            \begin{equation}
                \rho_1 \|A x_n-Ax\|^2_{S^{-1}} \leq \scal{A x_n-Ax}{x_n-x}.
            \end{equation}
            Since $Ax_n-Ax\to u-Cx-Ax$ and
            $x_n\weakly x$, $\scal{A x_n-Ax}{x_n-x}\to 0$,
            which implies that $Ax_n\to Ax$. In conclusion
            $u= (A+C)x$, which proves that
            $\gra(A+C)$ is sequentially weak-strong closed.
    \end{proof}
	
	\section{Nonlinear Forward-Backward for comonotone operators}\label{sec:MNFB}
	We first focus on the case when  $B=0$ and $D=0$ in Problem~\ref{pro:main}. In this specific case, the dual variable in Problem~\ref{pro:main} vanishes. Hence, the problem reduces to solving a primal inclusion and it can be addressed efficiently, for example, by the algorithms in \cite{BuiCombettesWarped2020,Giselsson2021NFBS,Lions1979SIAM,MorinBanertGiselsson2022,passty1979JMAA}. We investigate this primal problem under the assumption that the sum of the operators involved is comonotone. 
	For notational convenience, and for reasons that will become clear later, we use bold symbols for the variables in this section.
	
	Let us list our standing assumptions on the involved operators.
	\begin{asume}\label{assum:2}
		Let $\bm{\H}$ be the underlying real Hilbert space. Let 
		$\bm{S}\in\mathcal{B}(\bm{\H},\bm{\H})$ be a self-adjoint and strongly monotone operator, let $\bm{A}\colon\bm{\H}\to2^{\bm{\H}}$ be a set-valued operator, let $\bm{C}\colon\bm{\H} \to \bm{\H}$ be a $\beta$-cocoercive operator with respect to $\bm{S}$, for some $\beta\in \RPP$, and let $\bm{M}\colon \bm{\H}\to \bm{\H}$ be a single-valued operator. We suppose that
		\begin{enumerate}
			\item\label{assum:21} $(\bm{M}+\bm{A})^{-1}$ is single-valued and has full domain;
			\item\label{assum:22} $\gra(\bm{A}+\bm{C})$ is sequentially weak-strong  closed;
			\item\label{assum:23} 
			$\bm{A}$ is $(\rho \bm{S}^{-1})$-comonotone on $\pmb{\mathbb{S}}=\menge{(\x,-\bm{C}\x)\in\bm{\H}^2}{\x \in \zer (\bm{A}+\bm{C})}$  for some $\rho \in$\linebreak $]-\beta,+\infty[$.      
			\item \label{assum:24}  
            There exist $\zeta \in [0,1/2[$
            %[0,1/(2(1-4\hat{\rho}))[$ 
            and $\tau\in\, ]\underline{\tau}_\zeta,+\infty[$ such that $\tau \bm{M}- \bm{S}$ is $\zeta$-Lipschitzian with respect to $\bm{S}$
            where 
            \[
            \hat{\rho}=\min\{\rho,0\}\quad \text{and}
            \quad \underline{\tau}_\zeta = -\frac{\zeta\hat{\rho}}{\beta+\hat{\rho}}.
            \]
		\end{enumerate}
	\end{asume}
	Let us make some comments on these assumptions.
	\begin{rem}\ 
    \label{re:firstrem}
		\begin{enumerate}
            \item\label{re:firstremi} Since $\bm{S}$ is nonexpansive with respect to $\bm{S}$, it follows from Assumption~\ref{assum:2}\ref{assum:24} that $\bm{M} =(\tau \bm{M}-\bm{S}+\bm{S})/\tau$ is $\zeta_{\bm{M}}$-Lipschitzian with respect to $\bm{S}$, its Lipschitz modulus being $\zeta_{\bm{M}} = \xi/\tau$ and $\xi \in [1-\zeta,1+\zeta]$.
			\item Under Assumption~\ref{assum:2}\ref{assum:24}, if \(\bm A\) is maximally monotone, then \((\bm M+\bm A)^{-1}\) is single-valued and has full domain (see \cite[Proposition~3.1]{MorinBanertGiselsson2022}). Weaker conditions guaranteeing this assumption can be found in \cite[Proposition~3.8 \& Proposition~3.9]{BuiCombettesWarped2020}.
			\item Let $\bm{A}=\widetilde{\bm{A}}+\widetilde{\bm{B}}$ 
			with $\widetilde{\bm{A}}\colon \bm{\H} \to 2^{\bm{\H}}$ maximally $(-\rho)$-cohypomonotone for $\rho< 0$, $\widetilde{\bm{B}}\colon \bm{\H} \to \bm{\H}$ single-valued, and
			$\bm{M} = \Id/\tau - \widetilde{\bm{B}}$.
            Then, the operator $(\bm{M}+\bm{A})^{-1}=(\Id/\tau+\widetilde{\bm{A}})^{-1}$ is single-valued with full domain when $\rho/\tau> -1$, i.e., $\tau >-\rho$. In this situation, by considering $\bm{S}=\Id$, we have $\tau\bm{M}-\bm{S}=-\tau\widetilde{\bm{B}}$. Hence,
            if $\widetilde{\bm B}$ is
$\zeta_{\widetilde{\bm B}}$-Lipschitzian, where
$\zeta_{\widetilde{\bm B}}\in\left]0,-\rho^{-1}\min\left\{\beta+\rho,\frac12\right\}
\right[$,
then Assumptions~\ref{assum:2}\ref{assum:21} and
\ref{assum:2}\ref{assum:24} hold by choosing
$\tau\in\left]-\rho,
(2\zeta_{\widetilde{\bm B}})^{-1}\right[$.
			\item If $\bm{C}=\bm{0}$, Assumption~\ref{assum:2}\ref{assum:23} reduces to %the $\rho\bm{S}^{-1}$-comonotonicity of $\bm{A}$ in $\zer{\bm{A}}$, or, equivalently,  
			%the fact that $\bm{A}$ satisfies 
            the so-called weak Minty condition for $\bm{A}$ at every point of  $\zer{\bm{A}}$ (see for instance \cite{AlacaogluKimWright2024,ChoroburaNecoaraPesquet2025,DiakonikolasDaskalakisJordan2020,EvensLatafatPatrinos2025CP,EvensPasLatafatPatrinos2025DR}).
		\end{enumerate}	
	\end{rem}	
	The following proposition provides conditions on $\bm{A}$ and $\bm{C}$ for which Assumption~\ref{assum:2}\ref{assum:22}  holds.

	\begin{prop}\label{prop:A+Ccomo}
		Let $\bm{S}$ be a self-adjoint and strongly monotone linear operator, let $\bm{A}\colon \bm{\H} \to 2^{\bm{\H}}$ be a $\rho \bm{S}^{-1}$-comonotone operator for 
		some 
		$\rho \in \R$ and let $\bm{C}\colon\bm{\H} \to \bm{\H}$ be a $\beta$-cocoercive  
		operator  with respect to $\bm{S}$ for some 
		$\beta 
		\in \RPP$. The following statements hold:
		\begin{enumerate}
			\item\label{prop:A+Ccomo1} If $\rho\|\bm{S}^{-1}\|> - 1$, $\bm{T}:=\Id-J_{\bm{A}}$ is single-valued on $\ran(\Id+\bm{A})$. Moreover, if $\ran(\Id+\bm{A})$ is weakly sequentially closed, $\bm{T}$ is demiclosed.
			\item\label{prop:A+Ccomo2} If $\beta + {\rho} > 0$, then $(\bm{A}+\bm{C})$ is
			$\left(\frac{\beta{\rho}}{\beta+{\rho}}\bm{S}^{-1}\right)$-comonotone.
			\item\label{prop:A+Ccomo3} Suppose that $\beta +{\rho} > -\beta{\rho}\|\bm{S}^{-1}\|$ and $\ran(\Id+\bm{A}+\bm{C})$ is weakly sequentially closed. Then, $\gra(\bm{A}+\bm{C})$ is sequentially weak-strong  closed.
		\end{enumerate}
	\end{prop}
	\begin{proof}
		\begin{enumerate}
			\item If $\rho \geq 0$, the result follows from \cite[Proposition~23.8 \& Theorem~4.27]{bauschkebook2017}. 
            If $\rho <0$, since $(\forall\bm{u}\in\bm{\H})$ $\rho\|\bm{u}\|_{\bm{S}^{-1}}^2 \geq \rho \|\bm{S}^{-1}\|\|\bm{u}\|^2$, $\bm{A}$ is $\rho\|\bm{S}^{-1}\|$-comonotone. Since $\rho \|\bm{S}^{-1}\| >-1$, $J_{\bm{A}}$ is single-valued on $\ran(\Id+\bm{A})$ and $\alpha$-conically nonexpansive for $\alpha = \frac{1}{2(\rho\|\bm{S}^{-1}\|+1)}$ \cite[Proposition~3.6]{BauschkeMoursiXianfu2021}. Let $\bm{N} \colon \ran(\Id+\bm{A}) \to \bm{\H}$ be a nonexpansive operator such that $J_{\bm{A}} = (1-\alpha)\Id + \alpha \bm{N}$, i.e., $\bm{T} = \alpha(\Id-\bm{N})$.      
            Let $(\bm{x}_n)_{n \in \N}$ be a sequence in $\ran(\Id+\bm{A})$ such that $\bm{x}_n \weakly \bm{x}$ and $\bm{T}\bm{x}_n \to \bm{u}$.
            Since $\ran(\Id+\bm A)$ is weakly sequentially closed and $\bm{x}_n\weakly\bm{x}$, we have $\bm{x}\in\ran(\Id+\bm{A})$.
			%hence $\alpha(\Id-\bm{N})\bm{x}_n \to \bm{u}$.
			Thus, it follows from the nonexpansiveness of $\bm{N}$ that
			\begin{align*}
				&\|\bm{T}\bm{x}-\bm{u}\|^2\\
				&= \|\alpha(\Id-\bm{N})\bm{x}-\bm{u}\|^2\\
				&= \|\alpha(\bm{x}_n-\bm{N}\bm{x})-\bm{u}\|^2-2\alpha\scal{\bm{x}_n-\bm{x}}{\alpha(\bm{x}-\bm{N}\bm{x})-\bm{u}}-\alpha^2\|\bm{x}_n-\bm{x}\|^2\\
				&= \|\bm{T}\bm{x}_n-\bm{u}\|^2+\alpha^2\|\bm{N}\bm{x}_n-\bm{N}\bm{x}\|^2+2\alpha\scal{\bm{N}\bm{x}_n-\bm{N}\bm{x}}{\bm{T}\bm{x}_n-\bm{u}}\\
				& \hspace*{2cm}
				-2\alpha\scal{\bm{x}_n-\bm{x}}{\bm{T}\bm{x}-\bm{u}}-\alpha^2\|\bm{x}_n-\bm{x}\|^2\\
				&\le \|\bm{T}\bm{x}_n-\bm{u}\|^2+2\alpha\scal{\bm{N}\bm{x}_n-\bm{N}\bm{x}}{\bm{T}\bm{x}_n-\bm{u}}
				-2\alpha\scal{\bm{x}_n-\bm{x}}{\bm{T}\bm{x}-\bm{u}}\\    
				&\le \|\bm{T}\bm{x}_n-\bm{u}\|^2+2\alpha\|\bm{N}\bm{x}_n-\bm{N}\bm{x}\| \|\bm{T}\bm{x}_n-\bm{u}\|
				-2\alpha\scal{\bm{x}_n-\bm{x}}{\bm{T}\bm{x}-\bm{u}}\\
				&\le \|\bm{T}\bm{x}_n-\bm{u}\|^2+2\alpha\|\bm{x}_n-\bm{x}\| \|\bm{T}\bm{x}_n-\bm{u}\|				-2\alpha\scal{\bm{x}_n-\bm{x}}{\bm{T}\bm{x}-\bm{u}}.
			\end{align*}
			By taking $n \to +\infty$ and by using the fact
			that $(\x_n)_{n\in \mathbb{N}}$ is bounded,       
			we conclude that $\bm{T}\x=\bm{u}$.
			\item Suppose that $\beta +{\rho}> 0$.	 
            Since $\bm{A}$ is $(0,\rho \bm{S}^{-1})$-semimonotone and 
            $\bm{C}$ is $(0,\beta  \bm{S}^{-1})$-semimonotone, the result
            follows directly from 
            Lemma~\ref{le:sumsemimon3}.
			\item We first note that $\beta+\rho>0$. Indeed, this is immediate
when $\rho\geq0$, whereas, when $\rho<0$, it follows from
$\beta + {\rho} > -\beta{\rho}\|\bm{S}^{-1}\|$.
            Let $\overline{\rho}=\beta\rho/(\beta+\rho)$.
            By \ref{prop:A+Ccomo2}, 
            $\bm{A}+\bm{C}$ is 
            $\overline{\rho}\bm S^{-1}$-comonotone.
           The assumption
           $\beta + {\rho} > -\beta{\rho}\|\bm{S}^{-1}\|$ being equivalent to
           $\overline{\rho}\|\bm{S}^{-1}\|> -1$, it follows from \ref{prop:A+Ccomo1} that $\Id-J_{\bm{A}+\bm{C}}$ is demiclosed. Let $(\bm{x}_n,\bm{u}_n)_{n\in \mathbb{N}}$
			be a sequence in $\gra(\bm{A}+\bm{C})$ such that $\bm{x}_n\weak \bm{x}$
			and $\bm{u}_n\to \bm{u}$.
			Note that
			\begin{equation}
				\bm{u}_n \in (\bm{A}+\bm{C})\bm{x}_n \Leftrightarrow (\Id-J_{\bm{A}+\bm{C}})(\bm{x}_n+\bm{u}_n) = \bm{u}_n.
			\end{equation}
			Since $\bm{x}_n+\bm{u}_n \weak \bm{x}+\bm{u}$, we conclude that $(\Id-J_{\bm{A}+\bm{C}})(\bm{x}+\bm{u})=\bm{u}$ and $\bm{u} \in (\bm{A}+\bm{C})\bm{x}$.
		\end{enumerate}
	\end{proof}
	\begin{rem}\label{rem:sumsemimon} \ 
		\begin{enumerate}
			\item     Suppose that, in the standard metric $\bm{S}=\Id$, $\bm{A}$ is $\rho$-comonotone for $\rho>-\beta/(1+\beta)$.
			It follows from \cite[Theorem~2.17]{BauschkeMoursiXianfu2021} that $\bm{A}$ is maximally $\rho$-comonotone if and only if $\ran(\Id+\bm{A}) = \bm{\H}$. Therefore, if $\bm{A}+\bm{C}$ is maximally $\overline{\rho}$-comonotone for $\overline{\rho}\in ]-1,\frac{\beta \rho}{\beta+\rho}]$,     $\ran(\Id+\bm{A}+\bm{C})$ is weakly sequentially closed and $\gra(\bm{A}+\bm{C})$ is sequentially weak-strong closed.
			\item Additional conditions ensuring the single-valuedness of $(\bm{M}+\bm{A})^{-1}$ and the sequential weak-strong closedness of $\gra \bm{A}$ for comonotone operators can be found in \cite[Lemma~1 \& Proposition~3]{PapadimitriouVu2023}.
		\end{enumerate}
	\end{rem}
	
	We now focus our attention on the following problem.
	\begin{pro}\label{pro:1}
		Suppose that Assumption~\ref{assum:2} holds.
		%		Let $\bm{\H}$ be a real Hilbert space, let $\bm{A}:\bm{\H}\to2^{\bm{\H}}$ be a set-valued operator, and let $\bm{C}:\bm{\H} \to \bm{\H}$ be a $\beta$-cocoercive operator, for some $\beta\in \RPP$. The problem is to
		We want to
		\begin{equation}
			\textnormal{find  } \x \in \bm{\H} \textnormal{ such that } \bm{0} \in (\bm{A}+\bm{C})\x,
		\end{equation}
		under the assumption that its solution set, denoted by $\widetilde{\bm{Z}}$, is nonempty.
	\end{pro}
	The next proposition provides a sufficient condition for the solution to Problem~\ref{pro:1} to be unique.
	\begin{prop}\label{prop:uniquesol}
		In the context of Problem~\ref{pro:1}, suppose additionally that $\bm{A}$ is $(\mu\bm{S},\rho\bm{S}^{-1})$-semimonotone on $\pmb{\mathbb{S}}$ for some $\mu \in \RPP$. Then, the solution set of Problem~\ref{pro:1} is a singleton.
	\end{prop}
	\begin{proof}
		Let $(\x,\y) \in \widetilde{\bm{Z}}^2$. We have that $-\bm{C}\x \in \bm{A}\x$ and $-\bm{C}\y \in \bm{A}\y$. Then, by the $(\mu\bm{S},\rho\bm{S}^{-1})$-semimonotonicity of $\bm{A}$ on $\pmb{\mathbb{S}}$, we have 
		\begin{align}\label{eq:uniquesol1}
			-\scal{\x-\y}{\bm{C}\x-\bm{C}\y}\geq \mu\|\x-\y\|^2_{\bm{S}}+\rho\|\bm{C}\x-\bm{C}\y\|^2_{\bm{S}^{-1}}.
		\end{align}
		Moreover, since $\bm{C}$ is $\beta$-cocoercive with respect to $\bm{S}$,  we have that
		\begin{align}\label{eq:uniquesol2}
			\scal{\x-\y}{\bm{C}\x-\bm{C}\y}\geq \beta\|\bm{C}\x-\bm{C}\y\|^2_{\bm{S}^{-1}}.
		\end{align}
		Summing up \eqref{eq:uniquesol1} and \eqref{eq:uniquesol2}, we obtain
		\begin{align}
			0\geq \mu\|\x-\y\|^2_{\bm{S}}+(\rho+\beta)\|\bm{C}\x-\bm{C}\y\|^2_{\bm{S}^{-1}}.
		\end{align}
		Since $\rho+\beta>0$,  we deduce that $\x=\y$.
	\end{proof}
	Now, we introduce the algorithm studied in this section for solving Problem~\ref{pro:1}. This algorithm was initially proposed in \cite{MorinBanertGiselsson2022} for solving monotone problems. We investigate its convergence in the absence of monotonicity by using Assumption~\ref{assum:2}\ref{assum:21} to define 
	a warped resolvent of $\bm{A}$.
	\begin{algo}\label{algo:NLFBMM} In the context of Problem~\ref{pro:1} and Assumption~\ref{assum:2}, let $(\x_0,\bu_0)\in \bm{\H}^2$, let $\theta \in\,]0,2[$, and consider the following recurrence:
\begin{subnumcases}{\label{eq:NLFBMM}(\forall n\in\N)\quad}
\begin{aligned}[t]
\p_{n+1}
&=(\bm{M}+\bm{A})^{-1}(\bm{M}\x_n-\bm{C}\x_n+\bu_n/\tau)
\end{aligned}
& \label{eq:NLFBMM-a}\\
\begin{aligned}[t]
\bu_{n+1}
&=(\tau\bm{M}-\bm{S})\p_{n+1}
  -(\tau\bm{M}-\bm{S})\x_n
\end{aligned}
& \label{eq:NLFBMM-b}\\
\begin{aligned}[t]
\x_{n+1}
&=(1-\theta)\x_n+\theta\p_{n+1}
\end{aligned}
& \label{eq:NLFBMM-c}
\end{subnumcases}
%        \begin{equation}
%        \label{eq:NLFBMM}
%			(\forall n\in\N)\quad 
%			\begin{cases}
%				\begin{array}{l}
%					\textcolor{blue}{\p_{n+1}}= (\bm{M} + \bm{A})^{-1} (\bm{M}\x_n-\bm{C}\x_n+\bu_n/\tau)\\
%					\bu_{n+1} = (\tau \bm{M} - \bm{S})\textcolor{blue}{\p_{n+1}} -(\tau \bm{M} -\bm{S})\x_n.\\
%                    \textcolor{blue}{\x_{n+1} = (1-\theta)\x_n + \theta \p_{n+1}}
%				\end{array}
%			\end{cases}
%            \end{equations}
	\end{algo}	
	The next proposition provides a key estimate for deriving the convergence of Algorithm~\ref{algo:NLFBMM}.
	\begin{prop}
		\label{prop:NLFBMMnew} In the context of Problem~\ref{pro:1}, let $\zeta_{\bm{M}}\in \RPP$ be the Lipschitz modulus of $\bm{M}$ with respect to $\bm{S}$, 
		let $(\x_0,\bu_0)\in \bm{\H}^2$, let $(\p_{n},\bu_{n},\x_{n})_{n \geq 1}$ be the sequence generated by Algorithm~\ref{algo:NLFBMM}. Let $\x \in \zer(\bm{A}+\bm{C})$, and define
		\begin{align}
			%\bm{\xi}_{n} &= \bm{S}^{-1}\big(\tau\bm{M}\textcolor{blue}{\p_{n}}-(\tau\bm{M}-\bm{S})\x_{n-1}\big),\label{e:defbxi}\\
            	&\eta = \begin{cases}
		1, &\textnormal{ if } \theta \in [1,2[ \textnormal{ and } (\bm{S}-\tau \bm{M}) \textnormal{ is monotone,}\\
		1+|1-\theta|, &\textnormal{ otherwise,}
		\end{cases}\label{def:nu}\\
            \lambda&=2-\theta-2\eta\zeta-\frac{(\tau+2\hat{\rho})^2}{2\tau\big(\beta+\hat{\rho}(1+\tau^{-1}\zeta)\big)} +2\hat{\rho}\tau
            \left( \left(\tau^{-1}\zeta+\zeta_{\bm{M}}\right)^2+4\tau^ {-1}\zeta\right),     %(\tau\zeta_{\bm{M}}\left({\color{blue} \zeta_{\bm{M}}+\frac{\zeta}{\tau}}\right)+2\zeta)+\lambda_1,
            \label{e:defmutn}  \\  
            \lambda_1 & =\zeta\big({2\eta-1}-2\hat{\rho}(\tau^{-1}\zeta+\zeta_{\bm{M}}+2)\big), \label{e:deflambda_1}  \\
			(\forall n\in \N)\quad a_{n+1} &= \|\x_{n+1}-\x\|^2_{\bm{S}}+2\theta\scal{\bu_{n+1}}{\p_{n+1}-\x} +\theta\lambda_1\|\p_{n+1}-\x_{n}\|^2_{\bm{S}}.\label{e:defan}
		\end{align}
		Then, the following assertions hold:
        \begin{enumerate}
\item\label{prop:NLFBMMnewi} For every $n \in \N\setminus\{0\}$,
\begin{align}\label{eq:quasifej}
			 a_{n+1}&\leq a_n-\theta\lambda\|\p_{n+1}-\x_n\|^2_{\bm{S}}.
		\end{align}
        \item\label{prop:NLFBMMnewii} $(a_n)_{n\geq 1}$ is a nonnegative sequence such that 
        \begin{equation}\label{eq:anboundbelow}
          (\forall n \in \N\setminus\{0\})  \qquad a_n \geq \left(1-\zeta \theta\right) \|\x_{n}-\x\|^2_{\bm{S}}.
        \end{equation}
        \end{enumerate}
	\end{prop}
    \begin{proof}
    First note that
\begin{equation}
\tau>\underline\tau_\zeta=
-\frac{\zeta\hat{\rho}}{\beta+\hat{\rho}}
\quad
\Leftrightarrow
\quad 
\beta\tau+\hat{\rho}(\tau+\zeta)>0,
\end{equation}
which ensures that $\lambda$ is well defined.
In addition, if $\zeta=0$, then, for every $n\in\N$,
    $\bu_{n+1}
=(\tau\bm{M}-\bm{S})\p_{n+1}
  -(\tau\bm{M}-\bm{S})\x_n=\bm{0}$
  and $\lambda_1 = 0$. The decay property \eqref{eq:quasifej} then follows directly from the arguments below after removing all terms involving $\bu_n$. 
  Moreover, for every $n\in\N\setminus\{0\}$,
  $a_n=\|\x_{n}-\x\|_{\bm{S}}^2$, so \eqref{eq:anboundbelow} holds immediately.
  In the remainder of the proof, we therefore assume that $\zeta>0$.    \begin{enumerate}
        \item Fix $n\in \N\setminus\{0\}$.
		It follows from \eqref{eq:NLFBMM-a} that 
	\begin{equation}
		\tau \bm{M}\x_n-\tau \bm{C}\x_n+\bu_n-\tau \bm{M} \p_{n+1}\in \tau \bm{A} \p_{n+1},
    \end{equation}    
    which, by using \eqref{eq:NLFBMM-b}, is equivalent to
    \begin{equation}
		\bm{S}\x_n + \bu_n  - (\bm{S}\p_{n+1}+\bu_{n+1})-\tau \bm{C}\x_n \in \tau \bm{A} \p_{n+1}.
	\end{equation}
		Since $\x \in \zer(\bm{A}+\bm{C})$, $-\bm{C}\x \in \bm{A}\x$. By the $\rho\bm{S}^{-1}$-comonotonicity of $\bm{A}$ on $\pmb{\mathbb{S}}$ (Assumption~\ref{assum:2}\ref{assum:23}), we deduce that 
		\begin{align}\label{eq:proof0pre}
				0\leq	2&\scal{\bm{S}\x_n + \bu_n  - (\bm{S}\p_{n+1}+\bu_{n+1})-\tau \bm{C}\x_n+\tau \bm{C}\x}{\p_{n+1}-\x}\nonumber\\
	& -\frac{2\rho}{\tau}\|\bm{S}\x_n + \bu_n  - (\bm{S}\p_{n+1}+\bu_{n+1})-\tau \bm{C}\x_n+\tau \bm{C}\x\|^2_{\bm{S}^{-1}},
		\end{align}
that is 
\begin{align}\label{eq:proof0}
				0\leq	2&
\scal{\x_n - \p_{n+1}}{\p_{n+1}-\x}_{\bm{S}}   +2             
                \scal{\bu_n  -\bu_{n+1}-\tau \bm{C}\x_n+\tau \bm{C}\x}{\p_{n+1}-\x}\nonumber\\
	& -\frac{2\rho}{\tau}\|\bm{S}\x_n + \bu_n  - (\bm{S}\p_{n+1}+\bu_{n+1})-\tau \bm{C}\x_n+\tau \bm{C}\x\|^2_{\bm{S}^{-1}}.
		\end{align}
Note that
	\begin{align}\label{eq:proof1}
		2\scal{\x_n - \p_{n+1}}{\p_{n+1}-\x}_{\bm{S}}
		=& \|\x_n-\x\|_{\bm{S}}^2-\|\x_n-\p_{n+1}\|_{\bm{S}}^2-\|\p_{n+1}-\x\|_{\bm{S}}^2.
	\end{align}
		Moreover, it follows from \eqref{eq:NLFBMM-c} that
	\begin{align}\label{eq:proof2}
	&2\scal{\bu_n -\bu_{n+1}}{\p_{n+1}-\x}\nonumber\\
    %&=
	%	-2\scal{\bu_{n+1}}{\p_{n+1}-\x} + 2\scal{\bu_n }{\p_{n}-\x}+2\scal{\bu_n}{\p_{n+1}-\p_{n}}\nonumber\\
		&=
		-2\scal{\bu_{n+1}}{\p_{n+1}-\x} + 2\scal{\bu_n }{\p_{n}-\x}+2\scal{\bu_n}{\p_{n+1}-\x_{n}}+2\scal{\bu_n}{\x_n-\p_{n}}\nonumber\\
		&=
		-2\scal{\bu_{n+1}}{\p_{n+1}-\x} + 2\scal{\bu_n }{\p_{n}-\x}+2\scal{\bu_n}{\p_{n+1}-\x_{n}}+2 (1-\theta)\scal{\bu_n}{\x_{n-1}-\p_{n}}.
	\end{align}
	From the $\zeta$-Lipschitz continuity of $\tau \bm{M}-\bm{S}$ with respect to $\bm{S}$, we deduce that
	\begin{align}\label{eq:proof3}
		\|\bu_{n}\|_{\bm{S}^{-1}} &= \|(\tau \bm{M} - \bm{S})\p_{n} -(\tau \bm{M} -\bm{S})\x_{n-1}\|_{\bm{S}^{-1}}\nonumber\\
		&\leq \zeta\|\p_{n}-\x_{n-1}\|_{\bm{S}}.
	\end{align}
	Hence, the third term on the right-hand side of equality \eqref{eq:proof2} can be upper bounded as
	\begin{align}\label{eq:proof4}
		2\scal{\bu_n}{\p_{n+1}-\x_n}  &\leq \frac{1}{\zeta}\|\bu_n\|^2_{\bm{S}{^{-1}}} +  \zeta\|\p_{n+1}-\x_n\|_{\bm{S}}^2\nonumber\\
		&\leq \zeta\|\p_{n}-\x_{n-1}\|^2_{\bm{S}} +  \zeta\|\p_{n+1}-\x_n\|^2_{\bm{S}}.
	\end{align}
For the last term on the right-hand side of \eqref{eq:proof2}, it follows from
\eqref{def:nu} and \eqref{eq:proof3} that	\begin{align}\label{eq:proof5}
		2 (1-\theta)\scal{\bu_n}{\x_{n-1}-\p_{n}}
		&=2 (\theta-1)\scal{(\tau \bm{M} - \bm{S})\p_{n} -(\tau \bm{M} -\bm{S})\x_{n-1}}{\p_n-\x_{n-1}}\nonumber\\
		&\leq 2(\eta-1)\zeta\|\p_{n}-\x_{n-1}\|^2_{\bm{S}}.	\end{align}
Since $\hat{\rho}\leq \rho$, combining 	\eqref{eq:proof0}-\eqref{eq:proof5}, we obtain
\begin{align}\label{eq:proof6}
		&\|\p_{n+1}-\x\|^2_{\bm{S}}+2\scal{\bu_{n+1}}{\p_{n+1}-\x} \nonumber\\
		&\leq  \|\x_{n}-\x\|^2_{\bm{S}}+2\scal{\bu_{n}}{\p_{n}-\x}+(2\eta-1)\zeta\|\p_{n}-\x_{n-1}\|^2_{\bm{S}} -(1-\zeta)\|\p_{n+1}-\x_n\|^2_{\bm{S}} \nonumber\\
		&-2\tau\scal{\p_{n+1}-\x}{\bm{C}\x_n-\bm{C}\x} -\frac{2\hat{\rho}}{\tau}\|\bm{S}\x_n + \bu_n  - (\bm{S}\p_{n+1}+\bu_{n+1})-\tau \bm{C}\x_n+\tau \bm{C}\x\|^2_{\bm{S}^{-1}}.
\end{align}	
Let us bound the last term in \eqref{eq:proof6}. First note that
\begin{align}\label{eq:proof7}
	\|\bm{S}\x_n + \bu_n  - (&\bm{S}\p_{n+1}+\bu_{n+1})-\tau \bm{C}\x_n+\tau \bm{C}\x\|^2_{\bm{S}^{-1}}\nonumber\\
	&=\|\bm{S}(\x_n-\p_{n+1}) + \bu_n  - \bu_{n+1}\|^2_{\bm{S}^{-1}} -2\tau \scal{\x_n-\p_{n+1}}{\bm{C}\x_n-\bm{C}\x} \nonumber\\
	&\quad-2\tau	\scal{\bu_n  - \bu_{n+1}}{\bm{C}\x_n-\bm{C}\x}_{\bm{S}^{-1}}+\tau^2\| \bm{C}\x_n-\bm{C}\x\|^2_{\bm{S}^{-1}}.
\end{align}
Equation~\eqref{eq:NLFBMM-b}
yields
\begin{align}\label{eq:proof8}
\|\bm{S}(\x_n-\p_{n+1}) &+ \bu_n  - \bu_{n+1}\|^2_{\bm{S}^{-1}} \nonumber\\
&=\|( \tau\bm{M}\p_{n+1}-\tau\bm{M}\x_n)-(( \tau\bm{M}-\bm{S})\p_{n}-(\tau\bm{M}-\bm{S})\x_{n-1})\|^2_{\bm{S}^{-1}}\nonumber\\
&\leq(\tau\zeta_{\bm{M}}\|\p_{n+1}-\x_n\|_{\bm{S}}+ \zeta\|\p_{n}-\x_{n-1}\|_{\bm{S}})^2\nonumber\\
&\leq \tau^2\zeta_{\bm{M}}^2 \left(1+\frac{\zeta}{\zeta_{\bm{M}}\tau}\right)\|\p_{n+1}-\x_n\|^2_{\bm{S}}+ \zeta^2\left(1+\frac{\zeta_{\bm{M}}\tau}{\zeta}\right)\|\p_{n}-\x_{n-1}\|_{\bm{S}}^2 \nonumber\\
&= \tau\zeta_{\bm{M}} \left(\tau\zeta_{\bm{M}}+\zeta\right)\|\p_{n+1}-\x_n\|^2_{\bm{S}}+ \zeta\left(\zeta+\zeta_{\bm{M}}\tau\right)\|\p_{n}-\x_{n-1}\|_{\bm{S}}^2.
\end{align}
In addition, by the Cauchy--Schwarz and Young inequalities, \eqref{eq:proof3} leads to
\begin{align}\label{eq:proof9}
	-2\tau\scal{\bu_n  - \bu_{n+1}}{\bm{C}\x_n-\bm{C}\x}_{\bm{S}^{-1}} 	&\leq 	2\tau	\|\bu_n  - \bu_{n+1}\|_{\bm{S}^{-1}}  \|\bm{C}\x_n-\bm{C}\x\|_{\bm{S}^{-1}} \nonumber\\
	&\leq 	\frac{\tau}{\zeta}	\|\bu_n  - \bu_{n+1}\|_{\bm{S}^{-1}}^2  + \zeta\tau \|\bm{C}\x_n-\bm{C}\x\|_{\bm{S}^{-1}}^2  \nonumber\\
    &\leq 	\frac{2\tau}{\zeta}(	\|\bu_n \|_{\bm{S}^{-1}}^2+\|\bu_{n+1}\|_{\bm{S}^{-1}}^2)  + \zeta\tau \|\bm{C}\x_n-\bm{C}\x\|_{\bm{S}^{-1}}^2  \nonumber\\
	&\leq 2\zeta\tau\left(\|\p_{n+1}-\x_n\|^2_{\bm{S}}+\|\p_{n}-\x_{n-1}\|_{\bm{S}}^2\right) %\nonumber\\
    %& \hspace{3cm}
    +\zeta\tau\|\bm{C}\x_n-\bm{C}\x\|_{\bm{S}^{-1}}^2.
\end{align}
Therefore, \eqref{eq:proof7}-\eqref{eq:proof9} allow us to deduce from \eqref{eq:proof6} that
\begin{align}\label{eq:proof10} 
	\|\p_{n+1}-&\x\|^2_{\bm{S}}+2\scal{\bu_{n+1}}{\p_{n+1}-\x} \nonumber\\
	&\leq  \|\x_{n}-\x\|^2_{\bm{S}}+2\scal{\bu_{n}}{\p_{n}-\x}+\lambda_1\|\p_{n}-\x_{n-1}\|^2_{\bm{S}} \nonumber\\
    %+\left(\zeta+\nu-\frac{\hat{\rho}}{\tau}(\zeta(\zeta+\zeta_{\bm{M}}\tau)+2\tau\zeta)\right)\|\p_{n}-\x_{n-1}\|^2_{\bm{S}} \nonumber\\
	&\hspace{1cm}-\big(1-\zeta+2\hat{\rho}(2\zeta +\zeta_{\bm{M}}(\tau \zeta_{\bm{M}}+ \zeta ))\big)\|\p_{n+1}-\x_n\|^2_{\bm{S}}-2\hat{\rho}(\tau+\zeta) \| \bm{C}\x_n- \bm{C}\x\|^2_{\bm{S}^{-1}} \nonumber\\
	&\hspace{1cm}-2\tau\scal{\p_{n+1}-\x}{\bm{C}\x_n-\bm{C}\x} +4\hat{\rho} \scal{\x_n-\p_{n+1}}{\bm{C}\x_n-\bm{C}\x}.
\end{align}	
Using the $\beta$-cocoercivity of $\bm{C}$, the last two terms of \eqref{eq:proof10} can be upper-bounded as follows: 
\begin{align}\label{eq:proof11}
	-2\tau&\scal{\bm{C}\x_n-\bm{C}\x}{\p_{n+1}-\x} -4\hat{\rho}\scal{\p_{n+1}-\x_{n}}{\bm{C}\x_n-\bm{C}\x} \nonumber\\
	&= -2(\tau+2\hat{\rho})\scal{\bm{C}\x_n-\bm{C}\x}{\p_{n+1}-\x_{n}}
	-2\tau\scal{\bm{C}\x_n-\bm{C}\x}{\x_{n}-\x}\nonumber\\
	&\le 2|\tau+2\hat{\rho}|
	\|\bm{C}\x_n-\bm{C}\x\|_{\bm{S}^{-1}}\|\p_{n+1}-\x_{n}\|_{\bm{S}}
	-2\tau\scal{\bm{C}\x_n-\bm{C}\x}{\x_{n}-\x}\nonumber\\
	&\leq \varepsilon|\tau+2\hat{\rho}|\|\p_{n+1}-\x_n\|^2_{\bm{S
	}}+\frac{|\tau+2\hat{\rho}|}{\varepsilon}\|\bm{C}\x_{n}-\bm{C}\x\|^2_{\bm{S
		}^{-1}}-2\beta\tau\|\bm{C}\x_n-\bm{C}\x\|^2_{\bm{S^{-1}
	}}\nonumber\\
	&= \varepsilon|\tau+2\hat{\rho}|\|\p_{n+1}-\x_n\|^2_{\bm{S
	}}+\frac{1}{\varepsilon}\left(|\tau+2\hat{\rho}|-2\beta\tau\varepsilon\right)\|\bm{C}\x_n-\bm{C}\x\|^2_{\bm{S
		}^{-1}}.
\end{align}
Then, combining \eqref{eq:proof10} and \eqref{eq:proof11} leads to
\begin{align}\label{eq:proof12}
	\|\p_{n+1}-\x\|^2_{\bm{S}}+&2\scal{\bu_{n+1}}{\p_{n+1}-\x} \nonumber\\
	&\leq  \|\x_{n}-\x\|^2_{\bm{S}}+2\scal{\bu_{n}}{\p_{n}-\x}+\lambda_1\|\p_{n}-\x_{n-1}\|^2_{\bm{S}} \nonumber\\
	&\hspace{1cm}-\big(1-\varepsilon|\tau+2\hat{\rho}|-\zeta+2\hat{\rho}(2\zeta +\zeta_{\bm{M}}(\tau \zeta_{\bm{M}}+ \zeta ))\big)\|\p_{n+1}-\x_n\|^2_{\bm{S}} \nonumber\\
	&\hspace{1cm}+ \frac{1}{\varepsilon}\left(|\tau+2\hat{\rho}|-2\varepsilon(\beta\tau+\hat{\rho}(\tau+\zeta))\right) \| \bm{C}\x_n- \bm{C}\x\|^2_{\bm{S}^{-1}}.
\end{align}	
According to Assumption \ref{assum:2}\ref{assum:24},
$\tau > \underline{\tau}_\zeta$ $\Rightarrow$
$\beta\tau+\hat{\rho}(\tau+\zeta)> 0$.
Thus, to eliminate the dependence
on $\|\bm{C}\x-\bm{C}\x_n \|^2_{\bm{S}^{-1}}$, we set
$\varepsilon =\dfrac{|\tau+2\hat{\rho}|}{2\beta\tau+2\hat{\rho}(\tau+\zeta)}$
in the resulting bound when $\tau+2\hat{\rho}\neq 0$, leading to
\begin{align}\label{eq:proof12bis}
	\|\p_{n+1}-\x\|^2_{\bm{S}}&+2\scal{\bu_{n+1}}{\p_{n+1}-\x} \nonumber\\
	&\leq  \|\x_{n}-\x\|^2_{\bm{S}}+2\scal{\bu_{n}}{\p_{n}-\x}+\lambda_1\|\p_{n}-\x_{n-1}\|^2_{\bm{S}} \nonumber\\
	&\hspace{0.5cm}-\Big(1-\frac{(\tau+2\hat{\rho})^2}{2\beta\tau+2\hat{\rho}(\tau+\zeta)}-\zeta+2\hat{\rho}(2\zeta +\zeta_{\bm{M}}(\tau \zeta_{\bm{M}}+ \zeta ))\Big)\|\p_{n+1}-\x_n\|^2_{\bm{S}}.
\end{align}	
It follows from \eqref{eq:proof12}
that \eqref{eq:proof12bis} also holds
if $\tau+2\hat{\rho}= 0$.
In addition, \eqref{eq:NLFBMM-c} yields
	\begin{align}\label{eq:proof13}
		\|\x_{n+1}-\x\|^{2}_{\bm{S}}=&(1-\theta)\|\x_n-\x\|_{\bm{S}}^{2}+\theta\|\p_{n+1}-\x\|_{\bm{S}}^{2}-\theta(1-\theta)\|\p_{n+1}-\x_n\|_{\bm{S}}^{2}. 
	\end{align}
Therefore, using \eqref{eq:proof12bis} and \eqref{eq:proof13}, and adding $\theta\lambda_1\|\p_{n+1}-\x_{n}\|^2_{\bm{S}}$, we conclude that
\begin{align}
	\|\x_{n+1}-\x\|^2_{\bm{S}}&+2\theta\scal{\bu_{n+1}}{\p_{n+1}-\x} +\theta\lambda_1\|\p_{n+1}-\x_{n}\|^2_{\bm{S}}\nonumber\\
	&\leq  \|\x_{n}-\x\|^2_{\bm{S}}+2\theta\scal{\bu_{n}}{\p_{n}-\x}+\theta \lambda_1\|\p_{n}-\x_{n-1}\|^2_{\bm{S}} \nonumber\\
	&-\theta\left(2-\theta-2\eta\zeta
    -\frac{(\tau+2\widehat\rho)^2}
{2\beta\tau+2\widehat\rho(\tau+\zeta)}+2\widehat\rho
\big(
2\zeta+\zeta_{\bm M}
(\tau\zeta_{\bm M}+\zeta)
\big)
%\frac{2\hat{\rho}}{\tau}\left( (\zeta+\zeta_{\bm{M}}\tau)^2+ 4\zeta\tau \right)
\right)\|\p_{n+1}-\x_n\|^2_{\bm{S}}
\end{align}	
and the result follows.
\item Let us now prove the second assertion. We have
\begin{align*}
    a_{n+1} &= \|\x_{n+1}-\x\|^2_{\bm{S}}+2\theta\scal{\bu_{n+1}}{\p_{n+1}-\x} +\theta\lambda_1\|\p_{n+1}-\x_{n}\|^2_{\bm{S}}\\
    &= \|\x_{n+1}-\x\|^2_{\bm{S}} +2\theta\scal{\bu_{n+1}}{\p_{n+1}-\x_{n+1}}+2\theta\scal{\bu_{n+1}}{\x_{n+1}-\x} +\theta\lambda_1\|\p_{n+1}-\x_{n}\|^2_{\bm{S}}\\
    &= \|\x_{n+1}-\x\|^2_{\bm{S}} +2\theta(1-\theta)\scal{\bu_{n+1}}{\p_{n+1}-\x_{n}}+2\theta\scal{\bu_{n+1}}{\x_{n+1}-\x}+\theta\lambda_1\|\p_{n+1}-\x_{n}\|^2_{\bm{S}}\\
%& \geq \|\x_{n+1}-\x\|^2_{\bm{S}} -\theta\nu\|\p_{n+1}-\x_{n}\|_{\bm{S}}^2+2\theta\scal{\bu_{n+1}}{\x_{n+1}-\x} +\theta\lambda_1\|\p_{n+1}-\x_{n}\|^2_{\bm{S}} \\
%& \geq \|\x_{n+1}-\x\|^2_{\bm{S}} -\theta\nu\|\p_{n+1}-\x_{n}\|_{\bm{S}}^2-2\theta\|\bu_{n+1}\|_{\bm{S}^{-1}}\|\x_{n+1}-\x\|_{\bm{S}}+\theta\lambda_1\|\p_{n+1}-\x_{n}\|^2_{\bm{S}} \\
%& \geq \|\x_{n+1}-\x\|^2_{\bm{S}} -\theta\nu\|\p_{n+1}-\x_{n}\|_{\bm{S}}^2-\frac{\theta}{\zeta}\|\bu_{n+1}\|_{\bm{S}^{-1}}^2+\theta\zeta \|\x_{n+1}-\x\|_{\bm{S}}^2 +\theta\lambda_1\|\p_{n+1}-\x_{n}\|^2_{\bm{S}} \\
& \geq \|\x_{n+1}-\x\|^2_{\bm{S}} -2\theta(\eta-1)\zeta\|\p_{n+1}-\x_{n}\|_{\bm{S}}^2-\theta( \zeta^{-1}\|\bu_{n+1}\|^2_{\bm{S}^{-1}} +\zeta\|\x_{n+1}-\x\|^2_{\bm{S}})\nonumber\\ &\quad+\theta\lambda_1\|\p_{n+1}-\x_{n}\|^2_{\bm{S}} \\
& \geq (1-\theta\zeta)\|\x_{n+1}-\x\|^2_{\bm{S}} -\theta((2\eta-1)\zeta-\lambda_1)\|\p_{n+1}-\x_{n}\|^2_{\bm{S}} \\
& = (1-\theta \zeta )\|\x_{n+1}-\x\|^2_{\bm{S}}-2\theta\hat{\rho}\zeta\big(\tau^{-1}\zeta+\zeta_{\bm{M}}+2\big)
%\frac{\theta\hat{\rho}\zeta}{\tau}(\zeta+\zeta_{\bm{M}}\tau+ 2\tau )
\|\p_{n+1}-\x_{n}\|^2_{\bm{S}}\\
& \geq (1-\theta \zeta )\|\x_{n+1}-\x\|^2_{\bm{S}}\geq 0,
\end{align*}
where the third equality stems from
\eqref{eq:NLFBMM-c},
the first inequality  from \eqref{eq:proof5},
and the second one from
\eqref{eq:proof3}. The positivity of $(1-\theta\zeta)$ is ensured since $\zeta < 1/2$ and $\theta < 2$.
\end{enumerate}
	\end{proof}
\begin{rem}
When $\theta\in[1,2[$ and $\bm{S}-\tau\bm{M}$ is monotone, one can take $\eta=1$ instead of $\eta=\theta$. This increases $\lambda$ by $2(\theta-1)\zeta$, thereby making the condition $\lambda>0$ less restrictive.
%The definition of parameter $\eta$  in \eqref{def:nu} shows that, when $(\bm{S}-\tau \bm{M})$ is monotone and $\theta\geq 1$, $\lambda$ is larger, which is favorable in terms of  convergence behaviour.
%yields less restrictive parameter conditions for guaranteeing convergence. %As shown in Section~\ref{sec:MR}, the assumption that $-(\tau \bm{M}-\bm{S})$ is monotone holds in our applications.
\end{rem}

    The theorem below establishes the convergence of Algorithm~\ref{algo:NLFBMM} to a solution to Problem~\ref{pro:1}.
	\begin{teo}
		\label{teo:NLFBMM} In the context of Problem~\ref{pro:1},
		let $(\x_0,\bu_0)\in \bm{\H}^2$ and consider the sequence $(\x_n,\bu_n)_{n\geq 1}$ generated by Algorithm~\ref{algo:NLFBMM}. Suppose that
		$\lambda$ defined in \eqref{e:defmutn}
		is positive.		
		Then the following hold.
		\begin{enumerate}
			\item\label{teo:NLFBMM1}  $(\x_n)_{n\in\N}$ converges 
			weakly to a solution to Problem~\ref{pro:1}.
			\item\label{teo:NLFBMM2} If $\bm{A}$ is $(\mu\bm{S},\rho\bm{S}^{-1})$-semimonotone on $\pmb{\mathbb{S}}$ for some $\mu \in \RPP$, then $(\x_n)_{n\in \N}$
			converges R-linearly to the unique solution to
			Problem~\ref{pro:1}.
		\end{enumerate}
	\end{teo}
	\begin{proof}
		\begin{enumerate}
			\item Let us use the same notation as in  Proposition~\ref{prop:NLFBMMnew}. 
			%Let us first notice
			%that $(a_n)_{n \in \N}$, defined in \eqref{e:defan}, is nonnegative since since \textcolor{blue}{$(1-\zeta \theta) >0$} .
			%$\bm{\xi}_{n}$ be defined as in \eqref{e:defbxi}.
			Since {$(a_n)_{n \in \N}$, defined in \eqref{e:defan}, is nonnegative} (Proposition~\ref{prop:NLFBMMnew}\ref{prop:NLFBMMnewii}) and $\lambda$ defined in \eqref{e:defmutn} is assumed to be positive, Proposition~\ref{prop:NLFBMMnew}\ref{prop:NLFBMMnewi} shows that $(a_n)_{n\in \N}$ is nonincreasing, 
            convergent,  
			and $\sum_{n \in \N} \|\p_{n+1}-\x_{n}\|^2_{\bm{S}} < + \infty$. 
            In particular,  we deduce from \eqref{eq:anboundbelow} %that %$(\x_n)_{n \in \N}$ is bounded and, consequently, $(\p_n)_{n \in \N}$ is also bounded. We conclude 
            that $(\x_n)_{n \in \N}$ is bounded.
            Since $\p_{n}-\x_{n-1}\to\bm{0}$, $(\p_n)_{n\in\N}$ is also bounded and 
            it follows from \eqref{eq:proof3}
            that $\bu_{n}\to \bm{0}$.
            We deduce from these properties and  \eqref{e:defan} that
            $\lim_{n\to+\infty} \|\x_{n}-\x\|^2_{\bm{S}} = \lim_{n\to+\infty} a_n$
            is finite.
            Additionally, by the Lipschitz property of $(\tau \bm{M}-\bm{S})$ with respect to $\bm{S}$, we have
			\begin{align}\label{eq:prooft1}
				\|\bm{S}^{-1}\tau( \bm{M}{\p_{n+1}}-\bm{M}\x_n)\|_{\bm{S}}&=\|\tau \bm{M}{\p_{n+1}}-\tau\bm{M}\x_n\|_{\bm{S}^{-1}} \nonumber\\
				&\leq \|(\tau \bm{M}-\bm{S}){\p_{n+1}}-(\tau\bm{M}-\bm{S})\x_n\|_{\bm{S}^{-1}}+\|{\p_{n+1}}-\x_n\|_{\bm{S}}\nonumber\\
				&\leq (\zeta+1)\|\p_{n+1}-\x_n\|_{\bm{S}}.
			\end{align}
			Since $\tau>0$ and $\bm{S}$ is bounded, $\bm{M}\p_{n+1}-\bm{M}\x_n \to 0$. %Additionally, from the convergence of $(a_n)_{n \in \N}$
			%we deduce the convergence of 
			%$(\|\bm{\xi}_{n}-\x\|_{\bm{S}})_{n\in \N}$, 
			%which, from the definition of $(\bm{\xi}_n)_{n\in \N}$, and the fact that $\bm{S}^{-1}(\tau\bm{M}\x_{n+1}-\tau\bm{M}\x_n) \to 0$, allows us to conclude that $(\|\x_{n}-\x\|_{\bm{S}})_{n\in \N}$
			%is convergent. Hence, $(\x_n)_{n \in \N}$,  is bounded. 
            %Now, note that,
			%\begin{align}\label{eq:prooft2}
			%	\|\bu_{n+1}\|_{\bm{S}^{-1}} &= \|(\tau \bm{M} - \bm{S})\p_{n+1} -(\tau \bm{M} -\bm{S})\x_n\|_{\bm{S}^{-1}}\nonumber\\
			%	&\leq \zeta\|\p_{n+1}-\x_n\|^2_{\bm{S}}.
			%\end{align}
			Moreover, the cocoercivity of $\bm{C}$ with respect to $\bm{S}$ yields
			\begin{align}\label{eq:prooft3}
				\|\bm{C}\p_{n+1}-\bm{C}\x_n\|_{\bm{S}^{-1}}\leq \frac{1}{\beta}\|\p_{n+1}-\x_n\|_{\bm{S}}.
			\end{align}
			Therefore, it follows from  \eqref{eq:prooft1}, %\eqref{eq:prooft2}
 \eqref{eq:prooft3} and the facts
 that $\bm{u}_n\to \bm{0}$ and $\p_{n+1}-\x_n \to \bm{0}$ that
			\begin{equation}\label{eq:wnto0}
				\bm{w}_n := \bm{M} \x_n-\bm{M}\p_{n+1}+\bu_n/\tau+(\bm{C}\p_{n+1}-\bm{C}\x_n) \to \bm{0}.
			\end{equation}
			Now, let $(\x_{n_k})_{k\in \N}$ be a weakly convergent subsequence of $(\x_n)_{n \in \N}$. Since $\p_{n+1}-\x_n\to\bm{0}$, $(\p_{n_k+1})_{k\in \N}$ is also weakly convergent and has the same limit. It follows from \eqref{eq:wnto0} that $\bm{w}_{n_k} \to \bm{0}$ and from \eqref{eq:NLFBMM-a} that 
			\begin{align}
				(\forall k \in \N)\quad & \bm{w}_{n_k}
				\in (\bm{A}+\bm{C})\p_{n_k+1}.
			\end{align}	
			Since $\gra(\bm{A}+\bm{C})$ is sequentially weak-strong closed (Assumption~\ref{assum:2}\ref{assum:22}), we conclude that every weak cluster point of $(\x_{n})_{n \in \N}$ belongs to $\zer(\bm{A}+\bm{C})$. 
			The result follows from Opial's lemma \cite[Lemma 2.47]{bauschkebook2017} applied in
            the Hilbert space
        $(\bm{\H},\scal{\cdot}{\cdot}_{\bm S})$.
			\item By Proposition~\ref{prop:uniquesol}, $ \widetilde{\bm{Z}}$ is a singleton. Now, the $(\mu \bm{S},\rho\bm{S}^{-1})$-semimonotonicity of $\bm{A}$ allows us to refine
			\eqref{eq:quasifej} as shown below. First, \eqref{eq:proof0} becomes
			\begin{align}
				\tau\mu\|\p_{n+1}-\x\|^2_{\bm{S}} \leq &\scal{\bm{S}\x_n + \bu_n  - (\bm{S}\p_{n+1}+\bu_{n+1})-\tau \bm{C}\x_n+\tau \bm{C}\x}{\p_{n+1}-\x}\nonumber\\
	& -\frac{\hat{\rho}}{\tau}\|\bm{S}\x_n + \bu_n  - (\bm{S}\p_{n+1}+\bu_{n+1})-\tau \bm{C}\x_n+\tau \bm{C}\x\|^2_{\bm{S}^{-1}}.
			\end{align} Hence, by proceeding similarly to the proof of Proposition~\ref{prop:NLFBMMnew}\ref{prop:NLFBMMnewi}, we obtain
			\begin{align}\label{eq:desanref}
				a_n &\geq 2\tau\mu{\theta}\|\p_{n+1}-\x\|^2_{\bm{S}}+{\theta} \lambda 
				\|\p_{n+1}-\x_n\|^2_{\bm{S}} 
				+a_{n+1},
			\end{align}
			where $a_n$ is defined in \eqref{e:defan}. The parameter $\lambda_1$ defined in \eqref{e:deflambda_1}
            is nonnegative since 
            $\eta\geq 1$, $\zeta\geq 0$, and $\hat\rho\leq 0$.
            Let us 
            set $\vartheta_1=\theta\min\left\{\tau \mu , \frac{\lambda}{3}\right\}$, and $\vartheta_2=\vartheta_1\min\left\{1, \frac{1}{\lambda_1}\right\}$ if $\lambda_1>0$
            and $\vartheta_2=\vartheta_1$ otherwise.
            Since $\zeta\in~[0,1/2[$ and $\theta \in ]0,2[$, it follows
            from \eqref{eq:NLFBMM-c} and \eqref{eq:desanref}  that, for every $n \in \N$,
			\begin{align}
				&a_n-a_{n+1}\nonumber\\ 
				%\geq&\theta \tau\mu\|\p_{n+1}-\x\|^2_{\bm{S}}+\theta \tau\mu\|\p_{n+1}-\x\|^2_{\bm{S}}+\frac{\theta \lambda}{3} \|\p_{n+1}-\x_n\|^2_{\bm{S}}+\frac{\theta \lambda}{3} \|\p_{n+1}-\x_n\|^2_{\bm{S}} + \frac{\theta \lambda}{3} 
				%\|\p_{n+1}-\x_n\|^2_{\bm{S}}\nonumber\\ 
                \geq&\;\vartheta_1
                %\min\left\{\tau \mu , \frac{\lambda}{3}\right\}
                \Big(2\|\p_{n+1}-\x\|^2_{\bm{S}}+3\|\p_{n+1}-\x_n\|^2_{\bm{S}}\Big)\nonumber\\ 
                \geq&\; \vartheta_1
                %\min\left\{\tau \mu , \frac{\lambda}{3}\right\}
                \Big(\frac{1}{2}\|\x_{n+1}-\x\|^2_{\bm{S}}+\|\p_{n+1}-\x\|^2_{\bm{S}}+3\|\p_{n+1}-\x_n\|^2_{\bm{S}}-\|\p_{n+1}-\x_{n+1}\|^2_{\bm{S}}\Big)\nonumber\\ 
                = &\; \frac{\vartheta_1}{2}
                %\min\left\{\tau \mu , \frac{\lambda}{3}\right\} 
                \Big(\|\x_{n+1}-\x\|^2_{\bm{S}}+2\|\p_{n+1}-\x\|^2_{\bm{S}}+2{\big(3-(1-\theta)^2\big)}\|\p_{n+1}-\x_n\|^2_{\bm{S}}\Big)\nonumber\\ 
                = &\;\frac{\vartheta_1}{2}
                %\min\left\{\tau \mu , \frac{\lambda}{3}\right\} 
                \Big(\|\x_{n+1}-\x\|^2_{\bm{S}}+2(\|\p_{n+1}-\x\|^2_{\bm{S}}+\|\p_{n+1}-\x_n\|^2_{\bm{S}})+2(1+2\theta-\theta^2)\|\p_{n+1}-\x_n\|^2_{\bm{S}}\Big)\nonumber\\ 
                 \geq &\; \frac{\vartheta_2}{2}
                 %\min\left\{\tau \mu , \frac{\lambda}{3}\right\} 
                 %\min\left\{1, \frac{1}{\lambda_1}\right\} 
                 \Big(\|\x_{n+1}-\x\|^2_{\bm{S}}+2(\|\p_{n+1}-\x\|^2_{\bm{S}}+\|\p_{n+1}-\x_n\|^2_{\bm{S}})+\theta\lambda_1\|\p_{n+1}-\x_n\|^2_{\bm{S}}\Big),
                 \label{e:dantheta2an}
                 \end{align}
                 where we have used the fact that 
                 $2(1+2\theta-\theta^2)-\theta
=(2-\theta)(2\theta+1)>0$.
%\textcolor{red}{
%$\vartheta_2\leq\vartheta_1$}.
%and
%$\vartheta_2\lambda_1\leq\vartheta_1$}.
%                 \textcolor{blue}{$2(2+2\theta-\theta^2)=2(2+\theta(2-\theta))\geq 4 \geq 2\theta$ and that $\lambda > 0$ implies $2 > \lambda_1$.}
                 When $\zeta>0$, using \eqref{eq:proof3} 
                 and noting that $\theta\zeta\in ]0,1[$ yield
                 \begin{align}
                 &a_n-a_{n+1}\nonumber\\
                \geq &\; \frac{\vartheta_2}{2}
                %\min\left\{\tau \mu , \frac{\lambda}{3}\right\} \min\left\{1, \frac{1}{\lambda_1}\right\}
                \Big(\|\x_{n+1}-\x\|^2_{\bm{S}}+\theta\zeta (\|\p_{n+1}-\x\|^2_{\bm{S}}+\zeta^{-2}\|\bu_{n+1}\|^2_{\bm{S}^{-1}})+\theta\lambda_1\|\p_{n+1}-\x_n\|^2_{\bm{S}}\Big)\nonumber\\ 
                \geq &  \frac{\vartheta_2}{2}
                %\min\left\{\tau \mu , \frac{\lambda}{3}\right\}\min\left\{1, \frac{1}{\lambda_1}\right\} 
                \Big(\|\x_{n+1}-\x\|^2_{\bm{S}}+2\theta\scal{\p_{n+1}-\x}{\bu_{n+1}}+\theta\lambda_1\|\p_{n+1}-\x_n\|^2_{\bm{S}}\Big)\nonumber\\ 
                =&\;\frac{\vartheta_2}{2}
                %\min\left\{\tau \mu , \frac{\lambda}{3}\right\}\min\left\{1, \frac{1}{\lambda_1}\right\} 
                a_{n+1}.
		\end{align}
        When $\zeta=0$, $\lambda_1=0$
        and the inequality $a_n-a_{n+1}\geq (\vartheta_2 a_{n+1})/2$ follows directly from \eqref{e:dantheta2an}.
			Consequently,  $(a_n)_{n \in \N}$ converges Q-linearly to $0$ with rate 
            $1/(1+\vartheta_2/2)$
            and the R-linear convergence of 
			$(\x_n)_{n \in \N}$ to $\x$ follows from \eqref{eq:anboundbelow}.
            %Moreover, by \eqref{e:defan} and
			% \eqref{eq:desxix},
			% %the $\zeta$-Lipschitizian property of $\tau \bm{M}-\bm{S}$, and the definition of $\bm{\xi}_{n}$ in \eqref{e:defbxi}, we have that
			% the following inequalities hold:
			% \begin{align}\label{eq:xnlinear}
			% 	a_{n} &\geq \|\bm{\xi}_n-\x\|^2_{\bm{S}}
			% 	+\zeta(1-\zeta)\|\x_{n}-\x_{n-1}\|^2_{\bm{S}}\nonumber\\
			% 	&\geq \|\bm{\xi}_n-\x\|^2_{\bm{S}} + \frac{(1-\zeta)}{\zeta}
			% 	\|\bm{\xi}_{n}-\x_{n-1}\|^2_{\bm{S}} \nonumber\\
   %              &\geq \|\bm{\xi}_n-\x\|^2_{\bm{S}} + 
			% 	\|\bm{\xi}_{n}-\x_{n-1}\|^2_{\bm{S}}\nonumber\\
			% 	&\geq \frac{1}{2}\|\x_{n-1}-\x\|^2_{\bm{S}}.%\min\left\lbrace 1, \frac{(1-\zeta)}{\zeta}\right\rbrace 
			% \end{align} 
			% The linear convergence of 
			% $(\x_n)_{n \in \N}$ to $\x$ follows.
		\end{enumerate}
	\end{proof}
%\begin{rem}\label{rem:condzetarho}
%We can provide a simple necessary condition on $\zeta$
%for $\lambda>0$.
%, in accordance with Assumption~\ref{assum:2}\ref{assum:24}, 
%$\zeta \in [0,1/(2(1-4\hat{\rho}))[$
%is a necessary
%condition for $\lambda>0$.  Indeed, 
%Indeed, by setting $\xi = \zeta_{\bm{M}}\tau$, we have $2-\theta-2(1+\nu)\zeta+\frac{2\hat{\rho}}{\tau}
%            \left( \left(\zeta+\xi\right)^2+4\tau\zeta\right)\geq \lambda>0$, hence $\tau (2-\theta-2(1+\nu)\zeta+8\hat{\rho}\zeta) >-2\hat{\rho}(\zeta+\xi)^2$, which implies that $2-\theta-2(1+\nu-4\hat{\rho})\zeta> 0$.
%            This is satisfied when
%            $\zeta \in [0,(2-\theta)/(2(1+\nu-4\hat{\rho}))[$. 
            %If $\theta\in [1,2]$ and $(\bm{S}-\tau \bm{M})$ is  monotone, then $\zeta \in [0,(2-\theta)/(2(1-4\hat{\rho}))[$. Otherwise, $\zeta \in [0,(2-\theta)/(2(1-4\hat{\rho}+|1-\theta|))[$. 
%            This constitutes a subinterval of the one initially considered in Assumption~\ref{assum:2}\ref{assum:24}.
%\end{rem}
We now provide further insight into the choice of the step size \(\tau\) ensuring the convergence of Algorithm~\ref{algo:NLFBMM}.

    \begin{prop}\label{prop:overlinezeta}
        In the context of Problem~\ref{pro:1}, with the notation of
Proposition~\ref{prop:NLFBMMnew},
        the following properties hold.
\begin{enumerate}        
\item \label{prop:overlinezetai}       For fixed $\zeta$, $\xi$, and $\eta$, the convergence condition
        $\lambda>0$ is equivalent to the following cubic polynomial inequality:
   \begin{equation}
   \label{e:Psup0}
   P(\tau)= -\tau^3+a_2\tau^2+a_1\tau+a_0>0,
   \end{equation}
   where
          \begin{equation*}
        a_0 = 4\hat{\rho}^2\zeta(\zeta+\xi)^2, \quad a_1 = 2c\hat{\rho}\zeta - 4\hat{\rho}^2 + 4\hat{\rho}(\beta+\hat{\rho})(\zeta+\xi)^2, \quad 
        a_2 =2 c(\beta+\hat{\rho}) - 4\hat{\rho},
    \end{equation*} 
    with $c =2-\theta-2(\eta-4\hat{\rho})\zeta$
%    \begin{equation*}
%        a_0 = 4\hat{\rho}^2\zeta(\zeta+\xi)^2, \quad a_1 = 2(1-2\zeta+8\hat{\rho}\zeta)\hat{\rho}\zeta+4\hat{\rho}(\beta+\hat{\rho})(\zeta+\xi)^2-4\hat{\rho}^2, \quad 
%        a_2 = 2(1-2\zeta+8\hat{\rho}\zeta)(\beta+\hat{\rho})-4\hat{\rho}.
%    \end{equation*} 
    and $\xi \in [1-\zeta,1+\zeta]$ as defined in Remark~\ref{re:firstrem}\ref{re:firstremi}.
\item \label{prop:overlinezetaii} If \eqref{e:Psup0} holds, then $\zeta \in [0,(2-\theta)/(2(\eta-4\hat{\rho}))[$.
\item\label{prop:overlinezetaiii} If $\theta \in \Big]0 ,\frac{2}{1+\sqrt{{-\hat{\rho}/\beta }} }\Big[$, there exists 
\begin{equation}\label{e:overlinezeta}
        \overline{\zeta} \in \left]0, \min\left\lbrace \frac{2-\theta}{2(\eta-4\hat{\rho})}, \frac{(\beta+\hat{\rho})((2-\theta)(\beta+\hat{\rho})-4\hat{\rho})}
    {2(\eta-4\hat{\rho})(\beta+\hat{\rho})^2-2\hat{\rho}}
    %    \frac{\beta^2-\hat{\rho}^2}
    %{2(1-4\hat{\rho})(\beta+\hat{\rho})^2-2\hat{\rho}}
    \right\rbrace
    \right[
    \end{equation}
    such that,
    for every $\zeta\in [0,\overline{\zeta}[$ and
    $\tau \in ]\tau_{1,\zeta},\tau_{2,\zeta}[$, \eqref{e:Psup0}  holds, where
    $(\tau_{1,\zeta},\tau_{2,\zeta})\in [\underline{\tau}_\zeta,+\infty[^2$ with $\tau_{1,\zeta} <\tau_{2,\zeta}$.\footnote{The quantities $\tau_{1,\zeta}$ and $\tau_{2,\zeta}$ also depend on $\xi$ and $\eta$, but, for notational simplicity, we omit this dependence.}
    
%    In particular,
%    \begin{itemize}
%        \item  If  $\tau \bm{M}=\bm{S}$, then
%        $\zeta = 0$, $\overline{\theta}=2(\beta+\hat{\rho})^{-1}\sqrt{\beta} (\sqrt{\beta}-\sqrt{-\hat{\rho}})$,   
%        $\tau_{1,0} = -2\hat{\rho}$,
%        and $\tau_{2,0}=2\beta$.
%        \begin{align}
%        \overline{\theta} &= 2\frac{\sqrt{\beta}} {\sqrt{\beta}+\sqrt{-\hat{\rho}}}\\
        %{\beta+\hat{\rho}}\\
%            \tau_{1,0} &= (2-\theta)(\beta+\hat{\rho})-2\hat{\rho}
%            -\sqrt{((2-\theta)(\beta+\hat{\rho})-2\hat{\rho})^2+4\hat{\rho}\beta},\\
%            \tau_{2,0} &= (2-\theta)(\beta+\hat{\rho})-2\hat{\rho}
%            +\sqrt{((2-\theta)(\beta+\hat{\rho})-2\hat{\rho})^2+4\hat{\rho}\beta}.
%        \end{align}
%        \item If  $\bm{C}=\bm{0}$, 
%        then $\overline{\theta}=2$, $\overline{\zeta} = \frac{2-\theta}{2(1+\nu-4\hat{\rho})}$,
%        and, for every $\zeta \in ]0,\overline{\zeta}[$,
%        $\tau_{1,\zeta} = 0$ and $\tau_{2,\zeta}=-2\hat{\rho}(\zeta+\xi)^2/c$.
%        \item If $\bm{A}$ is monotone, then  $\overline{\theta}=2$,
%       $\overline{\zeta} = \frac{2-\theta}{2(1+\nu)}$ and, for every $\zeta \in ]0,\overline{\zeta}[$,
%        $\tau_{1,\zeta} =0$ and $\tau_{2,\zeta}=2\beta
%        (2-\theta-2(1+\nu)\zeta)$.        
%        (1-2\zeta)$.

%    \end{itemize}
    \end{enumerate}
    \end{prop}
    \begin{proof}
    \begin{enumerate}
    \item
        According to \eqref{e:defmutn} and Remark~\ref{re:firstrem}\ref{re:firstremi}, $\lambda$ can be written as \begin{align}\label{e:lambda2}
       \lambda &=\frac{P(\tau)}{2\tau \brk1{(\beta+\hat{\rho})\tau+\hat{\rho}\zeta}},
    \end{align}
    where 

\begin{align}\label{e:defP}
        P(\tau) &=  
        2\brk2{(2-\theta-2\eta\zeta)\tau+2\hat{\rho}\brk1{(\zeta+\xi)^2+4\tau\zeta}}\brk1{(\beta+\hat{\rho})\tau+\hat{\rho}\zeta}-\tau(\tau+2\hat{\rho})^2\\
       &=2\brk2{\tau c+2\hat{\rho}{(\zeta+\xi)^2}}\brk1{(\beta+\hat{\rho})\tau+\hat{\rho}\zeta}-\tau(\tau+2\hat{\rho})^2\nonumber\\
        & = -\tau^3+a_2\tau^2+a_1\tau+a_0.\nonumber
    \end{align}      %\begin{align}\label{e:defP}
    %    P(\tau) &=  
    %    2\brk2{(1-2\zeta)\tau+2\hat{\rho}\brk1{(\zeta+\xi)^2+4\tau\zeta}}\brk1{(\beta+\hat{\rho})\tau+\hat{\rho}\zeta}-\tau(\tau+2\hat{\rho})^2\\
    %    & = -\tau^3+a_2\tau^2+a_1\tau+a_0.\nonumber
    %\end{align} 
    In view of Assumption~\ref{assum:2}\ref{assum:24},
$(\beta+\hat\rho)\tau+\hat\rho\zeta>0$,
and since $\tau>0$, the denominator in \eqref{e:lambda2}
is positive. Hence, $\lambda>0$ if and only if $P(\tau)>0$.    

\item %We can provide a simple necessary condition on $\zeta$
Assume that
%for 
$\lambda>0$.
%, in accordance with Assumption~\ref{assum:2}\ref{assum:24}, 
%$\zeta \in [0,1/(2(1-4\hat{\rho}))[$
%is a necessary
%condition for $\lambda>0$.  Indeed, 
%Indeed, by setting $\xi = \zeta_{\bm{M}}\tau$, 
We have $2-\theta-2\eta\zeta+\frac{2\hat{\rho}}{\tau}
            \left( \left(\zeta+\xi\right)^2+4\tau\zeta\right)\geq \lambda>0$, hence $\tau (2-\theta-2\eta\zeta+8\hat{\rho}\zeta) >-2\hat{\rho}(\zeta+\xi)^2$, which implies that 
            %$2-\theta-2(1+\nu-4\hat{\rho})\zeta> 0$.
            $c>0$.
            This is equivalent to
            $\zeta \in [0,(2-\theta)/(2(\eta-4\hat{\rho}))[$. 
            %If $\theta\in [1,2]$ and $(\bm{S}-\tau \bm{M})$ is  monotone, then $\zeta \in [0,(2-\theta)/(2(1-4\hat{\rho}))[$. Otherwise, $\zeta \in [0,(2-\theta)/(2(1-4\hat{\rho}+|1-\theta|))[$. 
            %This constitutes a subinterval of the one initially considered in Assumption~\ref{assum:2}\ref{assum:24}.
\item     
Assume that $c>0$.
We have $a_0\geq 0$, %\textcolor{blue}{if $c \geq 0$,}
%by Assumption~\ref{assum:2}\ref{assum:24}, 
$a_2 >0$, and $a_1 \leq 0$.
% In addition,\textcolor{blue}{if $c \geq 0$,} $a_1\le 0$ since
 %\begin{align}
 %a_1 
 %&= 2(1-2\zeta+8\hat{\rho}\zeta)\hat{\rho}\zeta+4\hat{\rho}(\beta+\hat{\rho})(\zeta+\xi)^2-4\hat{\rho}^2\nonumber\\
 %&= 2(1-2\zeta)\hat{\rho}\zeta
 %+4\hat{\rho}(\beta+\hat{\rho})
%(\zeta+\xi)^2+4\hat{\rho}^2(4\zeta^2-1).    
 %\end{align}
     Consequently, for every $\tau\in]-\infty,0[$, $P(\tau)> 0$, which means that all the real roots of $P$ are nonnegative. 
  Note also that $P(\underline{\tau}_\zeta)\le  0$.
    Indeed, 
    $\underline{\tau}_\zeta= -\hat{\rho}\zeta/(\beta+\hat{\rho})$ $\Rightarrow$ 
    $P(\underline{\tau}_\zeta) = -\underline{\tau}_\zeta(\underline{\tau}_\zeta+2\hat{\rho})^2 \le  0$.
     Since  $P(0)=a_0\geq0$ and $P(\underline\tau_\zeta)\leq0$,     
%     $\lim_{\tau \to -\infty} P(\tau) = +\infty$, 
$P$ has a root $\tau_{0,\zeta} \in [0, \underline{\tau}_\zeta]$. 
  Two cases may arise: either \textsc{I}) $P$ is nonpositive on
$]\underline{\tau}_\zeta,+\infty[$ (in particular, if it is negative over $]\underline{\tau}_\zeta,+\infty[$, then $P$ has all its real roots in $[0,\underline{\tau}_\zeta]$) or \textsc{II}) there exists $\tau \in ]\underline{\tau}_\zeta,+\infty[$ such that $P(\tau) > 0$ and $P$ has two roots $\tau_{1,\zeta} \geq \underline{\tau}_{\zeta}$ and $\tau_{2,\zeta} > \tau_{1,\zeta}$. 
  In the latter case, for every $\tau \in ]\tau_{1,\zeta},\tau_{2,\zeta}[$, $P(\tau) > 0$.

%  Since $\hat{\rho}\big((\beta+\hat{\rho})\tau+\hat{\rho}\zeta)\le 0$,
%and $1+2\zeta\geq\zeta+\xi$,
We have, for every $\tau\in\R$,
% \begin{align}
% P(\tau)&= 
% -\tau^3+a_2\tau^2+
%  2\big((1-2\zeta+8\hat{\rho}\zeta)\hat{\rho}\zeta-2\hat{\rho}^2\big)\tau
%  +4\hat{\rho}(\zeta+\xi)^2\big((\beta+\hat{\rho})\tau+\hat{\rho}\zeta)\nonumber\\
%  &\geq 
%  -\tau^3+a_2\tau^2+
%  2\big((1-2\zeta+8\hat{\rho}\zeta)\hat{\rho}\zeta-2\hat{\rho}^2\big)\tau
%  +4\hat{\rho}(1+2\zeta)^2\big((\beta+\hat{\rho})\tau+\hat{\rho}\zeta).
% \label{eq:deltaP}
%  \end{align}
 \begin{align}
P(\tau)&= 
-\tau^3+a_2\tau^2+
2\hat{\rho}\big(c\zeta-2\hat{\rho}\big)\tau
 +4\hat{\rho}(\zeta+\xi)^2\big((\beta+\hat{\rho})\tau+\hat{\rho}\zeta).\label{e:defP2}
 %\\
 %&\geq 
 %-\tau^3+a_2\tau^2+
 %2\hat{\rho}\big(c\zeta-2\hat{\rho}\big)\tau
 %+4\hat{\rho}(1+2\zeta)^2\big((\beta+\hat{\rho})\tau+\hat{\rho}\zeta).
%\label{eq:deltaP}
 \end{align}
Let
\begin{equation}
    \widetilde{\tau}_\zeta =
    a_2/2 =c(\beta+\hat{\rho})-2\hat{\rho}.
    %=(1-2\zeta+8\hat{\rho}\zeta)(\beta+\hat{\rho})-2\hat{\rho}.
\end{equation}
%Note that, as $\zeta \to 0$, $\underline{\tau}_{\zeta}
%\to 0$\textcolor{blue}{, $c \to 2-\theta$,} and $\widetilde{\tau}_\zeta \to \textcolor{blue}{(2-\theta)\beta-\theta\hat{\rho}>0}$ , thus, for $\zeta$ small enough we have $\widetilde{\tau}_\zeta> \underline{\tau}_\zeta$. 
Now, from \eqref{e:defP2},
\begin{align}    P(\widetilde{\tau}_\zeta)&= 
\widetilde{\tau}_\zeta^3+
 2\hat{\rho}\big(c\zeta-2\hat{\rho}\big)\widetilde{\tau}_\zeta
 +4\hat{\rho}(\zeta+\xi)^2\big((\beta+\hat{\rho})\widetilde{\tau}_\zeta+\hat{\rho}\zeta)\nonumber\\&= 
\widetilde{\tau}_\zeta(\widetilde{\tau}_\zeta^2+
 2\hat{\rho}\big(c\zeta-2\hat{\rho}))
 +4\hat{\rho}(\zeta+\xi)^2(\beta+\hat{\rho})\big(\widetilde{\tau}_\zeta-\underline{\tau}_\zeta)\nonumber\\
 &= 
\widetilde{\tau}_\zeta c (\beta+\hat{\rho})(\widetilde{\tau}_\zeta-2\hat{\rho}-2\underline{\tau}_\zeta)
 +4\hat{\rho}(\zeta+\xi)^2(\beta+\hat{\rho})\big(\widetilde{\tau}_\zeta-\underline{\tau}_\zeta)\nonumber\\
 &=(\beta+\hat{\rho})\big(\widetilde{\tau}_\zeta c (\widetilde{\tau}_\zeta-2\hat{\rho}-2\underline{\tau}_\zeta)
 +4\hat{\rho}(\zeta+\xi)^2\big(\widetilde{\tau}_\zeta-\underline{\tau}_\zeta)\big).\label{eq:Ptildetau}\\
 %&= \textcolor{blue}{(\beta+\hat{\rho})\big(\widetilde{\tau}_\zeta c (-2\hat{\rho}-\underline{\tau}_\zeta)+
 %\big( \widetilde{\tau}_\zeta c +4\hat{\rho}(\zeta+\xi)^2\big)(\widetilde{\tau}_\zeta-\underline{\tau}_\zeta)\big)}
 %\label{e:Ptaut}
\end{align}
% \begin{align}
%     P(\widetilde{\tau}_\zeta) &= (1-2\zeta+8\hat{\rho}\zeta)^2(\beta+\hat{\rho})( \widetilde{\tau}_\zeta (\beta+\hat{\rho}) + 2\zeta \hat{\rho} ) +8\hat{\rho}\zeta(3+2(\zeta-2\hat{\rho})) ( \widetilde{\tau}_\zeta (\beta+\hat{\rho}) + \zeta \hat{\rho} )\nonumber\\
%     &= (1-2\zeta+8\hat{\rho}\zeta)^2(\beta+\hat{\rho})^2( \widetilde{\tau}_\zeta - 2\underline{\tau}_\zeta ) +8\hat{\rho}\zeta(3+2(\zeta-2\hat{\rho})) (\beta+\hat{\rho}) ( \widetilde{\tau}_\zeta - \underline{\tau}_\zeta)%\label{e:Ptz1}
%     \nonumber\\
%     &= (\beta+\hat{\rho})^2\big((1-2\zeta+8\hat{\rho}\zeta)^2( \widetilde{\tau}_\zeta - 2\underline{\tau}_\zeta ) -8\underline{\tau}_\zeta(3+2(\zeta-2\hat{\rho}))( \widetilde{\tau}_\zeta - \underline{\tau}_\zeta)\big).\label{e:Ptz2}
% \end{align}
Suppose
that 
\begin{equation}
\label{e:tautlesstauprho}
\widetilde{\tau}_\zeta >  2(\underline{\tau}_\zeta+\hat{\rho}).
\end{equation}
%The last additive term in the above expression is negative when $\widetilde{\tau}_\zeta - \underline{\tau}_\zeta>0$, thus, to allow $P(\widetilde{\tau}_\zeta)$ to be positive, we need the first term to be positive. This occurs when  $\widetilde{\tau}_\zeta >  2(\underline{\tau}_\zeta+\textcolor{blue}{\hat{\rho}})$.
From the expressions of $\widetilde{\tau}_\zeta$
and $\underline{\tau}_\zeta$, the latter inequality is satisfied if and only if 
\begin{equation}\label{eq:deszeta1}
    \zeta  < \zeta_1=\frac{(\beta+\hat{\rho})((2-\theta)(\beta+\hat{\rho})-4\hat{\rho})}
    {2(\eta-4\hat{\rho})(\beta+\hat{\rho})^2-2\hat{\rho}},
\end{equation}
where the upper bound is positive.
Note that $\widetilde{\tau}_\zeta > \underline{\tau}_\zeta$.
Indeed, if $-2\hat{\rho} \geq \underline{\tau}_\zeta $, then $\widetilde{\tau}_\zeta = c(\beta+\hat{\rho})-2\hat{\rho} >  \underline{\tau}_\zeta $. %and $\widetilde{\tau}_\zeta >  2(\underline{\tau}_\zeta+\hat{\rho})$. 
If $-2\hat{\rho} < \underline{\tau}_\zeta$, then, the condition $\widetilde{\tau}_\zeta >  2(\underline{\tau}_\zeta+\hat{\rho})$ implies that $\widetilde{\tau}_\zeta > \underline{\tau}_\zeta$.
Due to this inequality, the last additive term in \eqref{eq:Ptildetau} is nonpositive. Thus, \eqref{e:tautlesstauprho} is a necessary condition for $P(\widetilde{\tau}_\zeta)$ to be positive.

%which also holds for sufficiently small $\zeta$. 

Let us now look at 
what happens when $\zeta$
is small enough. 
Since $\xi\in [1-\zeta,1+\zeta]$, $\xi \to 1$
as $\zeta\to 0$.
In addition,
as $\zeta \to 0$, $\underline{\tau}_{\zeta}
\to 0$, $c \to 2-\theta$, $\widetilde{\tau}_\zeta \to (2-\theta)\beta-\theta\hat{\rho}>0$. Thus, we have
\begin{align*}
%{\underline{\tau}}_{\zeta}
%\to 0, \quad \widetilde{\tau}_\zeta \to \beta-\hat{\rho}, \quad\text{and} \quad
%(1-2\zeta+8\hat{\rho}\zeta)^2(\beta+\hat{\rho})^2( \widetilde{\tau}_\zeta - 2\underline{\tau} ) \to 
%(\beta-\hat{\rho})(\beta+\hat{\rho})^2 > 0,
\widetilde{\tau}_\zeta c (\widetilde{\tau}_\zeta-2\hat{\rho}-2\underline{\tau}_\zeta)
 +4\hat{\rho}(\zeta+\xi)^2\big(\widetilde{\tau}_\zeta-\underline{\tau}_\zeta)\to 
&((2-\theta)\beta-\theta\hat{\rho})((2-\theta)^2\beta + \theta^2 \hat{\rho}).
\end{align*}
The limit
is strictly positive when $\theta \in \Big]0 ,\frac{2}{1+\sqrt{{-\hat{\rho}/\beta }} }\Big[$ (note that $\frac{2}{1+\sqrt{{-\hat{\rho}/\beta }} } \in ]1,2]$). 
%whereas %\linebreak
%$8\hat{\rho}\zeta(3+2(\zeta-2\hat{\rho})) (\beta+\hat{\rho}) ( \widetilde{\tau}_\zeta - \underline{\tau})\sim \zeta \hat{\rho}(\beta+\hat{\rho})$.
Thus, it follows from \eqref{eq:Ptildetau}  that
there exists $\overline{\zeta} \in \RPP$,
independent of $\xi$, such that, for every $\zeta < \overline{\zeta}$, 
$P(\widetilde{\tau}_\zeta) > 0$, which ensures that
we are in case II) mentioned above. In addition to the condition $\overline{\zeta}<\zeta_1$, it follows from the initial assumption $c>0$ and \ref{prop:overlinezetaii} that $\overline{\zeta} \in [0,(2-\theta)/(2(\eta-4\hat{\rho}))[$.
%and, from the necessary condition $\widetilde{\tau}_\zeta >  2(\underline{\tau}_\zeta+\hat{\rho})$,  that
%\begin{equation}\label{eq:deszeta1}
%    \overline{\zeta}  < \zeta_1=\frac{(\beta+\hat{\rho})((2-\theta)(\beta+\hat{\rho})-4\hat{\rho})}
%    {2(1+\nu-4\hat{\rho})(\beta+\hat{\rho})^2-2\hat{\rho}}.
%\end{equation}
%Note that the last inequality also implies that $\widetilde{\tau}_\zeta >  \underline{\tau}_\zeta $.
%\begin{itemize}

\end{enumerate}
    \end{proof}

Let us analyze more precisely some noteworthy special cases.
\begin{cor}\label{coro:overlinezeta}
In the context of Problem~\ref{pro:1},
let $\eta$ be defined by \eqref{def:nu}, let $\xi$ be defined as Remark~\ref{re:firstrem}\ref{re:firstremi}, 
		let $(\x_0,\bu_0)\in \bm{\H}^2$ and consider the sequence $(\x_n,\bu_n)_{n\geq 1}$ generated by Algorithm~\ref{algo:NLFBMM}. 
        %Under the same assumptions as in Theorem \ref{teo:NLFBMM},
        Consider one of the following settings:
\begin{enumerate}        
        \item \label{coro:overlinezetai} (standard forward-backward algorithm)  $\tau \bm{M}=\bm{S}$ (i.e., $\zeta = 0$), %$\overline{\theta}=2(\beta+\hat{\rho})^{-1}\sqrt{\beta} (\sqrt{\beta}-\sqrt{-\hat{\rho}})$,   
        %$\tau_{1,0} = -2\hat{\rho}$,
        %and $\tau_{2,0}=2\beta$.
        \begin{align}
        \overline{\theta} &= \frac{2} {1+\sqrt{-\hat{\rho}/\beta}},\\
        %{\beta+\hat{\rho}}\\
            \tau_{1,0} &= (2-\theta)(\beta+\hat{\rho})-2\hat{\rho}
            -\sqrt{((2-\theta)(\beta+\hat{\rho})-2\hat{\rho})^2+4\hat{\rho}\beta},\\
            \tau_{2,0} &= (2-\theta)(\beta+\hat{\rho})-2\hat{\rho}
            +\sqrt{((2-\theta)(\beta+\hat{\rho})-2\hat{\rho})^2+4\hat{\rho}\beta};
        \end{align}
        \item (comonotone inclusion)  $\bm{C}=\bm{0}$, 
        $\overline{\theta}=2$, $\zeta < \overline{\zeta} = \frac{2-\theta}{2(\eta-4\hat{\rho})}$, 
        %\in ]0,\overline{\zeta}[$,
        $\tau_{1,\zeta}=-2\hat{\rho}(\zeta+\xi)^2/(2-\theta-2(\eta-4\hat{\rho})\zeta)$ and
        $\tau_{2,\zeta} =+\infty$
        ;
        \item \label{coro:overlinezetaiii}  (monotone+cocoercive operator splitting) $\bm{A}$ is monotone,  $\overline{\theta}=2$,
       $\zeta < \overline{\zeta} = \frac{2-\theta}{2\eta}$, 
       %for every $\zeta \in ]0,\overline{\zeta}[$,
        $\tau_{1,\zeta} =0$, and $\tau_{2,\zeta}=2\beta
        (2-\theta-2\eta\zeta)$.        
%        (1-2\zeta)$.

%    \end{itemize}
    \end{enumerate}
    Then, if $\theta\in ]0,\overline{\theta}[$ and
    $\tau \in ]\tau_{1,\zeta},\tau_{2,\zeta}[$, the same conclusions as in Theorem \ref{teo:NLFBMM} hold. In particular, the R-linear convergence property holds whenever the additional assumption in Theorem~\ref{teo:NLFBMM}\ref{teo:NLFBMM2} is satisfied.
    \end{cor}
\begin{proof}
The result follows from Theorem~\ref{teo:NLFBMM} and Proposition~\ref{prop:overlinezeta}\ref{prop:overlinezetai}. Let us use the same notation as 
in Proposition
\ref{prop:overlinezeta}.
\begin{enumerate}
        \item Assume that $\tau \bm{M}=\bm{S}$. Then
            $\zeta = 0$, $\zeta_{\bm{M}}=1/\tau$, $\xi = 1$, and
            \begin{align}
                P(\tau) &= \tau(-\tau^2+2((2-\theta)(\beta+\hat{\rho})-2\hat{\rho})\tau +4\hat{\rho}\beta ).
                %\nonumber\\
                %& = \tau (4\hat{\rho}\beta+4\beta\tau-\tau^2-2\theta\tau(\beta+\hat{\rho}) ).
            \end{align}
The quadratic factor has two distinct nonnegative roots $\tau_{1,0}$ and $\tau_{2,0}$
when $(2-\theta)(\beta+\hat{\rho})-2\hat{\rho}>2\sqrt{-\hat{\rho}\beta}$, that is $\theta < \overline{\theta}$.
Then $P(\tau)$ is positive if $\tau \in ]\tau_{1,0},\tau_{2,0}[$.
\item Since the zero operator is $\beta$-cocoercive for every $\beta>0$, we may let $\beta\to+\infty$ in the expression of $\lambda$. Thus we deduce from \eqref{e:defmutn} that
$\lambda>0$ if and only if
\begin{equation}
\tau c+2\hat{\rho}(\zeta+\xi)^2
%(1-2\zeta)\tau+ 2\hat{\rho} 
%\left((\zeta+\xi)^2+4\tau%\zeta\right)
> 0,
\end{equation}
which
%, under the assumption $\zeta < 1/(2(1-4\hat{\rho}))$, 
yields 
the intervals for $\zeta$ and $\tau$ ensuring convergence. 
\item 
In the case when $\bm{A}$ is monotone, that is $\hat{\rho}=0$, 
we have $P(\tau) = 
\tau^2(2c\beta-\tau)$, 
%\tau^2\left(
%2(1-2\zeta)\beta-\tau\right)$
%we have that $\lambda=1-2\zeta - \frac{\tau}{2\beta}$ and it 
which yields the result.
\end{enumerate}
%\end{itemize}
\end{proof}
    
\begin{rem}\label{rem:lambda}\ 
\begin{enumerate}
\item\label{rem:lambda0} The necessary condition for $\lambda > 0$ in Proposition~\ref{prop:overlinezeta}\ref{prop:overlinezetaii}
defines a subinterval of the one initially considered in Assumption~\ref{assum:2}\ref{assum:24}.
\item\label{rem:lambda1} Suppose that $\bm{S}=\Id$, $\tau \bm{M}=\Id$, $\bu_0=\bm{0}$, and $\theta=1$.
Then Algorithm~\ref{algo:NLFBMM} reduces to the standard unrelaxed FB algorithm. From our standing assumptions, $\tau \bm{A}$ is $(-\hat{\rho}/\tau)$-cohypomonotone on $\mathbb{S}$.
Let us assume that the cohypomonotonicity property extends to $\bm{\H}$ to provide a simple  
interpretations of the conditions in Corollary~\ref{coro:overlinezeta}\ref{coro:overlinezetai}.
The condition $\tau > \tau_{1,0}=-2\hat{\rho}$ guarantees that $J_{\tau \bm{A}}$ is an averaged operator \cite[Corollary~3.10(iii)]{BauschkeMoursiXianfu2021}, whereas $\tau < \tau_{2,0} = 2\beta$ is the classical condition ensuring the averagedness of 
$\Id-\tau \bm{C}$.
\item\label{rem:lambda2} In Corollary~\ref{coro:overlinezeta}\ref{coro:overlinezetaiii}, 
the obtained condition coincides with the assumption made in \cite[Theorem~3.1]{MorinBanertGiselsson2022}.
\item\label{rem:lambda3} It follows from \eqref{e:overlinezeta} that, for fixed $\hat{\rho}<0$, when $\beta\to-\hat{\rho}$, then $\overline{\zeta} \to 0$. 
\item Let $\lambda=\lambda(\tau,\theta,\beta,\hat{\rho},\zeta)$ be defined in \eqref{e:defmutn}. The supremum value of 
$\zeta>0$ for which there exists $\tau \in ]\underline{\tau}_\zeta(\beta,\hat{\rho}),+\infty[$ such  that $\lambda(\tau,\theta,\beta,\hat{\rho},\zeta) > 0$ depends on $\beta$, $\hat{\rho}$, and $\theta$.
With a slight abuse of notation, this supremum will be denoted by $\overline{\zeta}(\beta,\hat{\rho},\theta)$.
%parameter $\overline{\zeta}$ in Proposition~\ref{prop:overlinezeta} depends on $\beta$ and $\hat{\rho}$. 
This value is nondecreasing with respect to $\beta$, 
that is, if $\beta_2>\beta_1$, then $\overline{\zeta}(\beta_2,\hat{\rho},\theta)\geq\overline{\zeta}(\beta_1,\hat{\rho},\theta)$. Indeed, if $\zeta>0$ and $\beta_2>\beta_1$, then
$\underline{\tau}_\zeta(\beta_2,\hat{\rho})< \underline{\tau}_\zeta(\beta_1,\hat{\rho})$
and $\lambda(\tau,\theta,\beta_2,\hat{\rho},\zeta)\geq\lambda(\tau,\theta,\beta_1,\hat{\rho},\zeta)$.
Hence, every $\zeta$ feasible for $\beta_1$ is feasible for $\beta_2$. 
%$\lambda$ defined in \eqref{e:defmutn} is nondecreasing with respect to $\beta$. Now, let $\hat{\rho} \in ]-\infty,0[$, let $\beta_1 \in ]-\hat{\rho},\infty[$, and suppose by contradiction that there exists $\beta_2>\beta_1$ such that $\overline{\zeta}(\beta_2,\hat{\rho},\theta)<\overline{\zeta}(\beta_1,\hat{\rho},\theta)$. This implies there exists $\widetilde{\zeta} \in ]\overline{\zeta}(\beta_2,\hat{\rho},\theta),\overline{\zeta}(\beta_1,\hat{\rho},\theta)[$ such that, for every $\tau \geq -\frac{\hat{\rho}\widetilde\zeta}{\beta_2+\hat\rho}$, $\lambda(\tau,\theta,\beta_2,\hat{\rho},\widetilde{\zeta})\leq 0$. Since $\widetilde{\zeta}<\overline{\zeta}(\beta_1,\hat{\rho},\theta)$, there exists $\widetilde{\tau} > -\frac{\hat{\rho}\widetilde\zeta}{\beta_1+\hat\rho}>-\frac{\hat{\rho}\widetilde\zeta}{\beta_2+\hat\rho}$ such that $\lambda(\widetilde\tau,\theta,\beta_1,\hat{\rho},\widetilde{\zeta})>0$. However, since $\lambda$ is an increasing function of $\beta$, we have $\lambda(\widetilde\tau,\theta,\beta_2,\hat{\rho},\widetilde{\zeta})>\lambda(\widetilde\tau,\theta,\beta_1,\hat{\rho},\widetilde{\zeta})>0$, which leads to a contradiction.

Analogously, it can be proved that $\overline{\zeta}(\beta,\hat{\rho},\theta)$ is a nondecreasing function of $\hat{\rho}$ since $\lambda$ is increasing with respect to $\hat{\rho}$. Indeed, %$\frac{\partial \lambda}{\partial \hat{\rho}}$ is given by
\begin{align*}
    \frac{\partial \lambda}{\partial \hat{\rho}} &= \frac{2}{\tau}  \big((\zeta+\tau\zeta_M)^2+4\tau\zeta\big)-\frac{(\tau+2\hat{\rho})(4\beta\tau+ (2\hat{\rho}-\tau)(\tau+\zeta))}{2(\beta \tau +\hat{\rho}(\tau+\zeta))^2}.
\end{align*}
Moreover, since $\tau\zeta_M \geq 1-\zeta$,
\begin{align*}
    \frac{\partial \lambda}{\partial \hat{\rho}}\Bigg\vert_
   {\hat{\rho} =0}&= \frac{2}{\tau}  \big((\zeta+\tau\zeta_M)^2+4\tau\zeta\big)-\frac{4\beta-\tau-\zeta}{2\beta^2} \\
   &\geq\frac{2}{\tau}  \big(1+4\tau\zeta\big)-\frac{4\beta-\tau-\zeta}{2\beta^2}  
   \\
  &= \frac{(2\beta-\tau)^2+\tau\zeta(16\beta^2+1)}{2\beta^2\tau} >0.
\end{align*}
Besides, \begin{equation*}
    \frac{\partial^2 \lambda}{\partial \hat{\rho}^2} = - \frac{\tau^2 (2\beta - \tau - \zeta)^2}{\left( \beta\tau + \hat{\rho}(\tau + \zeta) \right)^3}\leq 0.
\end{equation*}
Then, $ \frac{\partial \lambda}{\partial \hat{\rho}}$ is nonincreasing as a function of $\hat{\rho}$ and it is positive at $\hat{\rho} =0$. We conclude that $ \frac{\partial \lambda}{\partial \hat{\rho}}$  is positive on the admissible domain $]-\beta\tau/(\tau+\zeta), 0]$ and that $\lambda$ is increasing with respect to $\hat{\rho}$ on this domain. In addition,  $\underline{\tau}(\beta,\hat{\rho})$ decreases as a function of $\hat{\rho}$.

On the other hand, for $\theta \neq 1$, we have 
\begin{equation*}
    \frac{\partial \lambda}{\partial \theta} = \begin{cases}
		-1, &\textnormal{ if } \theta \in [1,2[ \textnormal{ and } (\bm{S}-\tau \bm{M}) \textnormal{ is monotone,}\\
		-1 - 2\textnormal{sgn}(\theta-1)\zeta, &\textnormal{ otherwise.}
		\end{cases}
\end{equation*}
Since $\zeta \in ]0,1/2[$, it follows that $\frac{\partial \lambda}{\partial \theta} < 0$ in all cases. Since $\lambda$ is continuous at $\theta=1$, $\lambda$ is strictly decreasing with respect to $\theta$
and $\overline{\zeta}(\beta, \hat{\rho}, \theta)$ is nonincreasing with respect to $\theta$. Hence, smaller values of $\theta$ lead to less restrictive conditions %on $\beta$ and $\hat{\rho}$ to ensure that $\lambda > 0$ for some $\tau > 0$
for the existence of an admissible step size $\tau$. 
%In particular, due to the continuity of $\lambda$ with respect to $\theta$, if $\lambda|_{\theta = 0} > 0$, it is always possible to choose a sufficiently small $\theta \in ]0, 1[$ such that $\lambda > 0$. Letting $\theta \to 0$, we have $\nu \to 1$ and the bound for $\overline{\zeta}(\beta, \hat{\rho}, \theta)$ in \eqref{e:overlinezeta} reduces to
%\begin{equation}\label{e:overlinezetat=0}
 %       \overline{\zeta} \in \left]0, \min\left\lbrace \frac{1}{2(1-2\hat{\rho})}, \frac{(\beta^2-\hat{\rho}^2)}
%    {2(1-2\hat{\rho})(\beta+\hat{\rho})^2-\hat{\rho}}
    %    \frac{\beta^2-\hat{\rho}^2}
    %{2(1-4\hat{\rho})(\beta+\hat{\rho})^2-2\hat{\rho}}
%    \right\rbrace
%    \right[
%    \end{equation}
Figure~\ref{fig:z_vs_rho_and_zeta} shows examples of $\overline{\zeta}(\beta,\hat{\rho},\theta)$ as a function of $\hat{\rho}$ for selected values of $\beta$ and $\theta$.
\end{enumerate}
\end{rem}
\begin{figure}[htbp]
    \centering
    % Fila 1
    \begin{subfigure}[b]{0.4\textwidth}
        \centering
        \includegraphics[width=\linewidth]{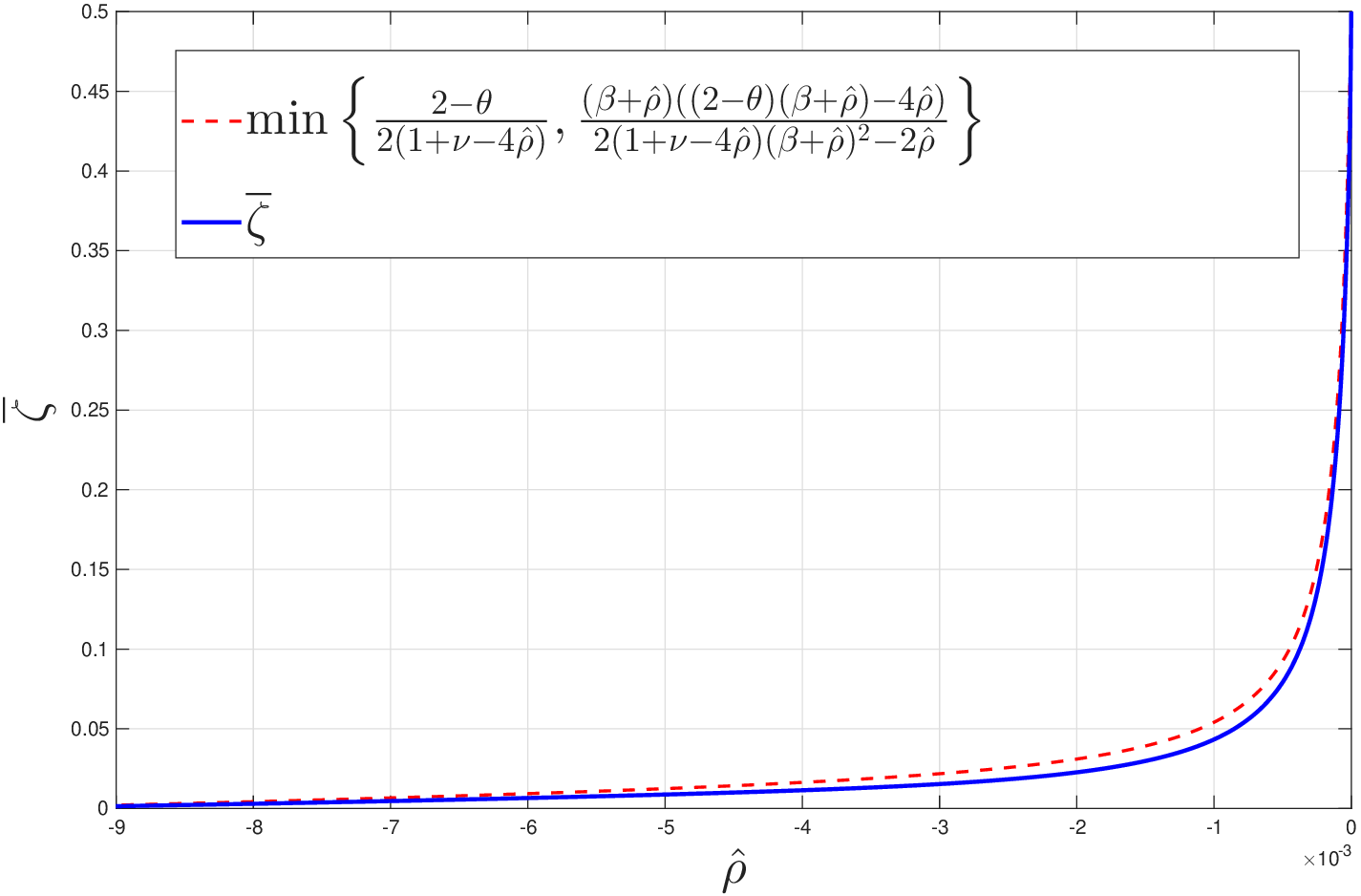}
        \caption{\scriptsize $\beta = 0.01$ and $\theta = 1$}
        \label{fig:beta_001}
    \end{subfigure}
    \hfill
    \begin{subfigure}[b]{0.4\textwidth}
        \centering
        \includegraphics[width=\linewidth]{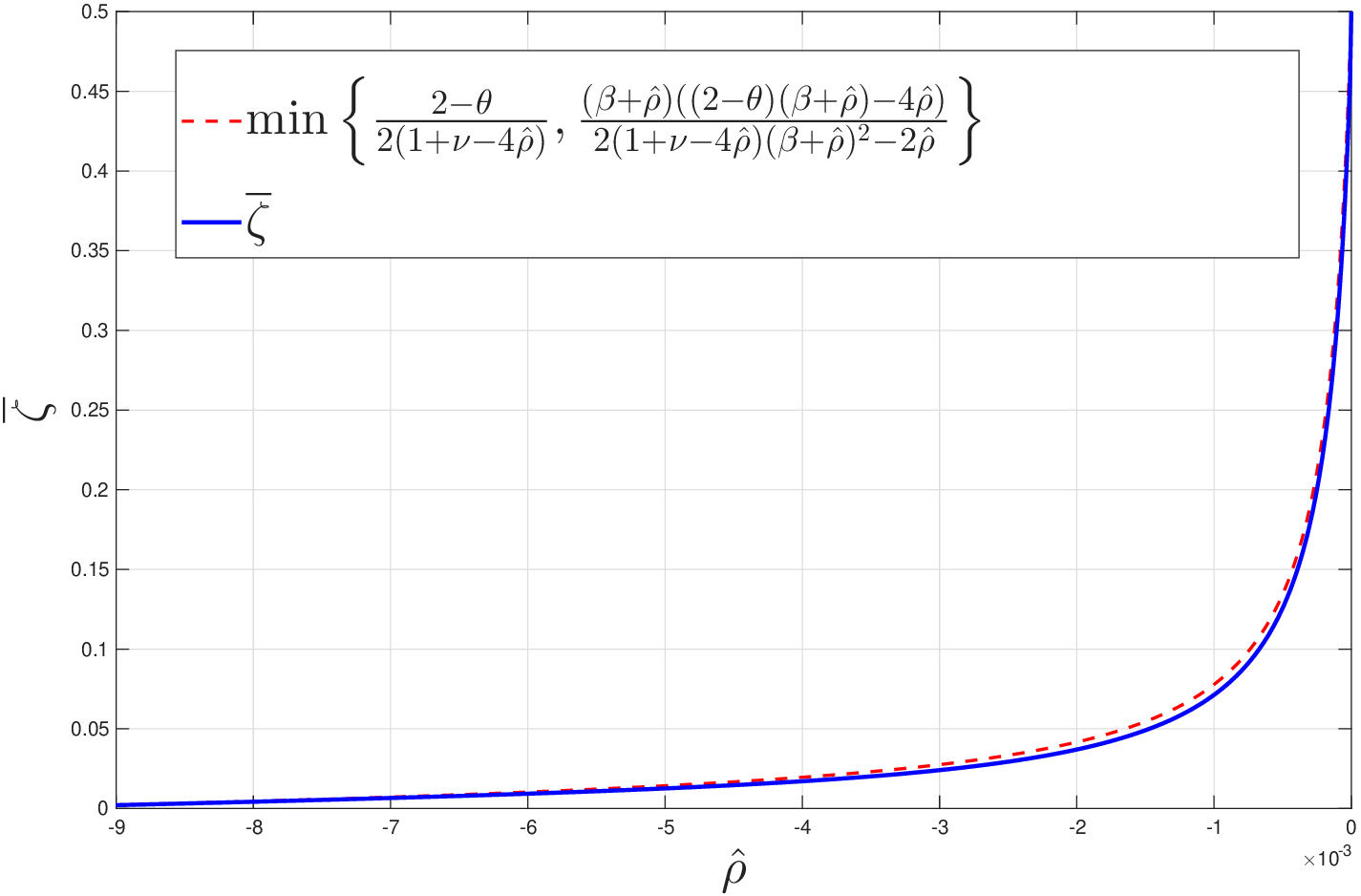}
        \caption{\scriptsize $\beta = 0.01$ and $\theta = 1/4$}
        \label{fig:beta_001_t}
    \end{subfigure}
    %\vspace{10pt} % Espacio vertical entre filas

    % Fila 2
    \begin{subfigure}[b]{0.4\textwidth}
        \centering
        \includegraphics[width=\linewidth]{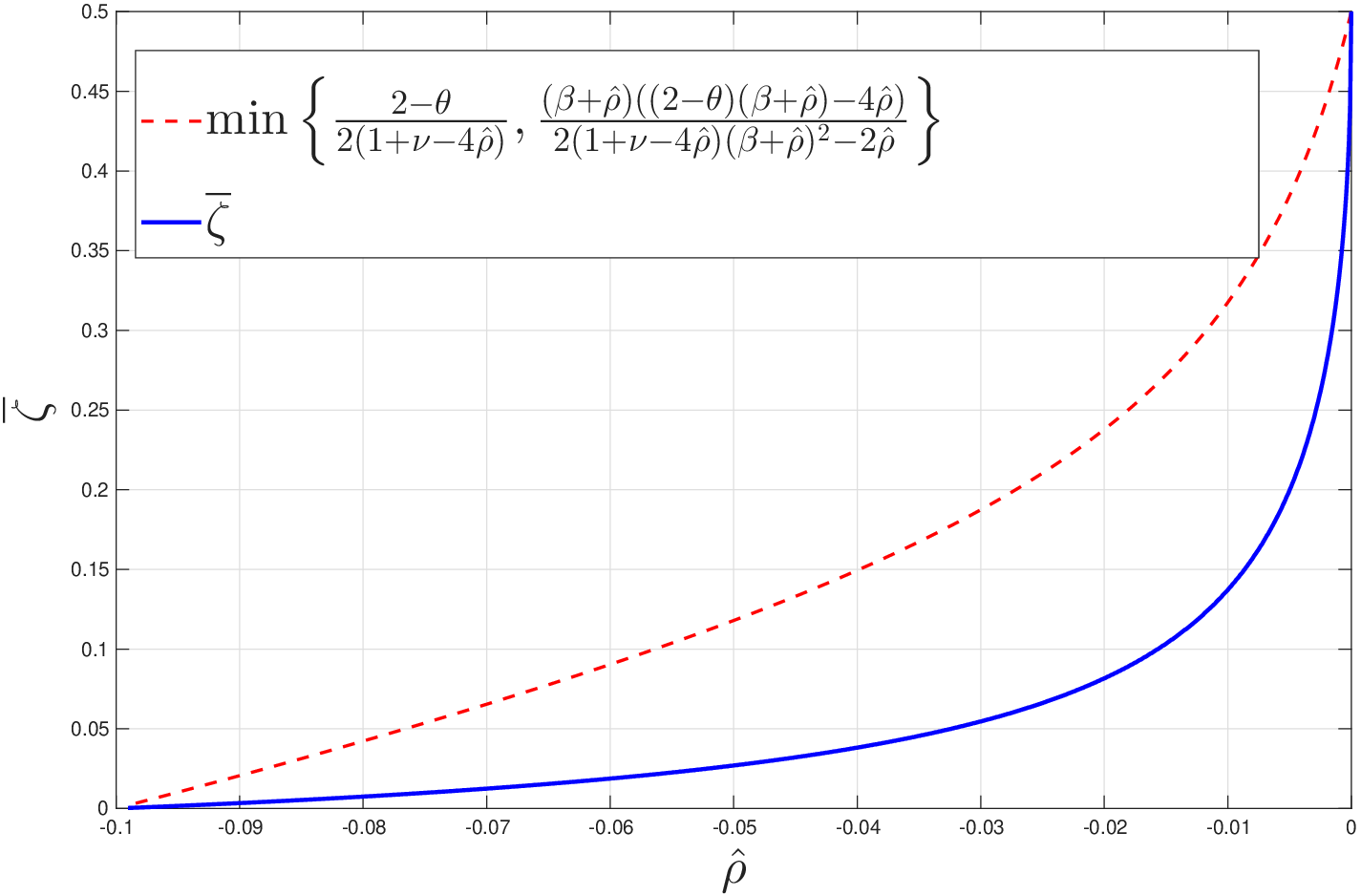}
        \caption{\scriptsize $\beta = 0.1$ and $\theta = 1$}
        \label{fig:beta_01}
    \end{subfigure}
    \hfill
    \begin{subfigure}[b]{0.4\textwidth}
        \centering
        \includegraphics[width=\linewidth]{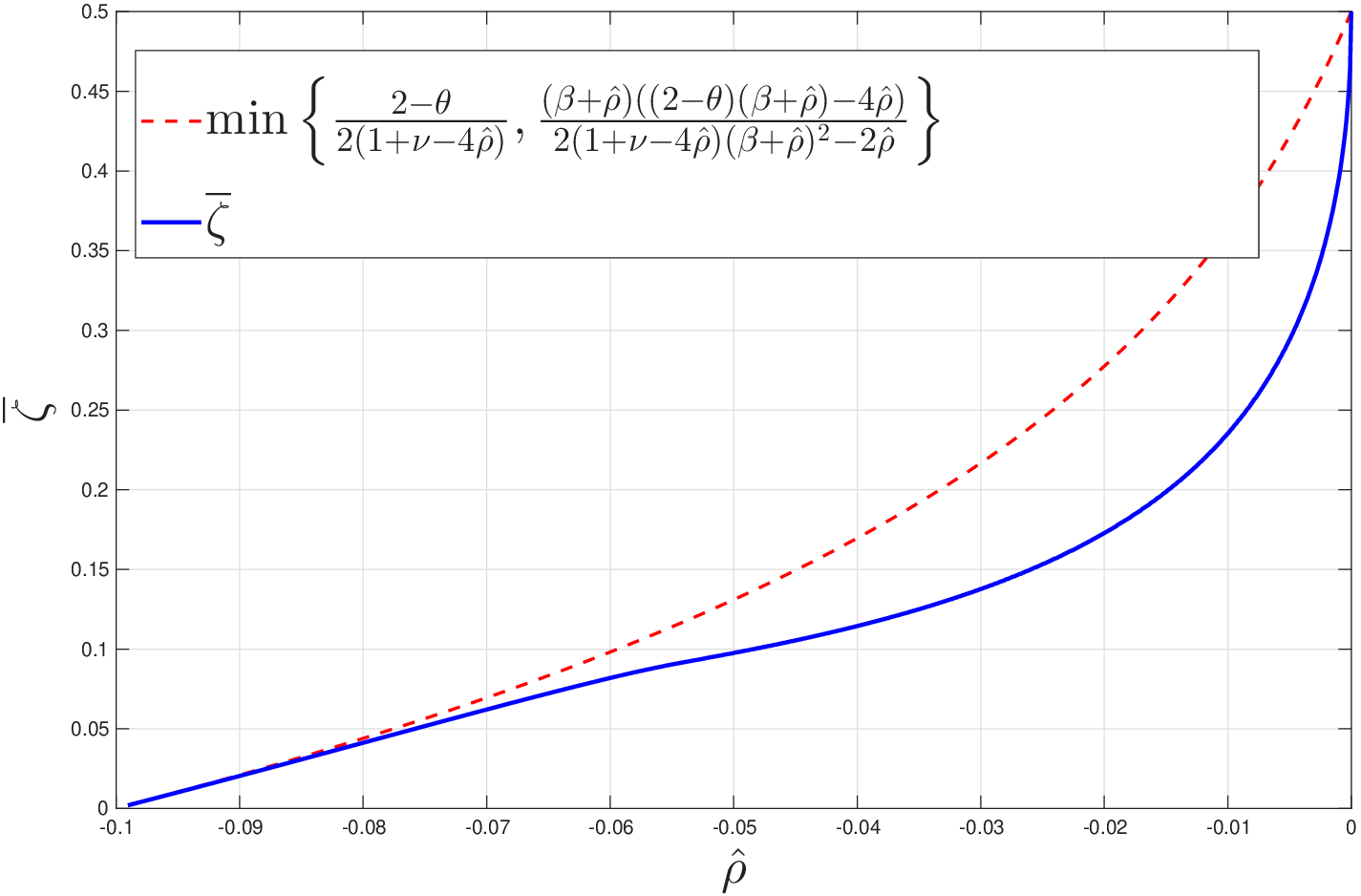}
        \caption{\scriptsize $\beta = 0.1$ and $\theta = 1/4$}
        \label{fig:beta_01_t}
    \end{subfigure}
      % Fila 3
    \begin{subfigure}[b]{0.4\textwidth}
        \centering
        \includegraphics[width=\linewidth]{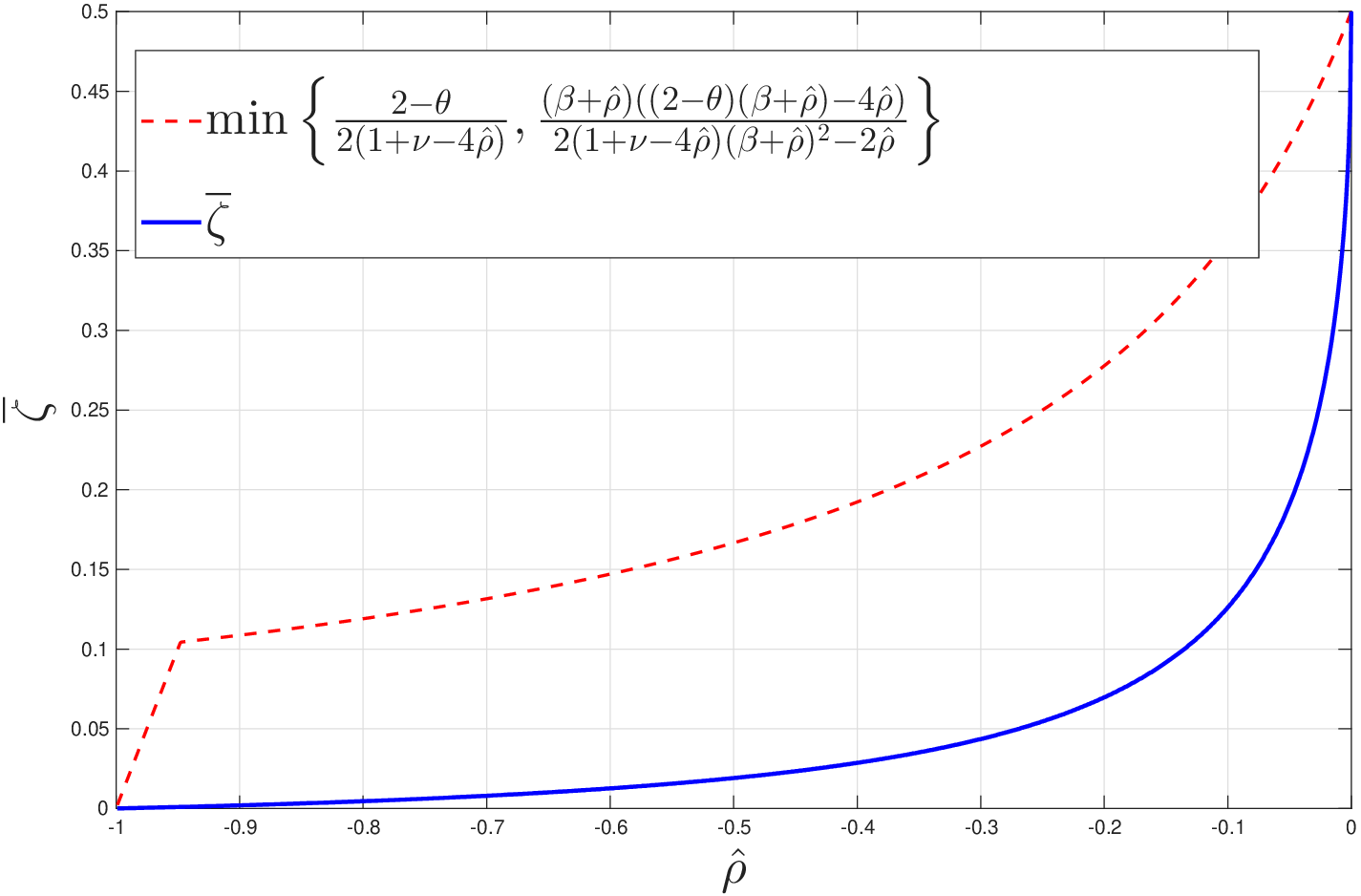}
        \caption{\scriptsize $\beta = 1$ and $\theta = 1$}
        \label{fig:beta_1}
    \end{subfigure}
    \hfill
    \begin{subfigure}[b]{0.4\textwidth}
        \centering
        \includegraphics[width=\linewidth]{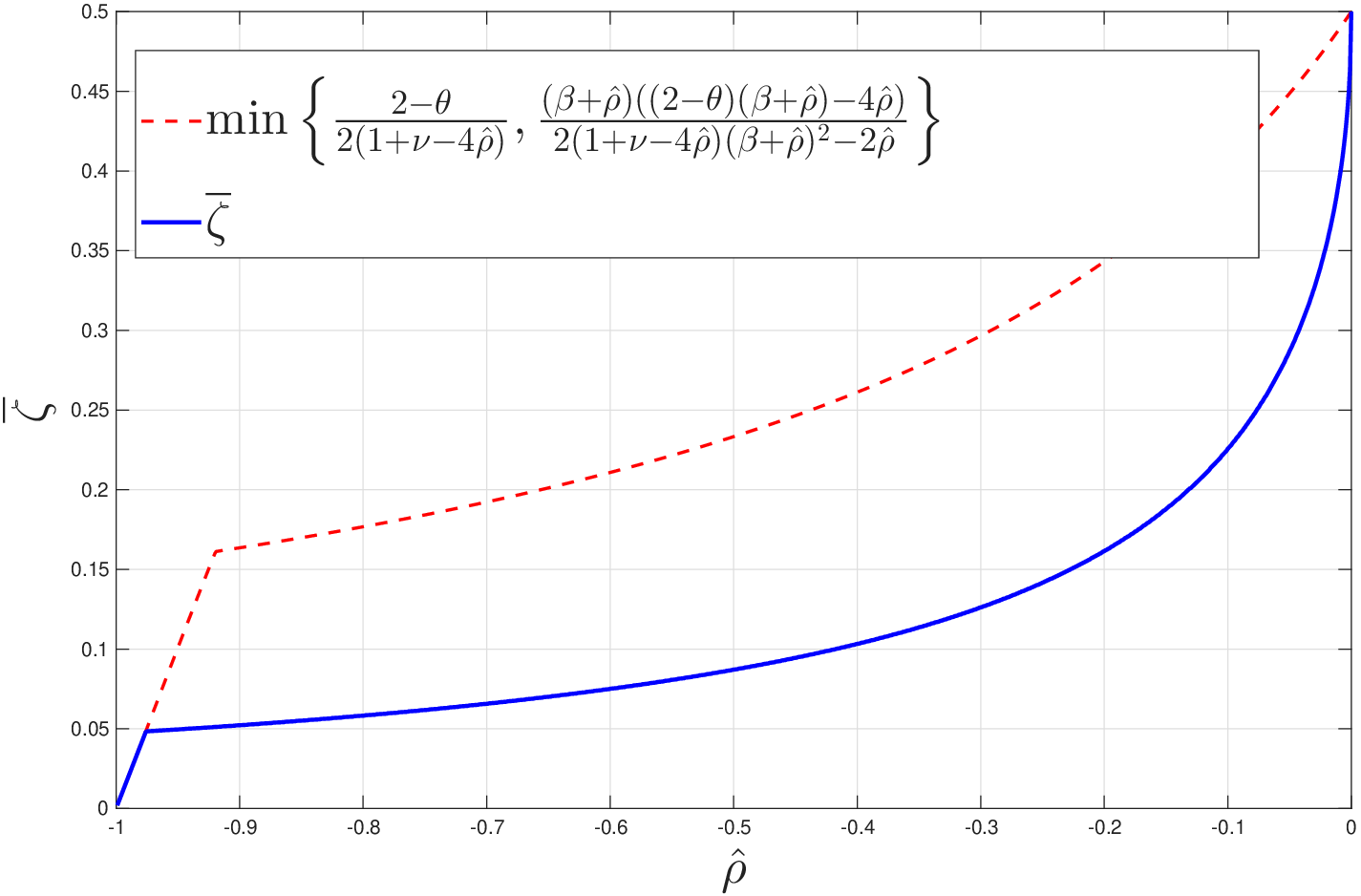}
        \caption{\scriptsize $\beta = 1$ and $\theta = 1/4$}
        \label{fig:beta_1_t}
    \end{subfigure}

   % Fila 4
      % Fila 1
    \begin{subfigure}[b]{0.4\textwidth}
        \centering
        \includegraphics[width=\linewidth]{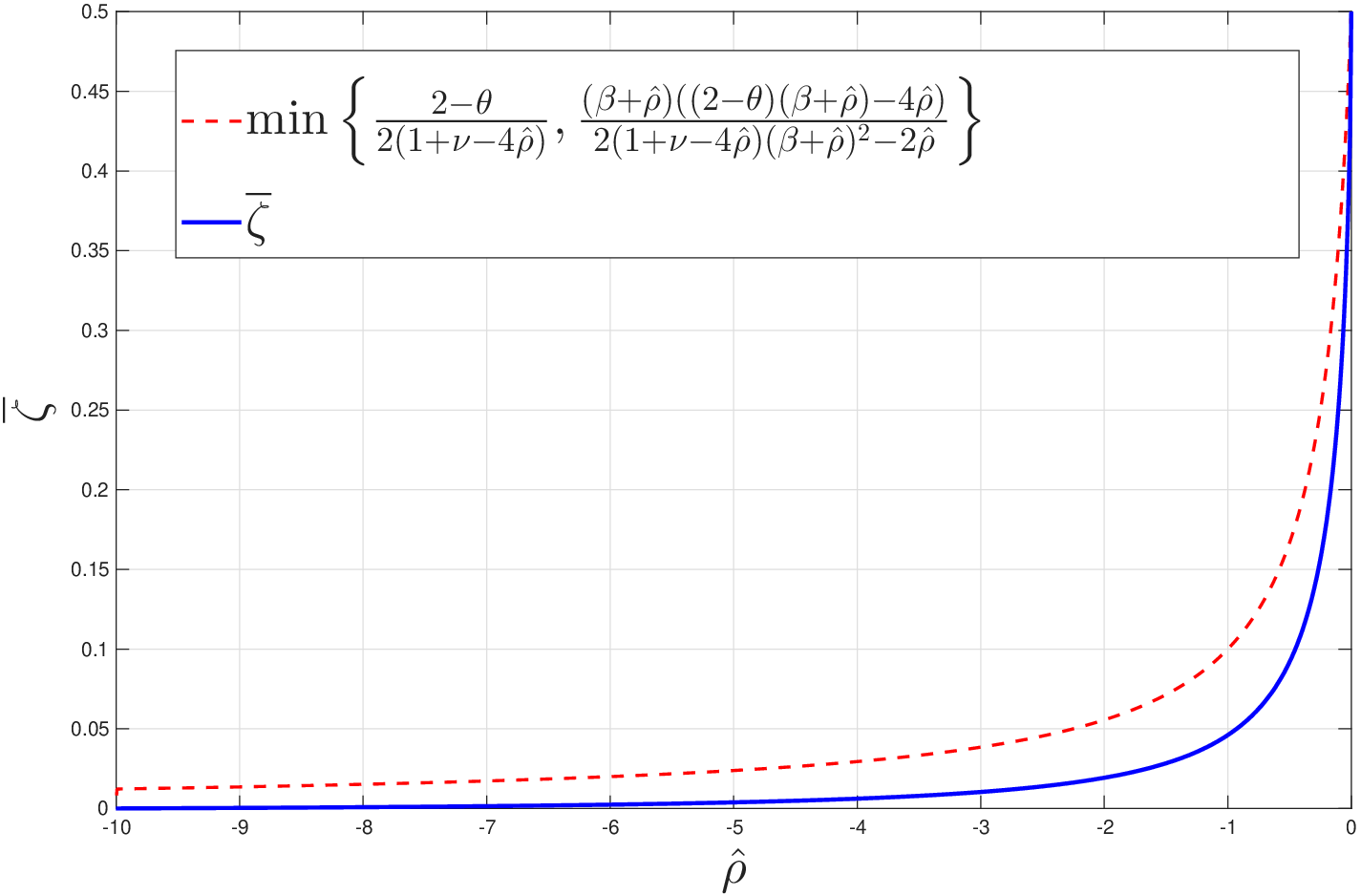}
        \caption{\scriptsize $\beta = 10$ and $\theta = 1$}
        \label{fig:beta_10}
    \end{subfigure}
    \hfill
    \begin{subfigure}[b]{0.4\textwidth}
        \centering
        \includegraphics[width=\linewidth]{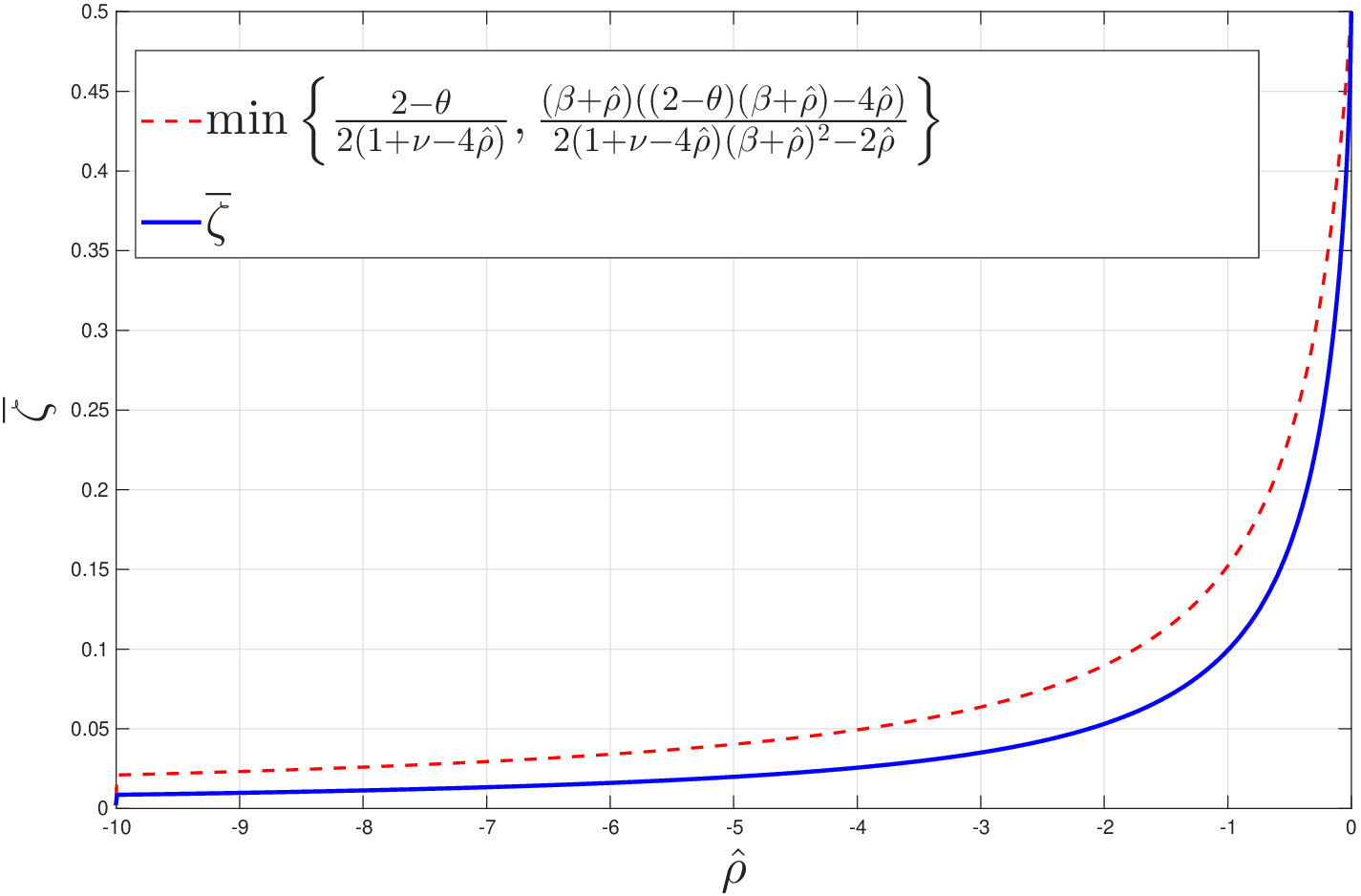}
        \caption{\scriptsize $\beta = 10$ and $\theta = 1/4$}
        \label{fig:beta_10_t}
    \end{subfigure}
    \caption{$\overline{\zeta}(\beta,\cdot,\theta)$ for $\beta\in \{0.01,0.1,1,10\}$ and $\theta \in \{1,1/2\}$. In each case $\xi=1+\zeta$.}
    \label{fig:z_vs_rho_and_zeta}
\end{figure}

 	\section{Solution of Problem~\ref{pro:main}}\label{sec:MR}
	
	In this section, we present our main algorithm for solving Problem~\ref{pro:main}. First, we introduce some notation and results that are useful to establish the convergence of the proposed method.

    \subsection{Preliminary results}
	\label{se:FHRB}
    We first introduce some additional notation and assumptions. 
	\begin{notation}
    \label{no:Sec4}
		In the context of Problem~\ref{pro:main}, 
    consider the following operators
        defined on the product space $\bm{\H} = \H\times \G$:
		\begin{equation}\label{def:operators}
			\begin{aligned}
				&\bm{B} \colon \bm{\H} \to 2^{\bm{\H}} \colon (z,v) \mapsto (A+D)z \times B^{-1}v\\
				&\bm{L} \colon \bm{\H} \to \bm{\H} \colon (z,v) \mapsto (L^*v,-Lz)\\
				&\bm{A} \colon \bm{\H} \to 2^{\bm{\H}} \colon (z,v) \mapsto (\bm{B}+\bm{L})(z,v)\\
				&\bm{C} \colon \bm{\H} \to {\bm{\H}} \colon (z,v) \mapsto (Cz,0)
			\end{aligned}
		\end{equation}
		and define 
		\begin{equation}\label{def:setS}
			\pmb{\mathbb{S}}=\menge{(\x^*,-\bm{C}\x^*)}{\x^* \in \zer(\bm{A}+\bm{C})}.
		\end{equation}
	\end{notation}
	Note that, based on the above definitions, $\zer(\bm{A}+\bm{C}) = \bm{Z}$ .
	\begin{asume}\label{assume:1}
		In the context of Problem~\ref{pro:main}, let $(\mu_A,\mu_B,\rho_A,\rho_B)\in \R^4$ and $(\nu_A,\nu_B)\in \RP^2$. Suppose that
		\begin{enumerate}
			%\item \label{eq:neqassume1} {\color{red} $D_K\neq 0$},
			\item\label{eq:assume12} $\rho_A+\rho_B\geq 0$, $\mu_A+\mu_B\geq0$, 
            $(\rho_A +\rho_B =0\Rightarrow \rho_A=\rho_B=0)$, and $(\mu_A +\mu_B =0\Rightarrow \mu_A=\mu_B=0)$;
			\item\label{eq:monAassume1}  $A+D$ is $(\nu_A\id+\mu_AL^*L,\rho_A)$-semimonotone on  $\pmb{\mathbb{S}}_1=\menge{(x^*,-Cx^*-L^*u^*)}{(x^*,u^*) \in \bm{Z}}$;
			\item\label{eq:monBassume1b} $B$ is $(\mu_B,\nu_B\id+\rho_BLL^*)$-semimonotone on $\pmb{\mathbb{S}}_2=\menge{(Lx^*,u^*)}{(x^*,u^*) \in \bm{Z}}$;
			\item\label{eq:monAassume2}   there exist nonempty sets $\Omega_A \subset \RPP$ and $\Omega_B \subset \RPP$ such that, for every $(\tau,\sigma)\in  \Omega_A\times\Omega_B$, $J_{\tau A}$ and $J_{\sigma B^{-1}}$ are single-valued and have full domain;
			\item\label{assume:1v} $\gra(\bm{A}+\bm{C})$ is sequentially weak-strong closed.
		\end{enumerate} 
	\end{asume}	
    The following proposition showcases some
    situations where Assumption~\ref{assume:1}\ref{eq:monAassume1} holds, whenever the resulting value of $\nu_A$ is nonnegative.
    \begin{prop}
    \label{p:ex42(ii)}
        Let $(\nu,\mu,\rho,\widetilde{\rho})\in \mathbb{R}^4$, let $T$ be a bounded linear operator from $\H$
        to a real Hilbert space $\K$,
        let $K\neq 0$ be a linear bounded operator from $\K$ to $\H$,
        and let $\widetilde{D}\colon \K \to \K$ be a 
        $(\vartheta_{\widetilde{D}})^{-1}$-cocoercive operator with 
        $\vartheta_{\widetilde{D}}\in ]0,+\infty[$.
        Assume that $A$ is $(\nu\id+\mu L^*L,\rho)$-semimonotone. Then, $A+D$ is $(\nu_A\id+\mu L^*L,\rho_A)$-semimonotone if
        \begin{enumerate}
            \item $D=K\circ \widetilde{D}\circ T$, 
            %$\nu\geq \vartheta_{T^*-K}$, 
            $1+\vartheta_{\widetilde{D}}\rho\|K\|^2>0$,
            and $(\nu_A,\rho_A)=\big(\nu-\vartheta_{T^*-K},\rho/(1+\vartheta_{\widetilde{D}}\rho\|K\|^2)\big)$, where
            $\vartheta_{T^*-K}$ is the Lipschitz modulus of $(T^*-K)\circ \widetilde{D}\circ T$.
            \item     \label{p:ex42(ii)ii}
 $D$ is a bounded linear operator, $D-\widetilde{\rho} D^*D$ is  a monotone operator 
            with $\widetilde{\rho}+\rho >0$, and
            $(\nu_A,\rho_A) = \big(\nu,\rho\widetilde{\rho}/(\rho+\widetilde{\rho})\big)$.
            \item $D=0$ and $(\nu_A,\rho_A) = (\nu,\rho)$.
        \end{enumerate}
    \end{prop}
    \begin{proof}\ 
    \begin{enumerate}
        	\item\label{rem:A+D_Kmon3} 
            For every  $(x,y) \in \H^2$, we have
			\begin{align*}
				\scal{x-y}{K\widetilde{D}Tx-K\widetilde{D}Ty} %&=\scal{x-y}{T^*DTx-T^*DTy}+\scal{x-y}{(K-T^*)(DTx-DTy)}\\
				&= \scal{Tx-Ty}{\widetilde{D}Tx-\widetilde{D}Ty}+\scal{x-y}{(K-T^*)(\widetilde{D}Tx-\widetilde{D}Ty)}\\
				&\geq \vartheta_{\widetilde{D}}^{-1}\|\widetilde{D}Tx-\widetilde{D}Ty\|^2- {\vartheta}_{{K-T^*}}\|x-y\|^2\\
				&\geq \frac{1}{\vartheta_{\widetilde{D}}\|K\|^2}\|K\widetilde{D}Tx-K\widetilde{D}Ty\|^2- {\vartheta}_{{K-T^*}}\|x-y\|^2.
			\end{align*}
            %Assume now that $A$ is $(\nu,\rho)$-semimonotone with $(\nu,\rho) \in \R^2$
			In view of Lemma~\ref{le:sumsemimon3}, 
            if $1+\vartheta_{\widetilde{D}}\rho\|K\|^2 >0$, then $A+D$ is $\big((\nu-\vartheta_{{T^*-K}})\id +\mu L^*L,\rho/(1+\vartheta_{\widetilde{D}}\rho\|K\|^2)\big)$-semimonotone.
            
            \item\label{rem:A+D_Kmon5}
			If $D$ is linear, $D$ is $\widetilde{\rho}$-comonotone for $\widetilde{\rho} \in \R$ if and only if
			\begin{align}
				(\forall x \in \H) \quad \scal{x}{Dx}\geq \widetilde{\rho} \| Dx\|^2,
			\end{align}
            that is 
 			\begin{align}
				(\forall x \in \H) \quad  \scal{x}{Dx}-\widetilde{\rho}  \scal{x}{D^*Dx}\geq0.
			\end{align}           
			Hence, $D$ is $\widetilde{\rho}$-comonotone if and only if $D-\widetilde{\rho} D^*D$ is monotone. One can thus apply
Lemma~\ref{le:sumsemimon3} to claim that $A+D$ is
$(\nu\id+\mu L^*L,\rho\widetilde{\rho}/(\rho+\widetilde{\rho}))$-semimonotone.
\item The last case is straightforward.
    \end{enumerate}
    \end{proof}
    
	The following remark provides more insight on  Assumptions~\ref{assume:1}\ref{eq:assume12}-\ref{eq:monBassume1b}
    when semimonotonicity holds globally.

    	\begin{rem}\label{rem:A+D_Kmon}\ 
        \begin{enumerate}    
            \item\label{rem:A+D_Kmon0} If $\rho_A=\rho_B=0$
            and $\mu_A+\mu_B>0$, then
            $A+D+L^* \circ B\circ L$
            is monotone, which means 
            that \eqref{eq:primalinclu}
            is a monotone inclusion problem. In turn, if $\mu_A=\mu_B=0$ and $\rho_A+\rho_B>0$, then $B^{-1}-L\circ (A+D)^{-1}\circ (-L^*)$ is monotone, which implies that \eqref{eq:dualinclu}
            is a monotone inclusion problem when $C=0$. However, when $A+D$ or $B$ are not monotone, we will show that
            $\bm{A}$ is cohypomonotone.
			\item\label{rem:A+D_Kmoni} If $A+D$ is $\mu_A L^*L$-monotone
            with $\mu_A\in~]-\infty,0]$, then it
            is $\mu_A \|L\|^2$-monotone.
			Indeed, let $((z,v),(x,u)) \in (\gra (A+D))^2$. We have
			\begin{align*}
				\scal{z-x}{v-u} \geq \mu_A\scal{z-x}{L^*L(z-x)} = \mu_A\|L(z-x)\|^2\geq \mu_A \|L\|^2 \|z-x\|^2.
			\end{align*}
            %\textcolor{blue}{
            Similarly, if $B$ is $\rho_B LL^*$-comonotone for $\rho_B\in ]-\infty,0]$, then $B$ is $\rho_B \|L\|^2$-comonotone.
            			\item\label{rem:A+D_Kmon4} Assume that $L\neq 0$. If $A+D$ is $\nu_A$-strongly monotone for $\nu_A \in \RPP$, then it is $\mu_A L^*L$-monotone with $\mu_A= \nu_A/\|L\|^2$.
            %            (\tilde{\nu} L^*L/\|L\|^2)$-monotone. 
			Indeed, for every $((z,v),(x,u)) \in (\gra (A+D))^2$,
			\begin{align*}
				\scal{z-x}{v-u} \geq \nu_A\|z-x\|^2
				\geq \frac{\nu_A}{\|L\|^2}\|L(z-x)\|^2
				=\frac{\nu_A}{\|L\|^2}\|z-x\|^2_{L^*L}.
			\end{align*}
            %\textcolor{blue}{
            Similarly, if $B$ is $\nu_B$-cocoercive for $\nu_B>0$, $B$ is $\nu_B\|L\|^{-2}LL^*$-comonotone.
		\end{enumerate}
	\end{rem}

	The next proposition provides conditions going beyond maximal monotonicity, ensuring that Assumption~\ref{assume:1}\ref{eq:monAassume2} is fulfilled.

    \begin{prop}\label{p:Omegacond}
        Let $(\nu,\mu,\rho,\nu_B,\mu_B,\rho_B)\in \mathbb{R}^6$
        and let $L\in \mathcal{B}(\H,\G)$.
        Assume that $A$ is $(\nu \id + \mu L^*L,\rho)$-semimonotone and
        $B$ is $(\mu_B,\nu_B\id+\rho_B LL^*)$-semimonotone.
        Let $\widetilde{\Omega}_A\subset\RPP$, $\widetilde{\Omega}_B\subset\RPP$,
        and
        \begin{align}
        \Omega_A&=\menge{\tau\in\widetilde{\Omega}_A}{\ran(\id+\tau A)=\H},\\
        \Omega_B&=\menge{\sigma\in\widetilde{\Omega}_B}{\ran(\id+\sigma B^{-1})=\G}.
        \end{align}
        Then Assumption~\ref{assume:1}\ref{eq:monAassume2} holds if $\Omega_A\neq \varnothing$, $\Omega_B\neq \varnothing$, one
        of the following two conditions holds:
        \begin{enumerate}
             \item \label{p:Omegacondi} $(\nu,\mu)\in [0,+\infty[^2$
            and $\widetilde{\Omega}_A = ]\max\{0,-\rho\},+\infty[$,
             \item \label{p:Omegacondii} $\rho\in [0,+\infty[$ and %$\mu \in ]-\infty,0]$, and 
             \begin{equation}
                 \widetilde{\Omega}_A = \begin{cases}
                 ]0,+\infty[, &\mbox{if $\nu+\min\{\mu,0\}\|L\|^2\geq 0$,}\\
                 ]0,-1/(\nu+\min\{\mu,0\}\|L\|^2)[ & \mbox{otherwise,}
                 \end{cases}
             \end{equation}
        \end{enumerate}
        and one
        of the following two conditions holds:
        \begin{enumerate}
        \setcounter{enumi}{2}
            \item \label{p:Omegacondiii} $(\nu_B,\rho_B)\in [0,+\infty[^2$
            and $\widetilde{\Omega}_B= ]\max\{0,-\mu_B\},+\infty[$,
            \item \label{p:Omegacondiv} $\mu_B\in [0,+\infty[$
            and %$\rho_B\in ]-\infty,0]$, 
            \begin{equation}
        \label{e:Omegabiv}\widetilde{\Omega}_B=\begin{cases}
            ]0,+\infty[, &\mbox{if 
            $\nu_B+\min\{\rho_B,0\}\|L\|^2\geq 0$,}\\
            ]0,-1/(\nu_B+\min\{\rho_B,0\}\|L\|^2)[, &
            \mbox{otherwise.}
            \end{cases}
            \end{equation}
        \end{enumerate}       
            In addition, assume that $A$ is maximally $(\nu \id + \mu L^*L,\rho)$-semimonotone and
        $B$ is maximally $(\mu_B,\nu_B\id+\rho_B LL^*)$-semimonotone. If $\nu=\mu=0$ or $\rho=0$, then 
            $\Omega_A=\widetilde{\Omega}_A$ and, if $\mu_B=0$ or $\nu_B=\rho_B=0$,
            then $\Omega_B=\widetilde{\Omega}_B$.
    \end{prop}
    \begin{proof}
        If $A$ is 
            $(\nu \id + \mu L^*L,\rho)$-semimonotone with $(\nu,\mu)\in [0,+\infty[^2$,
            then it is $\rho$-comonotone. If $\tau> -\rho$,             
            %If $A$ is $\nu$-monotone and $\tau \nu>-1$, 
            then $J_{\tau A}$ is single-valued 
            on $\ran(\id+\tau A)$ \cite[Proposition~2.12]{BauschkeMoursiXianfu2021}.
            %and has full domain, as already noted in Section~\ref{se:notation}. 
            %Hence, we can set $\Omega_A = ]0,+\infty[$
            %if $\nu\geq 0$, and $\Omega_A=~]0,-1/\nu[$ otherwise. Similarly, if 
            %$A$ is $\rho$-comonotone with $\rho\in \R$, then $\tau A$ is $(\rho/\tau)$-comonotone,
            Let us set $\widetilde{\Omega}_A =\, ]0,+\infty[$ if $\rho \geq 0$,
            and $\widetilde{\Omega}_A =\, ]-\rho,+\infty[$
            otherwise. For every $\tau\in\Omega_A$, $\ran(\id+\tau A)=\H$.
            Therefore, $J_{\tau A}$ is single-valued and has full domain, for every $\tau\in\Omega_A$.
            
            On the other hand, if $A$ is 
            $(\nu \id + \mu L^*L,\rho)$-semimonotone with $\rho\in [0,+\infty[$ %and $\mu \in ]-\infty,0]$,
            it follows from Remark~\ref{rem:A+D_Kmon}\ref{rem:A+D_Kmoni} that
            it is $(\nu+\min\{\mu,0\}\|L\|^2)$-monotone.
            Then $J_{\tau A}$ is single-valued on its domain
            when $\tau \in \widetilde{\Omega}_A$ with $\widetilde{\Omega}_A= ]0,+\infty[$ if $\nu+\min\{\mu,0\}\|L\|^2\geq 0$, and
            $\widetilde{\Omega}_A =\, ]0,-1/(\nu+\min\{\mu,0\}\|L\|^2)[$
            otherwise.
            
            Similarly, if $B$ is $(\mu_B,\nu_B\id+\rho_B LL^*)$-semimonotone
            with $(\nu_B,\rho_B)\in [0,+\infty[^2$,
            $B$ is $\mu_B$-monotone, that is $B^{-1}$ is $\mu_B$-comonotone.
            %(resp. $\nu_B$-comonotone), 
            Then, $J_{\sigma B^{-1}}$ is single-valued 
            on $\ran(\id+\sigma B^{-1})$,
            %.
            %with full domain 
            if $\sigma\in \widetilde{\Omega}_B$, where $\widetilde{\Omega}_B=\, ]0,+\infty[$ if $\mu_B\geq 0$ 
            %(resp. $\nu_B \geq 0)$, 
            and $\widetilde{\Omega}_B=\,]-\mu_B,+\infty[$ otherwise.
            
            On the other hand, if $B$ is $(\mu_B,\nu_B\id+\rho_B LL^*)$-semimonotone
            with $\mu_B\in [0,+\infty[$
            %and $\rho_B\in ~]-\infty,0]$, 
            then it follows from Remark~\ref{rem:A+D_Kmon}\ref{rem:A+D_Kmoni}  that
            $B^{-1}$ is $(\nu_B+\min\{\rho_B,0\}\|L\|^2)$-monotone. Thus, we
            can set $\widetilde{\Omega}_B$
            as indicated in \eqref{e:Omegabiv}.

    Suppose now that $A$ and $B$ are maximally semimonotone with respect to the moduli stated in the proposition. We prove the additional range conclusions by considering separately the cases $\nu=\mu=0$, $\rho=0$, $\nu_B=\rho_B=0$, and $\mu_B=0$.
    
            If $\nu=\mu=0$,
            then $\tau A$ is maximally $\rho/\tau$-comonotone and, by 
            \cite[Corollary~2.11]{BauschkeMoursiXianfu2021}, $\ran(\id+\tau A)=\H$, for every $\tau\in\, ]\max\{0,-\rho\},+\infty[$.

            If $\rho=0$, then $A$ is $\nu \id+\mu L^*L$-monotone, that is 
            $\widetilde{A}=A-\nu\id-\mu L^*L$ is monotone.
            The graph of $A$ is mapped to the graph of $\widetilde{A}$ by
            $(x,u)\mapsto (x,u-\nu x-\mu L^*L x)$. Since this transformation is bijective, it preserves maximality. For every
            $\tau\in \RPP$, 
             \begin{equation}
            \id+\tau A = \tau \widetilde{A}+Q_A, \quad 
            Q_A = (1+\tau \nu)\id+\tau\mu L^*L.
            \end{equation}
            We have $Q_A\succeq \big(1+\tau(\nu+\min\{\mu,0\}\|L\|^2)\big)\id$. 
            If $\nu+\min\{\mu,0\}\|L\|^2\ge 0$, then
            $Q_A$ is  strongly monotone and maximally monotone, for every $\tau\in \RPP$.
            Maximal strong monotonicity of 
            $Q_A$ also holds if
            $\nu+\min\{\mu,0\}\|L\|^2< 0$ and $1+\tau(\nu+\min\{\mu,0\}\|L\|^2)> 0$, that is $\tau \in\, ]0,-1/(\nu+\min\{\mu,0\}\|L\|^2[$.
            Under these conditions, since $\widetilde{A}$ is maximally monotone and $Q_A$ is  strongly monotone and maximally monotone with full domain, $\tau \widetilde{A}+Q_A$ is strongly monotone 
            and maximally monotone \cite[Corollary~25.5]{bauschkebook2017}.
            Hence its range is $\H$  \cite[Corollary~23.37]{bauschkebook2017}.

        The $(\mu_B,\nu_B\id+\rho_B LL^*)$-semimonotonicity of $B$
         is equivalent to
         the monotonicity of 
         $\widetilde{B}= B^{-1}-\nu_B\id-\rho_B LL^*$, if $\mu_B=0$, and to the $\mu_B$-comonotonicity of $B^{-1}$, if $\nu_B=\rho_B=0$. By arguments analogous to those used above,
         if ($\nu_B=\rho_B=0$ or 
         $\mu_B=0$) and $\sigma \in \widetilde{\Omega}_B$,
         then $\ran(\id+\sigma B^{-1})=\G$.
    \end{proof}
    
    We now state a number of results which will be useful in the convergence proofs of the next section.
    In particular, the next result highlights cohypomonotonicity properties of $\bm{A}$. \begin{prop}\label{prop:ursemimongen}
		In the context of
        Notation~\ref{no:Sec4} and Assumption~\ref{assume:1}, let $(\tau,\sigma)
		\in \RPP^2$, 
        %set $\bm{\H}=\H\times\G$, 
        %consider 
        %the operators defined in \eqref{def:operators}, and 
        let $\bm{S}\in \mathcal{B}(\bm{\H},\bm{\H})$ be a self adjoint and strongly monotone, and suppose that
		\begin{equation}\label{eq:ineqS}
        (\forall (z,v)\in \bm{\H})\quad
			\underline{\chi}\|(z,v)\|_{\bm{S}}^2 \leq \|z\|^2+\frac{\tau}{\sigma}\|v\|^2\quad \text{ and }\quad \|z\|^2+\frac{\sigma}{\tau}\|v\|^2\leq\overline{\chi}\|(z,v)\|_{\bm{S}^{-1}}^2,
		\end{equation}
		where $(\underline{\chi},\overline{\chi})\in \RPP^2$.
		In addition, define the parameters
		\begin{equation}\label{def:murho1}
			\hat{\mu}=
            %\underline{\chi}
            \min\left\{\nu_A,\frac{\sigma\nu_B}{\tau}\right\}, \   \hat{\rho}=\begin{cases}
				%\overline{\chi}
                \min\left\{0,\rho_{A,B},\frac{\tau}{\sigma}\mu_{A,B}\right\}, & \textnormal{ if } \min\{\rho_A+\rho_B,\mu_A+\mu_B\}>0\\
				%\overline{\chi}
                \min\left\{0,\frac{\tau}{\sigma}\mu_{A,B}\right\}, & \textnormal{ if } \rho_A=\rho_B=0 \textnormal{ and } \mu_B+\mu_A>0\\
				%\overline{\chi}
                \min\left\{0,\rho_{A,B}\right\}, & \textnormal{ if } \mu_A=\mu_B=0 \textnormal{ and }\rho_A+\rho_B>0
				\\	0, & \textnormal{ if } \mu_A=\mu_B=\rho_B=\rho_A=0
			\end{cases}
		\end{equation}
        with $\rho_{A,B} = \rho_A \rho_B/(\rho_A+\rho_B)$ if $\rho_A+\rho_B >0$, and
        $\mu_{A,B} = \mu_A \mu_B/(\mu_A+\mu_B)$ if $\mu_A+\mu_B>0$.
		Then, $\bm{A}
		$ is $\left(\underline{\chi}\hat{\mu}\bm{S},\overline{\chi}\hat{\rho}\bm{S}^{-1}\right)$-semimonotone on $\pmb{\mathbb{S}}$.
	\end{prop}
	\begin{proof} 
		Let $(\bm{x},\bm{u})=((z,v), (x,u)) \in \gra\bm{A}$ and let $\bm{x}^*=(z^*,v^*) \in \zer (\bm{A}+\bm{C})$, thus $(\x^*,-\bm{C}\x^*) \in  \pmb{\mathbb{S}}$. We have then $x-L^*v \in (A+D)z$, $u +Lz\in B^{-1}v$, $-Cz^*-L^*v^*\in (A+D)z^*$, and $Lz^* \in B^{-1}v^*$. 
		First, suppose that $\min\{\rho_A+\rho_B,\mu_B+\mu_A\}>0$. It follows from Assumption~\ref{assume:1}\ref{eq:monAassume1}-\ref{eq:monBassume1b}, Lemma~\ref{lem:des}, and \eqref{eq:ineqS} that
		\begin{align*}
			\scal{\bm{x}-\x^*}{\bm{u}+\bm{C}\x^*}&= \scal{(z-z^*,v-v^*)}{(x+Cz^*,u)}\\
			&= \scal{z-z^*}{x+Cz^*} +\scal{v-v^*}{u}\\
			&=\scal{z-z^*}{x-L^*v+Cz^*+L^*v^*}+\scal{v-v^*}{u+Lz-Lz^*}\\
			&\geq \nu_A\|z-z^*\|^2+\mu_A\|z-z^*\|^2_{L^*L}+\rho_A\|x-L^*v+Cz^*+L^*v^*\|^2\\
			& \hspace{1cm}+\nu_B\|v-v^*\|^2+\rho_B\|v-v^*\|^2_{LL^*}+\mu_B\|u+Lz-Lz^*\|^2\\
			&  \geq \nu_A\|z-z^*\|^2+\nu_B\|v-v^*\|^2\\
			& \hspace{1cm}+\frac{\rho_A\rho_B}{\rho_A+\rho_B}\|x+Cz^*\|^2+\frac{\mu_A\mu_B}{\mu_B+\mu_A}\|u\|^2\\
			&  \geq \min\{\nu_A,\sigma\nu_B/\tau\}\left(\|z-z^*\|^2+\frac{\tau}{\sigma}\|v-v^*\|^2\right)\\
			& \hspace{1cm}+\min\left\{0,\frac{\rho_A\rho_B}{\rho_A+\rho_B},\frac{\tau \mu_A\mu_B}{\sigma(\mu_B+\mu_A)}\right\}\left(\|x+Cz^*\|^2+\frac{\sigma}{\tau}\|u\|^2\right)\\
			&  \geq \underline{\chi}\hat{\mu}\|\x-\x^*\|^2_{\bm{S}}+\overline{\chi}\hat{\rho}\|\bm{u}+\bm{C}\x^*\|_{\bm{S}^{-1}}^2.
		\end{align*}
		The remaining cases are treated analogously.
	\end{proof}
	\begin{prop}\label{prop:STT}
		In the context of Notation \ref{no:Sec4} and Assumption~\ref{assume:1}, let $(\sigma,\tau) \in \RPP^2$ be such that $\sigma\tau\|L\|^2<1$, and consider the operators
		\begin{equation}\label{def:S1}
			\begin{aligned}
				&\bm{S} \colon \bm{\H} \to {\bm{\H}} \colon (z,v) \mapsto (z-\tau L^*v,-\tau L z + \tau v/\sigma)\\
				&T_1 \colon\H \to \H \colon z \mapsto (\id-\sigma\tau L^*L)z\\
				&T_2 \colon\G \to \G \colon v \mapsto (\id-\sigma\tau LL^*)v.
			\end{aligned}
		\end{equation}
		The following assertions hold:
		\begin{enumerate}
			\item\label{prop:STT1} $T_1$, $T_2$, and $\bm{S}$ are positive definite, self-adjoint, and strongly monotone linear operators. Furthermore,
			\begin{equation}
				\big(\forall (z,v) \in \bm{\H}\big) \quad \bm{S}^{-1}(z,v)=\big(T^{-1}_1(z+\sigma L^*v),\sigma T^{-1}_2(Lz+v/\tau)\big).
			\end{equation}
			\item\label{prop:STT2} $L\circ T_1^{-1}=T_2^{-1} \circ L$ and $L^*\circ T_2^{-1}=T_1^{-1} \circ L^*$.
			\item\label{prop:STT3} For every $(z,v) \in \bm{\H}$, 
			\begin{equation}\label{eq:STT31}
				\|(z,v)\|^2_{\bm{S}^{-1}} = \|z+\sigma L^*v\|^2_{T_1^{-1}}+\frac{\sigma}{\tau}\|v\|^2=\|z\|^2+\frac{\sigma}{\tau}\|v+\tau Lz\|^2_{T_2^{-1}}.
			\end{equation}
			%In particular,  
            In addition,
            \begin{equation} \label{eq:STT32}
				\|z\|^2 +\frac{\sigma}{\tau}\|v\|^2 \leq (1+{\sqrt{\sigma\tau} \|L\|)}\|(z,v)\|^2_{\bm{S}^{-1}}.
			\end{equation}
			\item\label{prop:STT4}  For every  $(z,v) \in \bm{\H}$, 
			\begin{equation}\label{eq:STT4}
				\|(z,v)\|^2_{\bm{S}} \leq (1+\sqrt{\sigma \tau}\|L\|)\left(\|z\|^2+\frac{\tau}{\sigma}\|v\|^2\right).
			\end{equation}
				\item\label{prop:STT5}  $\bm{A}$ is $\left(\underline{\chi}\hat{\mu}\bm{S},\overline{\chi}\hat{\rho}\bm{S^{-1}}\right)$-semimonotone on $\pmb{\mathbb{S}}$, where $\hat{\mu}$ and $\hat{\rho}$ are defined in \eqref{def:murho1} with  $\overline{\chi}=1+\sqrt{\sigma\tau} \|L\|$ and $\underline{\chi}=\overline{\chi}^{-1}$.
%                and $\overline{\chi}=1+\sigma\tau \|L\|^2$.
			\end{enumerate}
		\end{prop}
		\begin{proof}\ 
			\begin{enumerate}
				\item The first assertion follows directly from existing results, see, e.g., \cite[Lemma~6.1]{MorinBanertGiselsson2022}.
				\item Note that $LT_1=T_2L$, thus,
				\begin{align*}
					(\forall z \in \H) \quad Lz=LT_1T_1^{-1}z						=T_2LT_1^{-1}z.    
				\end{align*}
				Then, we conclude that $T_2^{-1}Lz=LT_1^{-1}z$. The second identity is deduced similarly.

				\item For every $(z,v) \in \bm{\H}$,
				\begin{align*}
					\|(z,v)\|^2_{\bm{S}^{-1}} &= \scal{\bm{S}^{-1}(z,v)}{(z,v)}\\
					&=\scal{T_1^{-1}(z+\sigma L^*v)}{z}+\frac{\sigma}{\tau}\scal{T^{-1}_2(v+\tau Lz)}{v}\\
					&=\scal{T_1^{-1}(z+\sigma L^*v)}{z+\sigma L^*v}-\sigma\scal{T_1^{-1}(z+\sigma L^*v)}{L^*v}+\frac{\sigma}{\tau}\scal{T^{-1}_2(v+\tau Lz)}{v}\\
					&=\|z+\sigma L^*v\|^2_{T_1^{-1}}-\sigma\scal{T_1^{-1}(z+\sigma L^*v)}{L^*v}+\frac{\sigma}{\tau}\scal{T^{-1}_2(v+\tau Lz)}{v}\\
					&\overset{\ref{prop:STT2}}{=}\|z+\sigma L^*v\|^2_{T_1^{-1}}-\sigma\scal{T_2^{-1}L(z+\sigma L^*v)}{v}+\frac{\sigma}{\tau}\scal{T^{-1}_2(v+\tau Lz)}{v}\\
					&=\|z+\sigma L^*v\|^2_{T_1^{-1}}+\frac{\sigma}{\tau}\scal{T^{-1}_2(v-\sigma\tau LL^*v)}{v} \\
					&\overset{\eqref{def:S1}}{=}\|z+\sigma L^*v\|^2_{T_1^{-1}}+\frac{\sigma}{\tau}\|v\|^2.
				\end{align*}
				The identity $  \|(z,v)\|^2_{\bm{S}^{-1}} = \|z\|^2+\frac{\sigma}{\tau}\|v+\tau Lz\|^2_{T_2^{-1}}$ is  deduced in a similar manner. 

                Let $(z',v') = \bm{S}^{-1}(z,v)$. 
                Set $\omega = \sqrt{\sigma\tau}\|L\|$ and $u=\sqrt{\tau/\sigma}v'$.We have then
                \begin{align*}
                \|z\|^2+\frac{\sigma}{\tau}\|v\|^2
                &=
                \|z'-\tau L^*v'\|^2+\sigma \tau\left\|-Lz'+\frac{v'}{\sigma}\right\|^2\\
                &= \|z'\|^2-2\tau\scal{Lz'}{v'}
                + \tau^2\|L^*v'\|^2+\sigma\tau \|Lz'\|^2-2\tau\scal{Lz'}{v'}+\frac{\tau}{\sigma}\|v'\|^2\\
                &\le (1+\omega^2)
                \|z'\|^2-4\sqrt{\sigma\tau} \scal{Lz'}{u}+(1+\omega^2) \|u\|^2
\end{align*}
and 
\begin{equation*}
                \|(z,v)\|^2_{\bm{S}^{-1}}
                = \|(z',v')\|^2_{\bm{S}}
                = \|z'\|^2-2\sqrt{\sigma\tau} \scal{Lz'}{u}+\|u\|^2.
\end{equation*}
Using Cauchy-Schwarz inequality yields
\begin{align*}
&(1+\omega)\|(z,v)\|^2_{\bm{S}^{-1}}-\|z\|^2-\frac{\sigma}{\tau}\|v\|^2\\
&\geq (1-\omega)\big(\omega\|z'\|^2
+2\sqrt{\sigma\tau}
\scal{Lz'}{u}+\omega \|u\|^2\big)\\
&\geq (1-\omega)\big(\omega\|z'\|^2
-2\sqrt{\sigma\tau} \|Lz'\|\|u\|
+\omega \|u\|^2\big)\\
&\geq (1-\omega)\omega(\|z'\|-\|u\|)^2.
\end{align*}
Since $\omega\in [0,1[$, we deduce that
\begin{equation}
\|z\|^2+\frac{\sigma}{\tau}\|v\|^2
\leq 
(1+\omega)\|(z,v)\|^2_{\bm{S}^{-1}}. 
\end{equation}
%\begin{align*}
%                 &\le (1+\sigma\tau \|L\|^2)
%                \left(\|z'\|^2-2\tau \scal{Lz'}{v'}+\frac{\tau}{\sigma}\|v'\|^2\right)\\
%                & = (1+\sigma\tau \|L\|^2) \|(z',v')\|^2_{\bm{S}}\\
%                & = (1+\sigma\tau \|L\|^2) 
%                \|(z,v)\|^2_{\bm{S}^{-1}},
%                \end{align*}
%                where, in the last inequality, we have used the fact that 
%                \begin{equation}
%                    \sigma\tau\|L\|^2 \le  %1
%                    \quad \Leftrightarrow \quad \frac{2}{1+\sigma\tau\|L\|^2} \geq 1.
%                \end{equation}
				\item   Assume that $L\neq 0$. For every $(z,v) \in \bm{\H}$ and $\delta \in ]0,+\infty[$,    \begin{align*}
					\|(z,v)\|^2_{\bm{S}} &= \scal{(z-\tau L^*v,\tau v/\sigma -\tau Lz)}{(z,v)}\\
 					&=\|z\|^2-2\tau\scal{Lz}{v}+\frac{\tau}{\sigma}\|v\|^2\\
					&\leq \|z\|^2+2\tau\|L\|\|z\|\|v\|+\frac{\tau}{\sigma}\|v\|^2\\
                    & \leq \|z\|^2+\delta\|z\|^2+\frac{\tau^2\|L\|^2}{\delta}\|v\|^2+\frac{\tau}{\sigma}\|v\|^2\\
                    &\leq \max\left\{1+\delta,1+\frac{\sigma\tau\|L\|^2}{\delta}\right\}                    \left(\|z\|^2+\frac{\tau}{\sigma}\|v\|^2\right).
				\end{align*}
                The tightest value of the bound is obtained by setting $\delta = \sqrt{\sigma \tau}\|L\|$, which yields
\begin{equation}
\|(z,v)\|^2_{\bm{S}} \le 
(1+\sqrt{\sigma \tau}\|L\|) \left(\|z\|^2+\frac{\tau}{\sigma}\|v\|^2\right).
\end{equation}
This inequality remains valid when $L=0$.
				\item The last assertion follows directly from \ref{prop:STT3}, \ref{prop:STT4}, and Proposition~\ref{prop:ursemimongen}. 
			\end{enumerate}
		\end{proof}
\begin{prop}\label{prop:singlevalued}
    In the context of Notation \ref{no:Sec4} and Assumption~\ref{assume:1}, let $(\tau,\sigma) \in \Omega_A\times\Omega_B$.
    Define
			\begin{equation}\label{eq:defMPD}
				\bm{M}\colon \bm{\H} \to {\bm{\H}} \colon (z,v) \mapsto (z/\tau-D z-L^*v,-Lz+v/\sigma).
			\end{equation}
            Then $(\bm{M}+\bm{A})^{-1}$ is single-valued and has full domain. Moreover, 
            \begin{equation}\label{eq:resPD}
                (\forall (x,w) \in \bm{\H}) \quad (\bm{M}+\bm{A})^{-1}(x,w) = \left(J_{\tau A}(\tau x),J_{\sigma B^{-1}}\big(\sigma( w + 2 LJ_{\tau A} (\tau x))\big)\right). 
            \end{equation}
\end{prop}
\begin{proof}
It follows from \eqref{def:operators} and \eqref{eq:defMPD} that
\begin{equation*}
  (\forall (z,v) \in \bm{\H}) \quad  (\bm{M}+\bm{A})(z,v) = (Az+z/\tau) \times (B^{-1}v+v/\sigma -2Lz).
\end{equation*}
Therefore, given $(z,v) \in \bm{\H}$ and $(x,w) \in \bm{H}$,  $(z,v) \in (\bm{M}+\bm{A})^{-1}(x,w)$ if and only if $x\in 
(\tau^{-1}\id+A)z$ and $w+2Lz \in (\sigma^{-1}\id+B^{-1})v$. Since $J_{\tau A}$ and $J_{\sigma B^{-1}}$ are single-valued, we conclude that $z = J_{\tau A} (\tau x)$ and $v=J_{\sigma B^{-1}}\big(\sigma (w + 2 LJ_{\tau A} (\tau x))\big)$. Since $J_{\tau A}$ and $J_{\sigma B^{-1}}$ are single-valued and have full domain, \eqref{eq:resPD} shows that that $(\bm{M}+\bm{A})^{-1}$ is also single-valued and has full domain. 
\end{proof}
		
		\subsection{Algorithm and convergence analysis}
        In this subsection, we will show that the following primal-dual algorithm provides a numerical method for solving Problem~\ref{pro:main}.
        \begin{algo}\label{algo:MTPD} In the context of 
			Problem~\ref{pro:main} and Assumption~\ref{assume:1}, let $(\tau,\sigma)
			\in \Omega_A\times \Omega_B$, let $\theta \in ]0,2[$, let $(z_{-1},z_0,p_0,v_0) \in \H^3\times\G$, and set $y_{-1}=D 
			z_{-1}$. Consider the iteration:
			\begin{equation}
				\label{e:algonMTPD}
				(\forall n \in \N) \quad \left\lfloor
				\begin{aligned}
					&y_{n} =  D z_{n}\\
					&x_{n} = z_{n}-\tau (y_{n}-y_{n-1}+{Dp_n}+Cz_{n}+L^*v_n)\\
					&p_{n+1} = J_{\tau A} x_{n}\\
					&w_{n} = v_n +\sigma L(2{p_{n+1}}-z_n)\\
					&q_{n+1} = J_{\sigma B^{-1}}w_{n}\\
                    &(z_{n+1},v_{n+1})=(1-\theta)(z_{n},v_{n})+\theta (p_{n+1},q_{n+1}).
				\end{aligned}
				\right.
			\end{equation} 
		\end{algo}

        Our main result concerning the convergence of the above algorithm is stated next.
		\begin{teo}\label{teo:MMPD}
			In the context of Problem~\ref{pro:main} and  
			Assumption~\ref{assume:1}, consider the sequence
            $(z_{n},v_n)_{n\in\N}$ generated by 
			Algorithm~\ref{algo:MTPD} where
            $(\tau,\sigma) \in \Omega_A\times\Omega_B$ is such that $\kappa = 1-\sigma\tau\|L\|^2>0$.  Let $\hat{\rho}\le 0$ be defined by \eqref{def:murho1} and let $\overline{\chi}=1+\sqrt{\sigma\tau} \|L\|$. Moreover
            %set  $\zeta = %\frac{\tau\vartheta_{K}}{\kappa}$, 
            %$\zeta_{\bm{M}}=\frac{\zeta+1}{\tau}$, 
            %and 
            define
            
            		\begin{equation}
                    \label{e:defetatheta}
			%\bm{\xi}_{n} &= \bm{S}^{-1}\big(\tau\bm{M}\textcolor{blue}{\p_{n}}-(\tau\bm{M}-\bm{S})\x_{n-1}\big),\label{e:defbxi}\\
            	\eta_\theta = \begin{cases}
		1, &\textnormal{ if } \theta \in [1,2[ \textnormal{ and } D \textnormal{ is monotone,}\\
		1+|1-\theta|, &\textnormal{ otherwise,} 
            \end{cases}
            \end{equation} 
            
%           $\eta_\theta = 1$, if $\theta \in [1,2[$, and $\eta_\theta = 2-\theta$, otherwise.
            Suppose that
            $\beta\kappa+\overline{\chi}\hat{\rho}(1+\vartheta/\kappa)>0$ and
			\begin{equation}\label{eq:stepsizesteoPD}
				\tau\Big(2-\theta  - 2\tau\frac{\eta_\theta\vartheta}{\kappa} \Big)>%2\tau^2\frac{\vartheta}{\kappa}+
                \frac{(\tau+2\overline{\chi}\hat{\rho})^2}{2(\beta\kappa+\overline{\chi}\hat{\rho}(1+\vartheta/\kappa))} -2\overline{\chi}\hat{\rho}
            \left( \left(2\tau\frac{\vartheta}{\kappa}+1\right)^2+4\tau^2\frac{\vartheta}{\kappa}\right).
			\end{equation}
			Then the following hold.
			\begin{enumerate}
				\item\label{teo:MTMMPD1}  The sequence $(z_{n},v_n)_{n\in\N}$  converges 
				weakly to a solution to Problem~\ref{pro:main}.
				\item\label{teo:MTMMPD2} If $\nu_A>0$ and $\nu_B>0$, then $(z_n,v_n)_{n\in \N}$
				converges R-linearly to the unique solution to
				Problem~\ref{pro:main}.
			\end{enumerate}
		\end{teo}
		\begin{proof} Let $(\tau,\sigma) \in \Omega_A\times \Omega_B$ and
			consider the operators defined in \eqref{def:operators}, operator $\bm{S}$ defined in \eqref{def:S1}, and operator $\bm{M}$ defined in \eqref{eq:defMPD}.
			Let, for every $n \in \N$,  $\x_{n}=(z_n,v_n)$, $\p_{n+1} = (p_{n+1},q_{n+1})$, and $\bm{u}_{n+1}=(\tau \bm{M} - \bm{S}){\p_{n+1}} -(\tau \bm{M} -\bm{S})\x_n$, where
            $\tau \bm{M} - \bm{S}\colon (z,v)\mapsto (-\tau D z,0)$.
			It follows from \eqref{e:algonMTPD} and \eqref{eq:resPD} that 
			\begin{align*}
					(p_{n+1},q_{n+1}) = (J_{\tau A} x_n, J_{\sigma B^{-1}} w_n)
				&\Leftrightarrow 
					\p_{n+1} = (\bm{M}+\bm{A})^{-1}(x_n/\tau,w_n/\sigma-2L{p}_{n+1}).
			\end{align*}	
			Moreover,
			\begin{align*}
				&\begin{cases}
					x_n/\tau = z_{n}/\tau- D z_{n}-{D p_n}+ D z_{n-1}-Cz_{n}-L^*v_n\\
					w_n/\sigma-2 L{p_{n+1}} = v_n/\sigma -Lz_n
				\end{cases}\\
				\Leftrightarrow \quad &(x_n/\tau,	w_n/\sigma-2L{p}_{n+1}) = \bm{M}\x_n-\bm{C}\x_n+\bm{u}_n/\tau.
			\end{align*}
			Therefore, by setting 
            $\bm{u}_0 = (-\tau D p_0+\tau D z_{-1},0)$, we conclude that	
			\begin{equation} 
				(\forall n\in\N)\quad 
				\begin{array}{l}
					\left\lfloor
					\begin{array}{l}
						\p_{n+1} = (\bm{M} + \bm{A})^{-1} (\bm{M}\x_n-\bm{C}\x_n+\bm{u}_n/\tau)\\
						\bm{u}_{n+1} = (\tau \bm{M} - \bm{S}){\p_{n+1}} -(\tau \bm{M} -\bm{S})\x_n\\
                        \x_{n+1}=(1-\theta)\x_n+\theta \p_{n+1},
					\end{array}
					\right.
				\end{array}
			\end{equation}
			which is of the form of Algorithm~\ref{algo:NLFBMM}. In addition, $(\bm{M}+\bm{A})^{-1}$ is single-valued and has full domain by Proposition~\ref{prop:singlevalued} and $\gra(\bm{A}+\bm{C})$ is sequentially  weak-strong closed by Assumption~\ref{assume:1}. According to Proposition~\ref{prop:STT}, $\bm{A}$ is $(\mu\bm{S},\rho\bm{S}^{-1})$-semimonotone on $\pmb{\mathbb{S}}$ with $(\mu,\rho) = (\underline{\chi}\hat{\mu},\overline{\chi}\hat{\rho})\in [0,+\infty[\times ]-\infty,0]$
            with $\underline{\chi}=(1+\sqrt{\sigma\tau} \|L\|)^{-1}$.
            Moreover, it follows from the proof of \cite[Corollary~6.1]{MorinBanertGiselsson2022} that $\bm{C}$ is $(\beta\kappa)$-cocoercive with respect to $\bm{S}$ and $\tau \bm{M} -\bm{S}$ is $\zeta$-Lipschitzian with respect to $\bm{S}$ with $\zeta = \tau\vartheta/\kappa$. 
            In addition, %$\bm{S}-\tau \bm{M} = (\tau D,0)$ is monotone, thus
            $\eta_\theta$ coincides with $\eta$ defined in \eqref{def:nu}. Finally, the condition $\beta\kappa+\overline{\chi}\hat{\rho}(1+\vartheta/\kappa)>0$
            is equivalent to $\tau>-\zeta\rho/(\rho+\beta\kappa)=\underline{\tau}_\zeta$, and 
            implies that
            $\rho> -\beta \kappa$.
            In summary, Assumption~\ref{assum:2} is satisfied provided that $\zeta<1/2$. In addition, 
            according to Remark~\ref{re:firstrem}\ref{re:firstremi}, $\bm{M}$ is $\zeta_{\bm{M}}$-Lipschitzian with respect to $\bm{S}$. The convergence results then follow from Theorem~\ref{teo:NLFBMM} by noticing that, in the considered setting, $\zer (\bm{A}+\bm{C})=\bm{Z}$ and the condition $\lambda > 0$ reads
            \begin{equation}\label{eq:stepsizesteoPDproof}
            \lambda =
				2-\theta-\frac{2\eta_\theta\tau\vartheta}{\kappa}-\frac{(\tau+2\rho)^2}{2\tau(\beta\kappa+\rho(1+\vartheta/\kappa))} +2\rho\tau
            \left( \left(\frac{\vartheta}{\kappa}+\zeta_{\bm{M}}\right)^2+\frac{4\vartheta}{\kappa}\right)> 0.
			\end{equation}
            By virtue of Remark~\ref{rem:lambda}\ref{rem:lambda0},
            this condition ensures that $\zeta< 1/2$.
            Majorizing $\zeta_{\bm{M}}$
            by $(1+\zeta)/\tau$, condition \eqref{eq:stepsizesteoPD} implies \eqref{eq:stepsizesteoPDproof}.
            
            The R-linear convergence is 
            deduced from Theorem~\ref{teo:NLFBMM} when $\hat{\mu}$, as defined by \eqref{def:murho1}, is positive.
		\end{proof}	
		\begin{rem}\ 
        Suppose that $L \neq 0$ and
        set $\kappa \in ~]0,1[$. For any $\tau \in \RPP$ we can find $\sigma \in \RPP$ such that $\kappa = 1-\sigma\tau\|L\|^2$. Then, $\sigma = (1-\kappa)/(\tau \|L\|^2)$, $\underline{\chi} = (1+\sqrt{1-\kappa})^{-1}$, and $\overline{\chi} =
        {1+\sqrt{1-\kappa}}$.
  			\begin{enumerate}
				\item When $\hat{\rho} = 0$ and $\theta =1$, \eqref{eq:stepsizesteoPD} reduces to
				\begin{equation}
                \tau < \frac{\kappa}{\frac{1}{2\beta}+2\vartheta},
				\end{equation}
				which is the condition in \cite[Corollary 6.1]{MorinBanertGiselsson2022} guaranteeing convergence of Primal-Dual Method with Block Triangular Resolvent splitting in the monotone case.
                \item When $\hat{\rho}\neq 0$ and $\rho_A=\rho_B=0$, it follows from \eqref{def:murho1}  that
                $\hat{\rho}= \widetilde{\rho} \tau^2$
                where $\widetilde{\rho} = \|L\|^2\mu_{A,B}/(1-\kappa) < 0$.
                %$\hat{\rho} = \frac{\varrho_2\tau \mu_A \mu_B}{\sigma(\mu_A+\mu_B)}$. Set $\widetilde{\rho} =  \frac{\varrho_2\mu_A \mu_B}{\sigma(\mu_A+\mu_B)}$, 
                Thus, \eqref{eq:stepsizesteoPD} reads 
                \begin{equation}
                    \frac{(1+2\overline{\chi}\widetilde{\rho}\tau)^2}{4\big(\beta\kappa+\overline{\chi}\widetilde{\rho}(1+\vartheta/\kappa)\tau^2\big)}-
                \overline{\chi}\widetilde{\rho}
            \left( \left(\frac{2\tau \vartheta}{\kappa}+1\right)^2+\frac{4\tau^2 \vartheta}{\kappa}
                    \right)+\frac{\eta_\theta\vartheta}{\kappa}<\frac{2-\theta
                    }{2\tau},
                \end{equation} 
                which holds for $\tau$ small enough. This is consistent with the choice of $\widetilde{\Omega}_A$ (resp. $\widetilde{\Omega}_B$) in  Condition~\ref{p:Omegacondii} (resp. \ref{p:Omegacondiii}) of Proposition~\ref{p:Omegacond}.      
               \item Situations where $\hat{\rho} \neq 0$ and $\mu_A=\mu_B=0$ will be studied in Section~\ref{se:specicase}.     %         
               %,  from \eqref{def:murho1}  we have $\hat{\rho} = \rho_{A,B}< 0$ 
               %\textcolor{red}{see later}
               %and  \eqref{eq:stepsizesteoPD} can be written as
                %	\begin{equation}\label{eq:stepsizesPDremarkiii}
				%P(\tau,\kappa):=(2-\theta)\tau-2\tau^2\frac{\vartheta\eta_\theta}{\kappa}-\frac{(\tau+2(2-\kappa)\rho_{A,B})^2}{2(\beta\kappa+(2-\kappa)\rho_{A,B}(1+\vartheta/\kappa))} +2(2-\kappa)\rho_{A,B}
            %\left( \left(2\tau\frac{\vartheta}{\kappa}+1\right)^2+4\tau^2\frac{\vartheta}{\kappa}\right)>0.
			%\end{equation}
            %Note that, given $\tau \in \R$, $\kappa \mapsto P(\tau,\cdot)$ is continuous in $\kappa = 1$. Therefore, since $P(\tau,1)>0$ corresponds with the condition \eqref{eq:tauB=L=0}, we have that \eqref{eq:stepsizesPDremarkiii} holds for $\kappa$ close enough to $1$ and $\theta =1$ under the conditions presented in Remark~\ref{rem:parcasescor}\ref{rem:parcasescor1}.
				\item  The authors of \cite{MorinBanertGiselsson2022} included an extrapolation step in the iterates of the Primal–Dual Method with Block Triangular Resolvent. Such an extrapolation step could also be incorporated into our analysis, but we omit it for the sake of simplicity. In the same article, the authors also proposed the Primal–Dual Algorithm with Resolvent-Corrected Kernel, which could likewise be extended to our setting.
                %; however, we omit this extension for the same reason.
			\end{enumerate}
		\end{rem}

        \subsection{Specific cases}
        \label{se:specicase}
        We now present iterative methods derived from Algorithm~\ref{algo:MTPD}  for specific
        instances of Problem~\ref{pro:main}.
        \subsubsection{Case $D=0$} The previous analysis allows us to establish convergence results for the Condat--V\~u algorithm \cite{Condat13,Vu13} in the semimonotone case.
		\begin{prop}\label{cor:FHRB0}
			In the context of Problem~\ref{pro:main} and  
			Assumption~\ref{assume:1}, suppose that $D=0$.
            Let $(\tau,\sigma) \in \Omega_A\times\Omega_B$,  let $\theta \in ]0,2[$, and let $\hat{\rho}\le 0$ be defined by \eqref{def:murho1}.
            % be such that $\sigma\tau\|L\|^2<1$ 
            %for %$\underline{\chi}=(1+\sqrt{\sigma\tau} \|L\|)^{-1}$ and 
            %$\overline{\chi}=$. 
            Moreover, assume that   
            \begin{align}
&   \sqrt{\sigma\tau} \|L\| < \frac{\beta+\hat{\rho}}{\beta},
\label{eq:stepsizescorCVpr}
\\
& -\left(1 + 2\beta\sigma\|L\|^2(2-\theta)\right) \tau^2 
-2\hat{\rho}\sqrt{\sigma}\|L\|\left(2\beta\sigma\|L\|^2 + \theta\right) \tau^{3/2}\nonumber\\ 
&\qquad+ 2\left(\beta(2-\theta) - \hat{\rho}(2\beta\sigma\|L\|^2+ \theta)\right) \tau 
+  4\beta\hat{\rho}\sqrt{\sigma}\|L\| \tau^{1/2} 
+ 4\beta\hat{\rho} > 0.
%\\
%&
%\left(1+
%2\sigma\|L\|^2
%\big(\beta(2-\theta)
%+\hat{\rho}(2\beta\sigma \|L\|^2+\theta)\big)\right)\tau^2-4\beta\hat{\rho}<
%2(\beta(2-\theta)-%\theta\hat{\rho})\tau.
\label{eq:stepsizescorCV}
            \end{align}
            Let $(z_0,v_0) \in \H \times \G$ and  consider the sequence $(z_{n},v_n)_{n\geq 1}$ generated by the following iteration:
			\begin{equation}
				\label{e:algoCV}
				(\forall n \in \N) \quad \left\lfloor
				\begin{aligned}
					&p_{n+1} = J_{\tau A}( z_{n}-\tau(Cz_{n}+L^*v_n))\\
					&q_{n+1} = J_{\sigma B^{-1}}(v_n +\sigma L(2{p_{n+1}}-z_n))\\
                    &(z_{n+1},v_{n+1})=(1-\theta)(z_{n},v_{n})+\theta (p_{n+1},q_{n+1}).
				\end{aligned}
				\right.
			\end{equation} 
			
			Then the following hold.
			\begin{enumerate}
				\item\label{teo:MTMMCV1}  $(z_n,v_n)_{n\in\N}$ converges 
				weakly to a solution to Problem~\ref{pro:main}.
                %where $D_K=0$.
				\item\label{teo:MTMMCV2} If $\nu_A>0$ and $\nu_B>0$, then $(z_n,v_n)_{n\in \N}$
				converges R-linearly to the unique  solution to Problem~\ref{pro:main}.
                %where $D_K=0$.
			\end{enumerate}
		\end{prop}
		\begin{proof}
        Let $\kappa = 1-\sigma\tau\|L\|^2$ and  $\overline{\chi}=1+\sqrt{\sigma\tau} \|L\|$.
			The result follows directly from Theorem~\ref{teo:MMPD} by noticing that, when $D=0$, $\vartheta=0$.
            The condition $\beta\kappa+\overline{\chi}\hat{\rho}>0$ implies that $\kappa>0$ and it is equivalent to \eqref{eq:stepsizescorCVpr}.
            %and $\zeta_{\bm{M}}=1/\tau$
            In addition,  \eqref{eq:stepsizesteoPD} can be rexpressed as
            \begin{align}
&2(\beta\kappa+\overline{\chi}\hat{\rho})((2-\theta) \tau +2\overline{\chi}\hat{\rho})>(\tau+2\overline{\chi}\hat{\rho})^2
\label{eq:stepsizescorCV0}\\
\Leftrightarrow\quad &2(2-\theta)\tau\beta\kappa
+
2(2\beta\kappa-\theta\tau)\overline{\chi}\hat{\rho}
>
\tau^2.
\label{eq:stepsizescorCV1}
			\end{align}
            After simplification, this yields the algebraic inequality \eqref{eq:stepsizescorCV}, which is a fourth-degree polynomial inequality in $\sqrt\tau$ when $\hat\rho$ is independent of \(\tau\).
            %\textcolor{blue}{ In addition, $\eqref{eq:stepsizescorCV}$ implies that $\kappa > 0$ and $\beta\kappa+\overline{\chi}\hat{\rho}>0$ This last assertions should be removed, isn't it?}.
 		\end{proof}
        \begin{rem} When $\theta =1$, under \eqref{eq:stepsizescorCVpr},
\eqref{eq:stepsizescorCV0} shows that \eqref{eq:stepsizescorCV} is equivalent to
            \begin{equation}
                -\hat{\rho}(1+\sqrt{\sigma\tau} \|L\|)< \frac{\tau}{2} < \beta (1-\sigma\tau\|L\|^2).
            \end{equation}
            %which implies that \eqref{eq:stepsizescorCVpr} is satisfied.
%            Note that this condition requires that $\beta+\hat{\rho}>0$ and 
%        $\sigma \tau \|L\|^2< (\beta+\hat{\rho})/(\beta-\hat{\rho})$.
        \end{rem}
        As shown below, the previous proposition recovers standard results in the monotone case, while also extending to settings where the operator $A$ or $B$ is hypomonotone or cohypomonotone.
        \begin{cor}\label{c:CVnomon}
            In the context of Problem~\ref{pro:main} and  
			Assumption~\ref{assume:1}\ref{assume:1v}, suppose that $D=0$, let $(z_0,v_0) \in \H \times \G$, and consider the sequence $(z_{n},v_n)_{n\geq 1}$ generated by iteration \eqref{e:algoCV} with $(\tau,\sigma)\in ]0,+\infty[^2$ and $\theta \in ]0,2[$. 
            %Let $\kappa = 1 -\tau\sigma\|L\|^2$  and $\overline{\chi}=1+\tau\sigma \|L\|^2$.
            Then, $(z_n,v_n)_{n\in\N}$ converges 
				weakly to a solution to Problem~\ref{pro:main} if one of the following conditions holds:
                \begin{enumerate}
                    \item\label{c:CVnomoni} $A$ and $B$ are maximally monotone and $\tau < 2\beta (2-\theta) (1-\sigma\tau\|L\|^2)$,
                    \item 
                    $\rho_A< 0$, $\rho_B > -\rho_A$, $A$ is maximally $\rho_A$-comonotone, $B$ is maximally $\rho_B LL^*$-comonotone, and 
                    \eqref{eq:stepsizescorCVpr}-\eqref{eq:stepsizescorCV} hold with $\hat{\rho}=\frac{\rho_A \rho_B}{\rho_A+\rho_B}$.
                    %$-\frac{2\rho_A \rho_B}{\rho_A+\rho_B} (1+\sigma\tau \|L\|^2)< \tau < 2\beta(1-\sigma\tau \|L\|^2)$,
                    \item\label{c:CVnomonii} $\rho_B < 0$, $\rho_A >  -\rho_B$, $L\neq 0$, $A$ is maximally  $\rho_A$-comonotone, $B$ is maximally $\rho_B LL^*$-comonotone, \eqref{eq:stepsizescorCVpr}-\eqref{eq:stepsizescorCV} hold with $\hat{\rho}=\frac{\rho_A \rho_B}{\rho_A+\rho_B}$,
                    %$-\frac{2\rho_A \rho_B}{\rho_A+\rho_B} (1+\sigma\tau \|L\|^2)< \tau < 2\beta(1-\sigma\tau \|L\|^2)$, 
                    and                    $\sigma < -\frac{1}{\rho_B\|L\|^2}$,      
                    \item\label{c:CVnomoniii}  $\mu_A<0$, $\mu_B > -\mu_A$, $L\neq 0$,
                    $A$ is maximally $\mu_A L^*L$-monotone, $B$ is maximally $\mu_B$-monotone,
                    \eqref{eq:stepsizescorCVpr}-\eqref{eq:stepsizescorCV} hold with $\hat{\rho}=\frac{\tau}{\sigma}\frac{\mu_A \mu_B}{\mu_A+\mu_B}$, and $\tau<-\frac{1}{\mu_A\|L\|^2}.$
                    %and $-\frac{2\mu_A \mu_B}{\mu_A\mu_B} (1+\sigma\tau \|L\|^2)< \tau < \min\left\{2\beta(1-\sigma\tau \|L\|^2),-\frac{1}{\mu_A\|L\|^2}\right\}$,                    
                    \item\label{c:CVnomoniv}  $\mu_B<0$, $\mu_A > -\mu_B$, $A$ is maximally $\mu_A L^*L$-monotone, $B$ is maximally $\mu_B$-monotone,
                    \eqref{eq:stepsizescorCVpr}-\eqref{eq:stepsizescorCV} hold with $\hat{\rho}=\frac{\tau}{\sigma}\frac{\mu_A\mu_B}{\mu_A+\mu_B}$,
                    %$-\frac{2\mu_A \mu_B}{\mu_A\mu_B} (1+\sigma\tau \|L\|^2)< \tau < 2\beta(1-\sigma\tau \|L\|^2)$, 
                    and $\sigma > -\mu_B$.
                \end{enumerate}
       \end{cor}
       \begin{proof}
           This result is deduced from Proposition~\ref{cor:FHRB0} by noticing that, in each of the considered cases, Assumptions~\ref{assume:1}\ref{eq:assume12}-\ref{eq:monAassume2} are satisfied.
           \begin{enumerate}
               \item Suppose that $\mu_A=\mu_B=\rho_A=\rho_B=0$.
               We have then $\hat{\rho}=0$.
               \item Suppose that $\nu_A=\nu_B=\mu_A=\mu_B=0$ and $\rho_A+\rho_B > 0$. 
               It follows from Proposition~\ref{p:Omegacond}\ref{p:Omegacondi}\&\ref{p:Omegacondiii} that $\Omega_A=]-\rho_A,+\infty[$
               and $\Omega_B = ]0,+\infty[$.
               From \eqref{def:murho1}, 
               $\hat{\rho} = \rho_{A,B} = \frac{\rho_A \rho_B}{\rho_A+\rho_B}$. In addition,
               it is deduced from
        \eqref{eq:stepsizescorCV1} that 
        $\tau>-2\overline{\chi}\hat{\rho}/(2-\theta)$. 
        The latter lower bound is larger than $-\rho_A$ 
        since $-\hat{\rho}>-\rho_A$ and $2\overline{\chi}/(2-\theta)>1$. 
               Thus, condition~\eqref{eq:stepsizescorCV} implies that $\tau > -\rho_A$.
               \item Similarly to the previous case, we deduce from Proposition~\ref{p:Omegacond}\ref{p:Omegacondi}\&\ref{p:Omegacondiv} that $\Omega_A=]0,+\infty[$
               and $\Omega_B = \left]0,-1/(\rho_B\|L\|^2)\right[$.
               \item Suppose that $\nu_A = \nu_B = \rho_A = \rho_B = 0$, and $\mu_A+\mu_B>0$. 
               It follows from Proposition~\ref{p:Omegacond}\ref{p:Omegacondii}\&\ref{p:Omegacondiii} that $\Omega_A=]0,-1/(\mu_A\|L\|^2)[$
               and $\Omega_B = ]0,+\infty[$.
               In addition, from \eqref{def:murho1}, 
               $\hat{\rho} = \frac{\tau}{\sigma}\mu_{A,B} = \frac{\tau}{\sigma}\frac{\mu_A \mu_B}{\mu_A+\mu_B}$.
               \item Similarly to the previous case, we deduce from Proposition~\ref{p:Omegacond}\ref{p:Omegacondii}\&\ref{p:Omegacondiii} that $\Omega_A=]0,+\infty[$
               and $\Omega_B = ]-\mu_B,+\infty[$.
           \end{enumerate}
       \end{proof}

		\begin{rem}
			%\begin{enumerate}\ 
				%\item 	
                In the particular case when $C=0$,  \eqref{e:algoCV} corresponds to the primal-dual algorithm proposed in \cite{ChambollePock2011}. 
                Then, by letting $\beta\to+\infty$, we obtain convergence conditions for this iterative method for semimonotone operators. Similar results were established in \cite[Corollary~5.2]{EvensLatafatPatrinos2025CP}.  
		\end{rem}   
        
        \subsubsection{Case $B=0$ and $L=0$}
		The following result is a specialization of Theorem~\ref{teo:NLFBMM}. It provides a comonotone version of the Forward-Half-Reflected-Backward algorithm proposed in \cite{Malitsky2020SIAMJO} for finding a zero of $A+C+D$.
        We will make the following assumption on the involved operators.
        \begin{asume}\label{a:FHRB}
Let $\rho_A\in \R$ and $\nu_A\in \RP$,
            and suppose that
		\begin{enumerate}
			%\item \label{eq:neqassume1} {\color{red} $D_K\neq 0$},
			\item \label{a:FHRBi} $A+D$ is $(\nu_A,\rho_A)$-semimonotone on $\menge{(x^*,-Cx^*)}{-Cx^*\in (A+D)x^*}$;%$\pmb{\mathbb{S}}_1=\menge{(x^*,-Cx^*)}{(x^*,u^*) \in \bm{Z}}$;
			\item\label{a:FHRBii}  there exists $\Omega_A \subset \RPP$ such that, for every $\tau\in  \Omega_A$, $J_{\tau A}$ is single-valued and has full domain;
			\item\label{a:FHRBiii}   $\gra(A+D+C)$ %$\gra(\bm{A}+\bm{C})$ 
            is sequentially weak-strong  closed.
		\end{enumerate}             
        \end{asume}
 		\begin{prop}\label{cor:FHRB}
			Consider Problem~\ref{pro:main} where $B=0$ and $L=0$. Let $\rho_A\in \R$ and $\nu_A\in \RP$,
            and suppose that
            Assumption~\ref{a:FHRB} holds.
        Let $\tau \in \Omega_A$, let $\theta \in ]0,2[$,
        and let $\hat{\rho}=\min(0,\rho_A)$.
        Let $(z_0,p_0,z_{-1}) \in \H^3$, set $y_{-1}=Dz_{-1}$, and consider the sequence $(z_{n})_{n \geq 1}$ generated by the following recurrence:
			\begin{equation}
				\label{e:algonMT}
				(\forall n \in \N) \quad \left\lfloor
				\begin{aligned}
					&y_{n} =  D z_{n}\\
					&x_{n} = z_{n}-\tau ({y_{n}-y_{n-1}+D p_n}+Cz_{n})\\
					&{p_{n+1}} = J_{\tau A} x_{n}\\
                    &{z_{n+1}}= (1-\theta)z_n+\theta p_{n+1}.
				\end{aligned}
				\right.
			\end{equation} 	
			Set $\eta_\theta$ as in \eqref{e:defetatheta}.
            %$\eta_\theta = 0$, if $\theta \in [1,2[$, and $\eta_\theta = |1-\theta|$, otherwise.} 
            Moreover,       
            assume that
\begin{align}
&\gamma = \beta+\hat{\rho}(1+\vartheta)>0,
\label{eq:tauB=L=0bpre}\\
& b_2 \tau^2+b_1\tau+b_0>0
\label{eq:tauB=L=0b}
\end{align}
with
\begin{equation}\label{eq:defb0b1b2}
    b_0 = 4\hat{\rho}
(\gamma-\hat{\rho}), \quad
b_1 = 2\left(\gamma(2-\theta+8\hat{\rho}\vartheta)-2\hat{\rho}\right),\quad
b_2 = 16\gamma\hat{\rho}\vartheta(1+\vartheta)
-4\gamma\vartheta\eta_\theta-1.
\end{equation}
%\begin{align}
%\left(1+4\gamma\vartheta\big(\textcolor{blue}{\eta_\theta-4\hat{\rho}%(1+\vartheta})\big)\right)\tau^2
%-4\hat{\rho}
%(\gamma-\hat{\rho})<2\left(\gamma(2-\theta)-2\hat{\rho}(1-4\gamma\vartheta)\right)\tau.
%\end{align}
%\begin{align}
%&1+4\gamma\vartheta\bigl(1+\eta_\theta\vartheta-4\hat{\rho}(1+\vartheta)\bigr)>0\\
%&2\gamma(2-\theta)-4\hat{\rho}(1-4\gamma\vartheta)\\
%&-4\hat{\rho}
%(\gamma-\hat{\rho})\geq 0.
%\end{align}

			Then, the following hold.
			\begin{enumerate}
				\item\label{teo:MTMM1}  $(z_n)_{n\in\N}$ converges 
				weakly to $z \in \zer{(A+C+D)}$.
				\item\label{teo:MTMM2} If $\nu_A>0$, then $(z_n)_{n\in \N}$
				converges R-linearly to the unique $z \in \zer{(A+C+D)}$.
			\end{enumerate}
		\end{prop}	
		\begin{proof} 
        In the considered setting,  Algorithm~\ref{algo:NLFBMM} reduces to 
        algorithm \eqref{e:algonMT}.
        We apply Theorem~\ref{teo:MMPD} with
        $\kappa = 1$, $\overline{\chi} = 1$,
         $\mu_B = 0$, 
        and choosing  $\nu_B$
        and $\rho_B$ arbitrarily large.
        In this context, \eqref{eq:stepsizesteoPD} reduces to 
        \begin{equation}
\label{eq:tauB=L=0}
		(2-\theta) \tau >2(\eta_\theta-4\hat{\rho})\tau^2\vartheta+\frac{1}{2\gamma}(\tau+2\hat{\rho})^2 -2\hat{\rho}
            %\left( 
\left(2\tau\vartheta+1\right)^2
    %+4\tau^2\vartheta\right).
			\end{equation}
            After some algebra, this condition can be rewritten as
        \eqref{eq:tauB=L=0b}.
		\end{proof}
		\begin{rem}\label{rem:parcasescor}  Consider the setting of Proposition~\ref{cor:FHRB}. 			\begin{enumerate}
                \item\label{rem:parcasescor1}  
                Define 
                \begin{align}
                &R(\tau ) = b_2\tau^2 +b_1 \tau +b_0,\\
               &\Delta = b_1^2-4b_2 b_0.
                 \end{align}
                Since $\hat{\rho}\le 0$ and 
                $\gamma > 0$,  we have $b_0 \leq 0$
                and $b_2 <0$. Hence, if $b_1\leq 0$ or $\Delta \leq  0$, $R(\tau)\leq 0$ for every $\tau >0$. Otherwise, if $b_1 > 0$ and $\Delta> 0$, $R(\tau)>0$ for $\tau \in \left]\tau_0,\tau_1\right[$.
Because of the signs of $(b_0,b_1,b_2)$, it follows that $\tau_0\geq 0$ ($\tau_0>0$ if $\hat{\rho}\neq 0$). Moreover,
\begin{align}
    \frac{\Delta}{4\gamma}=& \left( 
    (2-\theta)^2+16\hat{\rho}\vartheta(2-\theta+\eta_\theta-4\hat{\rho})\right)\gamma\nonumber\\
    &+
    4\hat{\rho}\left(\theta-1-4\vartheta(2+\eta_\theta)\hat{\rho}+16\vartheta(1+\vartheta)\hat{\rho}^2\right).
\end{align}
Suppose that $\theta\in [1,2[$
and
\begin{equation}
\label{e:condhrhovthbet}
-\hat{\rho}\vartheta<\frac{(2-\theta)^2}{16(2-\theta+\eta_\theta-4\hat{\rho})}, 
\quad
\beta>-\hat{\rho} \frac{ \theta^2 + \vartheta(2-\theta)^2 + 16\vartheta\hat{\rho} \big( \vartheta(2-\theta+\eta_\theta) - \theta \big)}{(2-\theta)^2+16\hat{\rho}\vartheta(2-\theta+\eta_\theta-4\hat{\rho})}.
\end{equation}
Then \eqref{eq:tauB=L=0bpre}-\eqref{eq:tauB=L=0b} hold for $\tau \in ]\tau_0,\tau_1[$ with
\begin{equation}
            0\leq \tau_0=\frac{-b_1+\sqrt{\Delta}}{2b_2}<\tau_1=\frac{-b_1-\sqrt{\Delta}}{2b_2}.
\end{equation} 
Indeed, since 
$
\theta-1
-4\vartheta(2+\eta_\theta)\hat\rho
+16\vartheta(1+\vartheta)\hat\rho^2
\geq0$, $\Delta$ is positive if and only if 
\begin{equation}
\label{eq:varthetarho}
(2-\theta)^2+16\hat{\rho}\vartheta(2-\theta+\eta_\theta-4\hat{\rho})>0
\quad\Leftrightarrow\quad -\hat{\rho}\vartheta<\frac{(2-\theta)^2}{16(2-\theta+\eta_\theta-4\hat{\rho})}
\end{equation}
and
\begin{equation}
\label{eq:betavarthetarho}
    \gamma > -4\hat{\rho}
    \frac{\theta-1
    +4\vartheta\hat{\rho}
    \big(4(1+\vartheta)\hat{\rho}-2-\eta_\theta
    \big)}
    {(2-\theta)^2+16\hat{\rho}\vartheta(2-\theta+\eta_\theta-4\hat{\rho})}.
\end{equation}
Note that \eqref{eq:varthetarho}
implies that
\begin{equation}
2-\theta+8\hat{\rho}\vartheta       
> 2-\theta-\frac{(2-\theta)^2}{2(2-\theta+\eta_\theta-4\hat{\rho})}
= (2-\theta)
\frac{2+2\eta_\theta-\theta-8\hat{\rho}}{2(2-\theta+\eta_\theta-4\hat{\rho})}>0,
\end{equation}
hence $b_1>0$.
Besides,
\eqref{eq:betavarthetarho}
implies that $\gamma >0$, i.e., \eqref{eq:tauB=L=0bpre}
holds.
Furthermore,
\eqref{eq:betavarthetarho} is equivalent to
\begin{equation}
\beta>
-4\hat{\rho}
    \frac{\theta-1
    +4\vartheta\hat{\rho}
    \big(4(1+\vartheta)\hat{\rho}-2-\eta_\theta
    \big)}
    {(2-\theta)^2+16\hat{\rho}\vartheta(2-\theta+\eta_\theta-4\hat{\rho})}
    -\hat{\rho}(1+\vartheta),
\end{equation}
which yields the second inequality in
\eqref{e:condhrhovthbet}.
\item\label{rem:parcasescor3} In the case when $C=0$ and $D\neq 0$, \eqref{e:algonMT} reduces to
				\begin{equation}\label{e:algonMT2}
					(\forall n \in \N) \quad \left\lfloor
					\begin{aligned}
						&y_{n} =  D z_{n}\\
						&x_{n} = z_{n}-\tau (y_{n}-y_{n-1}+D p_n)\\
                        &p_{n+1}= J_{\tau A} x_{n}\\
                        &z_{n+1} = (1-\theta)z_n+\theta p_{n+1}.
					\end{aligned}
					\right.
				\end{equation} 
				This recurrence coincides with the Forward--Reflected-Backward algorithm \cite{Malitsky2020SIAMJO} when $\theta=1$ and $p_0= z_0$. 
                The corresponding convergence conditions are obtained by letting $\beta\to +\infty$.
                Hence, by assuming $\theta\in [1,2[$, the convergence conditions reduce to 
                $(2-\theta)^2+16\hat{\rho}\vartheta(2-\theta+\eta_\theta-4\hat{\rho})>0$
                %$1+32(1-2\hat{\rho})\hat{\rho}\vartheta > 0$ 
                and 
                $\tau \in ]\tau_0,\tau_1[\cap \Omega_A$, where
                \begin{equation}
                \label{e:tauisimp}
                    \tau_i = 
                    \frac{-(2-\theta)-8\hat{\rho}\vartheta+(1-2i)\sqrt{(2-\theta)^2+16\hat{\rho}\vartheta(2-\theta+\eta_\theta-4\hat{\rho})}}{4\vartheta(4\hat{\rho}(1+\vartheta)-\eta_\theta)}, \quad i\in \{0,1\}. 
      %              \frac{-1-8\hat{\rho}\vartheta+(1-2i)\sqrt{1+32\hat{\rho}\vartheta(1-2\hat{\rho})}}{4\vartheta(2\hat{\rho}(1+\vartheta)-1)}, \quad i\in \{0,1\}.
                \end{equation}   
                In \cite{Malitsky2020SIAMJO}, it is further assumed that $A$ is maximally monotone and $D$ is monotone. In this case, we have $\hat{\rho} = 0$, $\Omega_A = ]0,+\infty[$, and
                $(\tau_0,\tau_1) = \big(0,(2-\theta)/(2\vartheta)\big)$.               
				\item\label{rem:parcasescor4} When $D=0$,   \eqref{e:algonMT} reduces to the Forward-Backward algorithm and we recover the convergence conditions in Section~\ref{sec:MNFB} 
                (see Corollary~\ref{coro:overlinezeta}\ref{coro:overlinezetai}).
			\end{enumerate}
            \end{rem}
                
		\section{Numerical Experiments}\label{se:numexp} In this section, we present a numerical illustration of our theoretical results. First, we introduce the adjoint mismatch problem. Next, we provide numerical experiments in the context of signal recovery in the presence of adjoint mismatch.    			\subsection{Application to adjoint mismatch problems} Let $F \colon \H \to \H$ be a $\rho$-cocoercive operator
        for $\rho\in\RPP$
        %maximally $\rho$-comonotone operator for $\rho \in \RPP$ 
        and let $T \colon \H \to \G$ and $K \colon \G \to \H$ be bounded linear operators such that $K T\neq 0$. We assume that $KT-\widetilde{\rho}\, T^* K^*KT$ is a monotone operator with                $\widetilde{\rho}\in \R$ such that $\widetilde{\rho}+\rho > 0$.  We want to
				\begin{equation}
                \label{e:mismatchex}
					\textnormal{find } x \in \H \textnormal{ such that } 0 = Fx+KTx.
				\end{equation}  
                By setting $D=KT$,
                it follows from
                Proposition~\ref{p:ex42(ii)}\ref{p:ex42(ii)ii}
%Remark~\ref{rem:A+D_Kmon}\ref{rem:A+D_Kmon5}
                that $F+D$ is $\rho_{F,D}$-comonotone with $\rho_{F,D} = \rho\widetilde{\rho}/(\rho+\widetilde{\rho})$. The problem can be solved by using algorithms \eqref{e:algonMT} and \eqref{e:algonMT2} under the following two splittings.
                \begin{itemize}
                    \item  Set $A=F$ and $C=0$. Assumption~\ref{a:FHRB}\ref{a:FHRBi} holds with $\nu_A=0$ and $\rho_A=\rho_{F,D}$.  In addition, since $A$ is maximally monotone, Assumption~\ref{a:FHRB}\ref{a:FHRBii} is satisfied for $\Omega_A = \RPP$. 
                    Since $F$ is $\rho$-cocoercive, $D$ is
$\widetilde{\rho}$-comonotone, $\rho+\widetilde{\rho}>0$,
and the bounded linear operator $D$ is weakly continuous,
Lemma~\ref{le:graclosedsemimon} shows that
$\gra(F+D)$ is sequentially weak-strong closed.
Thus, Assumption~\ref{a:FHRB}\ref{a:FHRBiii} also holds.
                    Let $\hat{\rho}=\min\{0,\rho\widetilde{\rho}/(\rho+\widetilde{\rho})\}$, let $(\tau_i)_{0\le i \le 1}$ be defined by \eqref{e:tauisimp} with $\vartheta = \|KT\|$.
                %and assume that
                %\begin{equation}
                %\label{e:rhotau1}
                %   \max\{0,-\rho_A\}
                %< \tau_1. 
                %\end{equation}
                As a consequence of Remark~\ref{rem:parcasescor}\ref{rem:parcasescor3}, if 
                %1) Assumption \ref{a:FHRB}\ref{a:FHRBii} holds, 
                i) $\theta\in [1,2[$,
                ii) $(2-\theta)^2+16\hat{\rho}\|KT\|(2-\theta+\eta_\theta-4\hat{\rho})>0$ 
                and iii) $\tau\in ]\tau_0,\tau_1[$,
                then
                algorithm \eqref{e:algonMT2} 
                converges weakly to a solution to \eqref{e:mismatchex}.
                %If $\widetilde{\rho} \geq 0$, then assumption \eqref{e:rhotau1}
                %obviously holds.
                %If $\widetilde{\rho}<0$, since
                %$\tau_1 \sim (2-\theta)/(2\|KT\|)$ when $KT\to 0$, the assumption is satisfied for $\|KT\|$
                %small enough. 

                    \item Set $A=0$ and $C=F$. 
                    %Suppose that $F$ is single-valued, thus $F$ is $\beta$-cocoercive with $\beta=\rho$.
                    Note that $0+D$ is $\rho_A$-comonotone
                    with $\rho_A=\widetilde{\rho}$
                    and Assumption~\ref{a:FHRB} holds with
                    $\Omega_A = \RPP$.
                    Let $\theta \in [1,2[$, 
                    %let $b_0$, $b_1$, and $b_2$ be as in \eqref{eq:defb0b1b2} with $\hat{\rho}=\min\{0,\widetilde{\rho}\}$, 
                    and suppose that condition~\eqref{e:condhrhovthbet}
                    holds with $\hat{\rho}=\min\{0,\widetilde{\rho}\}$,  $\beta=\rho$,
                    and $\vartheta =\|KT\|$.
                    %$b_1^2-4b_0b_2 >0$, 
                    Then, according to Proposition~\ref{cor:FHRB} and Remark~\ref{rem:parcasescor}\ref{rem:parcasescor1}, inclusion \eqref{e:mismatchex} can be solved by algorithm~\eqref{e:algonMT} when $\tau\in ]\tau_0,\tau_1[$ and $(\tau_0,\tau_1)$ are the bounds given in Remark~\ref{rem:parcasescor}\ref{rem:parcasescor1}.             
                \end{itemize}

        \subsection{Signal recovery example}        
		Let $N \in \N\setminus\{0\}$ and $\overline{x} \in \H=\R^N$ be a signal to be recovered from an observation $r\in \G=\R^M$ which is modelled as follows:
		\begin{equation}\label{eq:observation}
			r = T\overline{x} + e,
		\end{equation}
		where $M\in\N\setminus\{0\}$, $T\in\RR^{M\times N}$, and $e$ is a vector modelling an additive zero-mean white Gaussian noise. 
        To recover the original signal from the observation $r$, we solve the following optimization problem: 
		\begin{equation}\label{eq:opex}
			\min_{x \in \R^N} 
			\lambda H_\delta(Wx)+\frac{1}{2}\|Tx-r\|^2_2,
		\end{equation}
		where $H_\delta$ is the Huber function, $\delta \in \RPP$, $\lambda \in \RPP$ is a regularization parameter, and $W \in\RR^{N\times N}$ is an orthonormal wavelet decomposition. The Huber penalty is defined as
        \begin{equation*}
            H_\delta \colon \H \to \R \colon (x_i)_{1\le i\le N} \mapsto \sum_{i\in \mathbb{I}} h_\delta(x_i),
        \end{equation*}              
        where
        \begin{equation*}
            h_\delta \colon \R \to \R \colon t \mapsto \begin{cases} 
\frac{t^2}{2\delta}, & \text{if } |t| \le \delta, \\
|t| - \frac{\delta}{2}, & \text{if } |t| > \delta.
\end{cases}
        \end{equation*} 
and $\mathbb{I}\subset\{1,\ldots,N\}$ is the index set of wavelet (detail) coefficients.
%        Hereabove,  $H_\delta$ is  applied exclusively to the detail coefficients of the wavelet transform. 
         By Fermat's rule, the considered optimization problem is equivalent to 
		\begin{equation}\label{eq:incluopti1}
			\textnormal{ find } x \in \H \;\;\textnormal{ such that }\;\;  0 = \lambda W^*\nabla H_\delta (Wx) + T^*(Tx-r).
		\end{equation}
		In the case of an adjoint mismatch, $T^*$ is not available and needs to be approximated by $K\in \R^{N\times M}$, and the inclusion becomes
		\begin{equation}\label{eq:incluoptiMM}
			\textnormal{ find } x \in \H \;\;\textnormal{ such that }\;\;  0 = \lambda W^*\nabla H_\delta (Wx)+K(Tx-r),
		\end{equation}
		which is an instance of \eqref{e:mismatchex} for $F = \lambda W^*\circ\nabla H_\delta\circ W-Kr$.
        %, $B=0$, $C=0$, and $L=0$. In view of Remark~\ref{rem:A+D_Kmon}\ref{rem:A+D_Kmon5}, we have that $K\circ T$ is $\rho$-comonotone if and only if $KT-\rho T K^*KT$ is a positive operator. 
        Moreover, $F$ is $\rho$-cocoercive for $\rho=\frac{\delta}{\lambda}$, thus, %by Proposition~\ref{prop:A+Ccomo}~\ref{prop:A+Ccomo2},
    %$A+KT$ is $\widehat{\rho}:=\frac{\rho\delta}{\rho+\delta}$-comonotone if $\rho+\delta > 0$. In addition, $ K\circ L$ is $\|K L \|$-Lipschitz. Altogether, if 
    %\begin{equation}
    %    \hat{\rho} \in \left]\frac{1}{4}\left(1-\sqrt{1+\frac{1}{4\|KL\|}}\right),0\right[,
    %\end{equation}     
     problem~\eqref{eq:incluoptiMM} can be solved by algorithms \eqref{e:algonMT2} or \eqref{e:algonMT} 
   under the scenarios analysed in the previous subsection.  
    % with step-size $\tau \in ]\tau_0,\tau_1[$, where $\tau_0$ and $\tau_1$ are defined in \eqref{e:tauisimp} for $\vartheta = \|KL\|$ (see Remark~\ref{rem:parcasescor}.\ref{rem:parcasescor3}).

		To test these algorithms, we define $T$ as a random normalized matrix of size $M\times N = 256\times 256$ and a mismatched adjoint $K = T^*+s E$ where $E$ is %a random matrix of size $N\times N$ 
        a rank-one perturbation  matrix with unit Frobenius norm, and $s \in \R$.  
        Additionally, we consider $\delta = 10^{-1}$,  $W$ is a decomposition onto an orthonormal Daubechies wavelet basis using filters of length $4$ and $3$ resolution levels. We choose $\widetilde{\rho}$ such that $KT-\widetilde{\rho} T^* K^*KT$ is monotone, in particular,  the minimum eigenvalue of $Q:=(KT+T^*K^*)/2$  is $-2\times 10^{-3}$ 
        and that of $Q-\widetilde{\rho} T^* K^*KT$ is $4.982 \times 10 ^{-4}$ with $\widetilde{\rho} = -0.01$. Note that, in this situation, the affine data fidelity term in \eqref{eq:incluoptiMM} is not monotone. 
        %by calculating the smallest eigenvalue of $(KT+T^*K^*)/2-\widetilde{\rho} (T^*K^*KT)$. 
        We consider two test signals plotted in Figures~\ref{fig:restheta11} and \ref{fig:restheta11p}. The components of the additive Gaussian noise $e$ have a standard deviation of $10^{-2}$. For algorithm \eqref{e:algonMT} we consider the step size $\tau = -0.99(b_1+\sqrt{b_1^2-4b_0b_2})/(2b_2)$, where $b_0$, $b_1$, and $b_2$ are defined in \eqref{eq:defb0b1b2}. For algorithm \eqref{e:algonMT2} we consider $\tau = 0.99\tau_1$, where $\tau_1$ is defined in \eqref{e:tauisimp}. 
        To compare the behaviour of the algorithms in the monotone ($K=T^*$) and non-monotone case ($K\neq T^*$), we also solve the inclusion in \eqref{eq:incluopti1} with algorithms \eqref{e:algonMT2} and \eqref{e:algonMT} for $\hat{\rho} = 0$. As stopping criterion, we use the relative error expressed as $\|x_{n+1}-x_n\|/\|x_n\|$ with tolerance $10^{-7}$. All algorithms were initialized with $z_{-1}=z_0 = p_0= 0$.
        The results obtained for the first signal are presented in Figures~\ref{fig:restheta1} and  \ref{fig:restheta14} for $\theta = 1$ and $\theta = 1.3$,  respectively. Figures~\ref{fig:restheta1p} and  \ref{fig:restheta14p} show the results for the second signal for $\theta = 1$ and $\theta = 1.2$, respectively.    
        From Figures~\ref{fig:restheta11}, \ref{fig:restheta141}, \ref{fig:restheta11p}, and \ref{fig:restheta141p} we observe that the algorithm with and without mismatch recover similar signals.  In Figures \ref{fig:restheta12}, \ref{fig:restheta142},     \ref{fig:restheta13},    and \ref{fig:restheta143},
        we observe that both algorithms \eqref{e:algonMT} and \eqref{e:algonMT2} present a similar decay of the relative error versus the iteration number or computational time for $\lambda = 0.005$. On the other hand, for $\lambda = 0.5$, from Figures \ref{fig:restheta12p}, \ref{fig:restheta142p}, \ref{fig:restheta13p} and \ref{fig:restheta143p},  we observe that algorithm \eqref{e:algonMT2} presents a faster decay of the relative error. In all cases the algorithms with no mismatch are faster in terms of number of iterations or computational time. This is related to the fact that, in the considered scenarios, the negative comonotonicity parameter induced by the mismatch leads to smaller admissible step sizes, which explains the slower convergence we observe. Note that, in our experiments, increasing $\theta$ leads to more iterations to reach the stopping criterion. This can also be explained by the fact that larger values of $\theta$ impose smaller step-sizes.
		\begin{figure}
	\centering
	\subfloat[Original and restored signals]{\label{fig:restheta11}\includegraphics[width=0.32\textwidth]{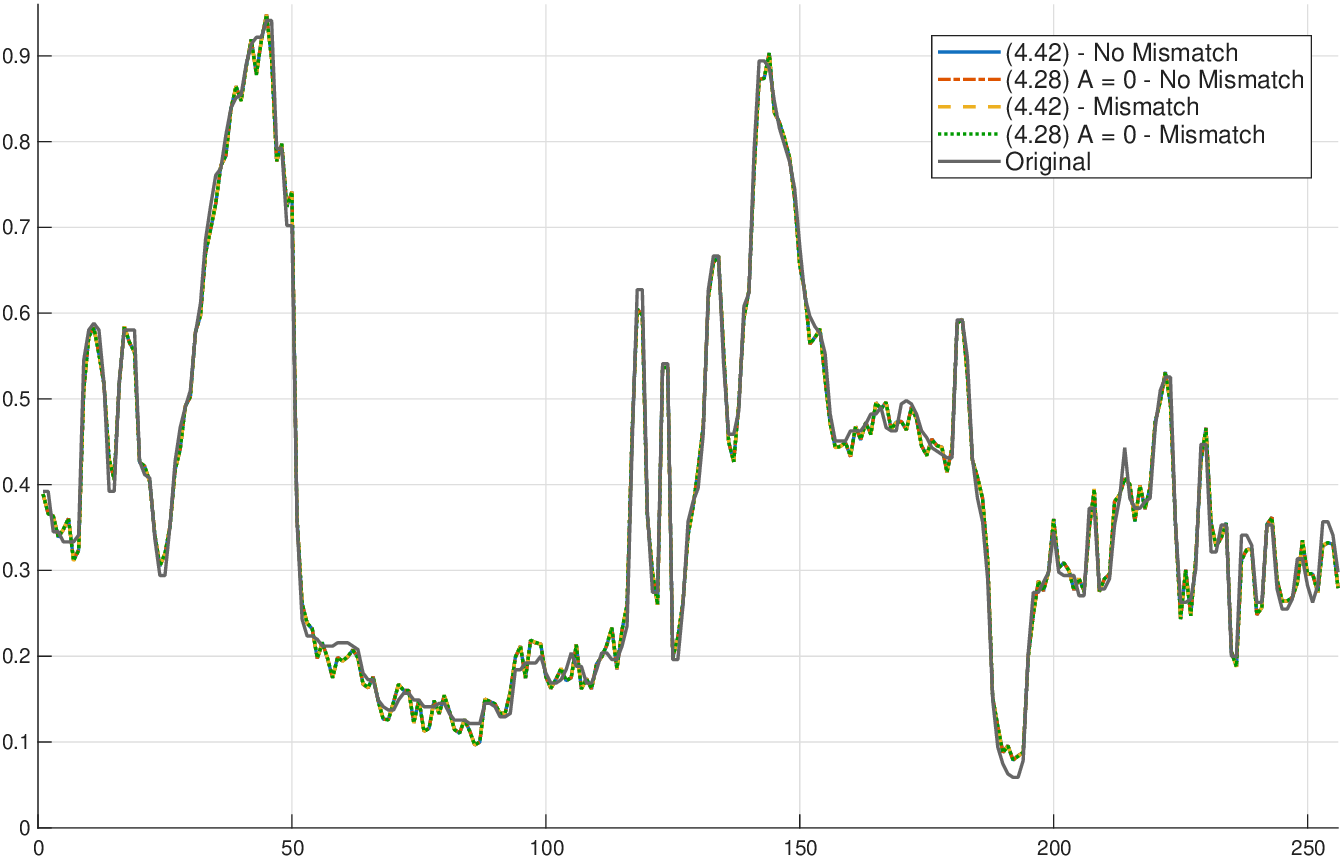}}\,
	\subfloat[Relative error vs iterations]{\label{fig:restheta12}\includegraphics[width=0.32\textwidth]{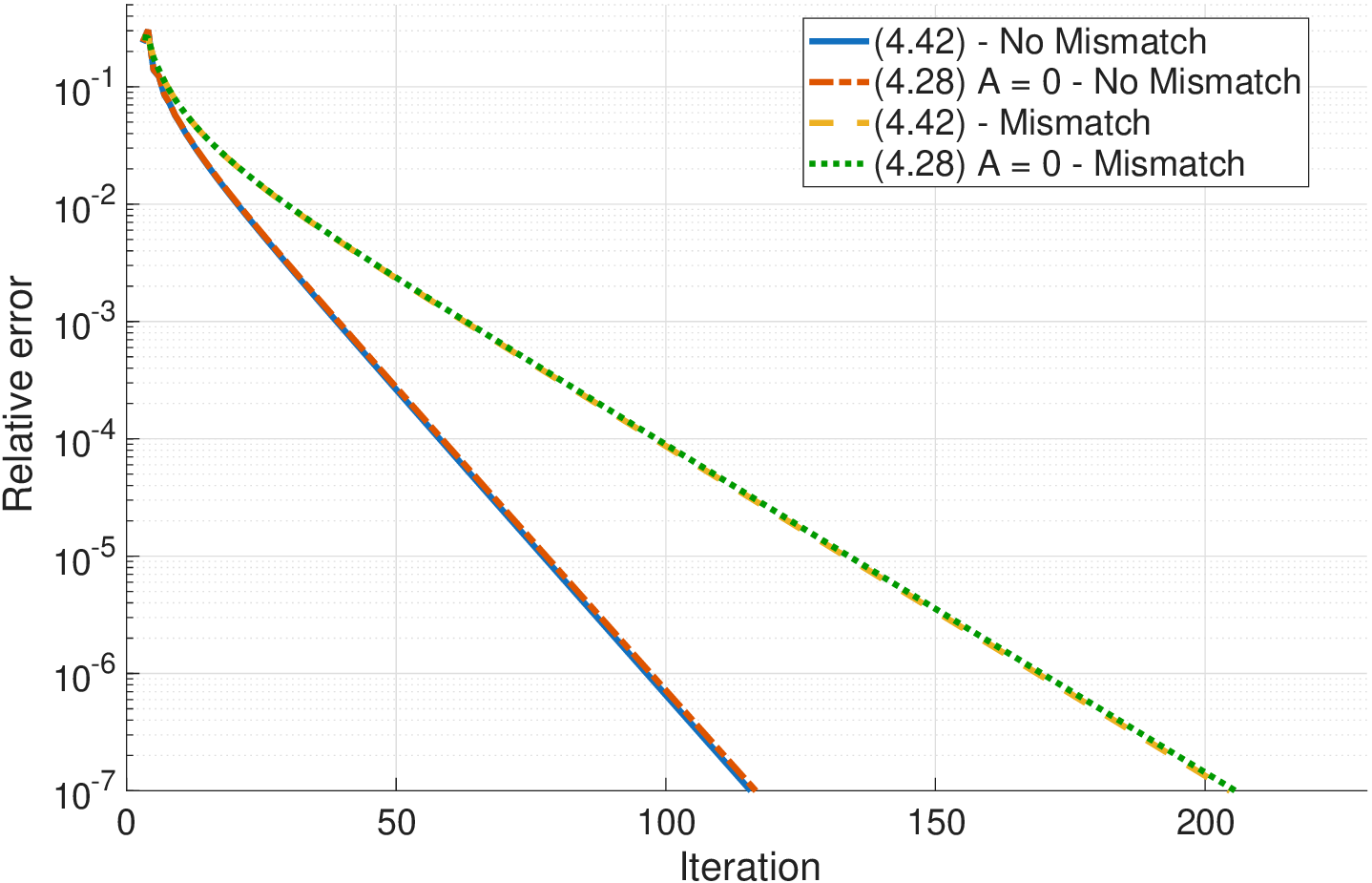}}
	\subfloat[Relative error vs time]{\label{fig:restheta13}\includegraphics[width=0.32\textwidth]{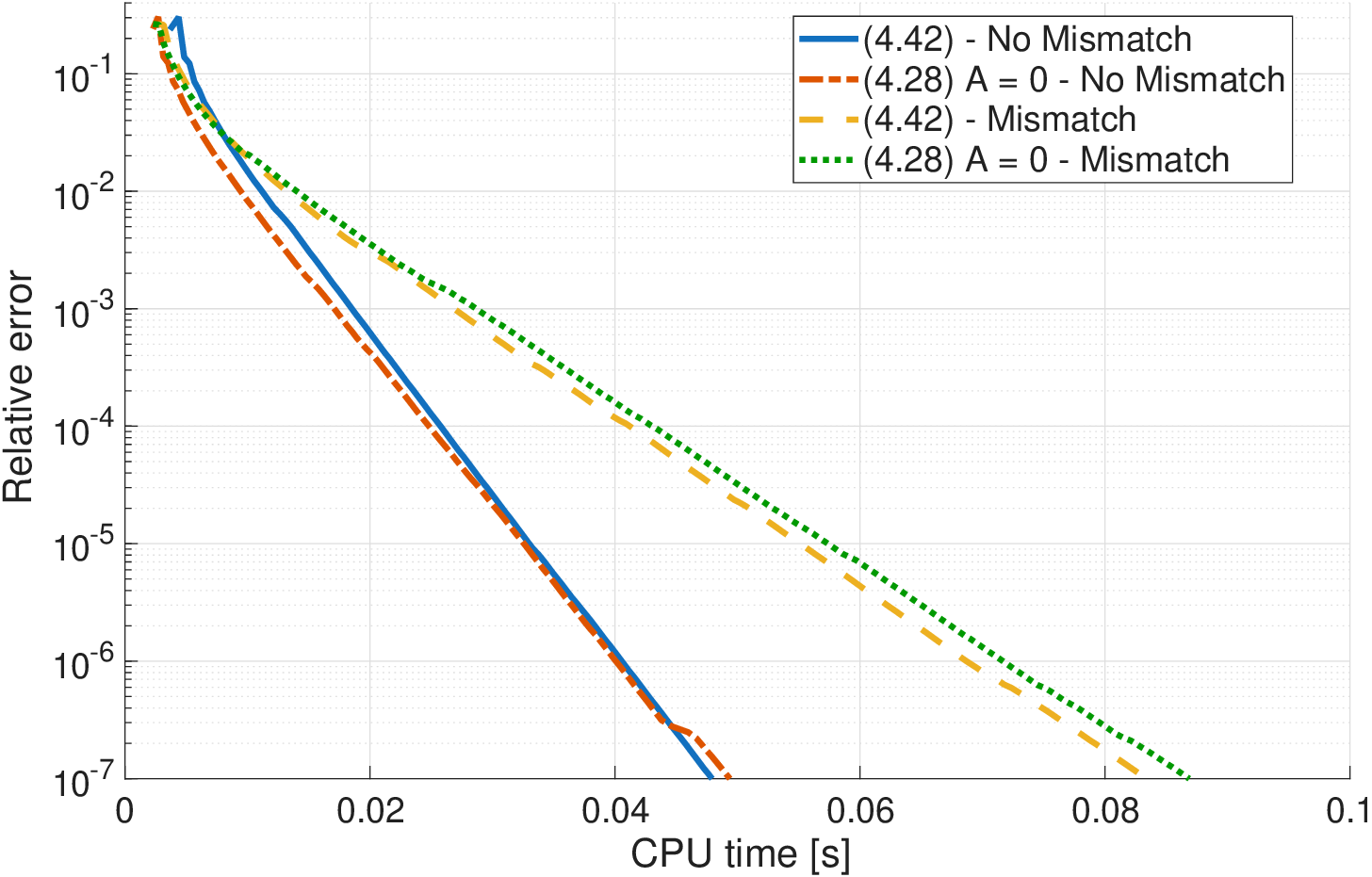}}
	\captionsetup{width=\textwidth} \caption{$\lambda = 0.005$ and $\theta=1$. The PSNR values of all restored signals are approximately 35.23~dB.}
    %: (4.42) No Mismatch 30.18 dB,  (4.28) No Mismatch 30.18 dB, (4.42) Mismatch 30.18 dB,  (4.28) No Mismatch 30.18 dB. }
	\label{fig:restheta1}
\end{figure}
	\begin{figure}
	\centering
	\subfloat[Original and restored signals]{\label{fig:restheta141}\includegraphics[width=0.32\textwidth]{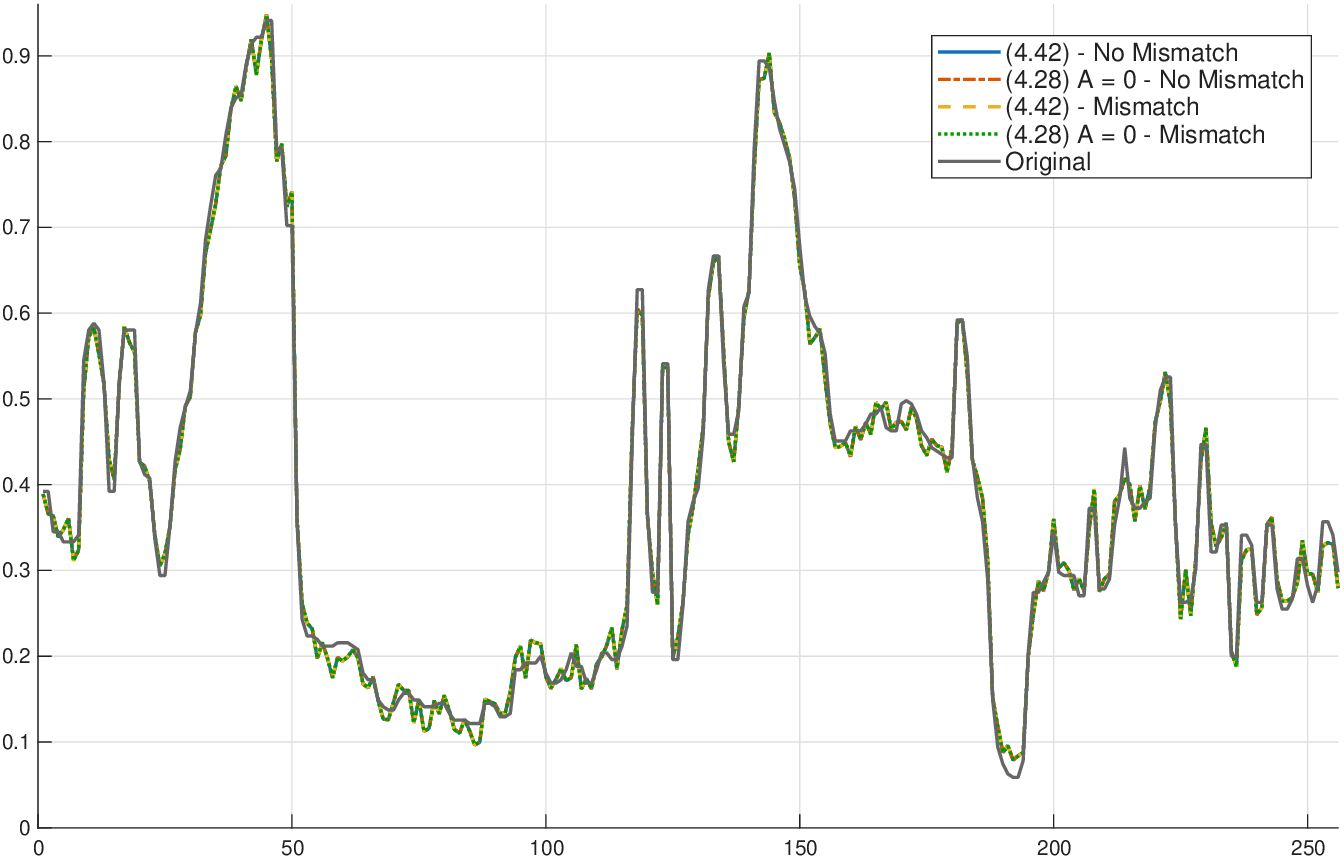}}\,
	\subfloat[Relative error vs iterations]{\label{fig:restheta142}\includegraphics[width=0.32\textwidth]{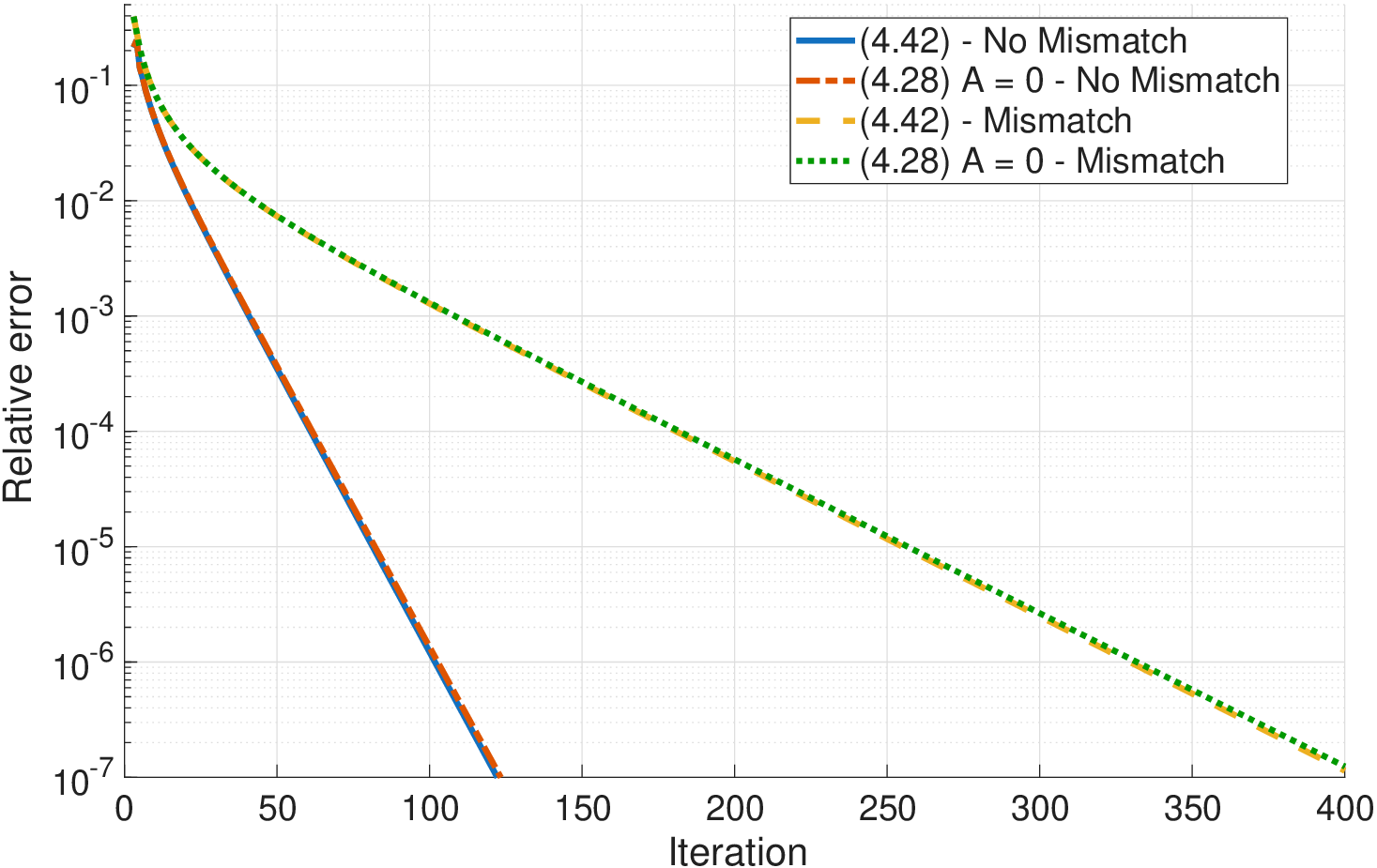}}
	\subfloat[Relative error vs time]{\label{fig:restheta143}\includegraphics[width=0.32\textwidth]{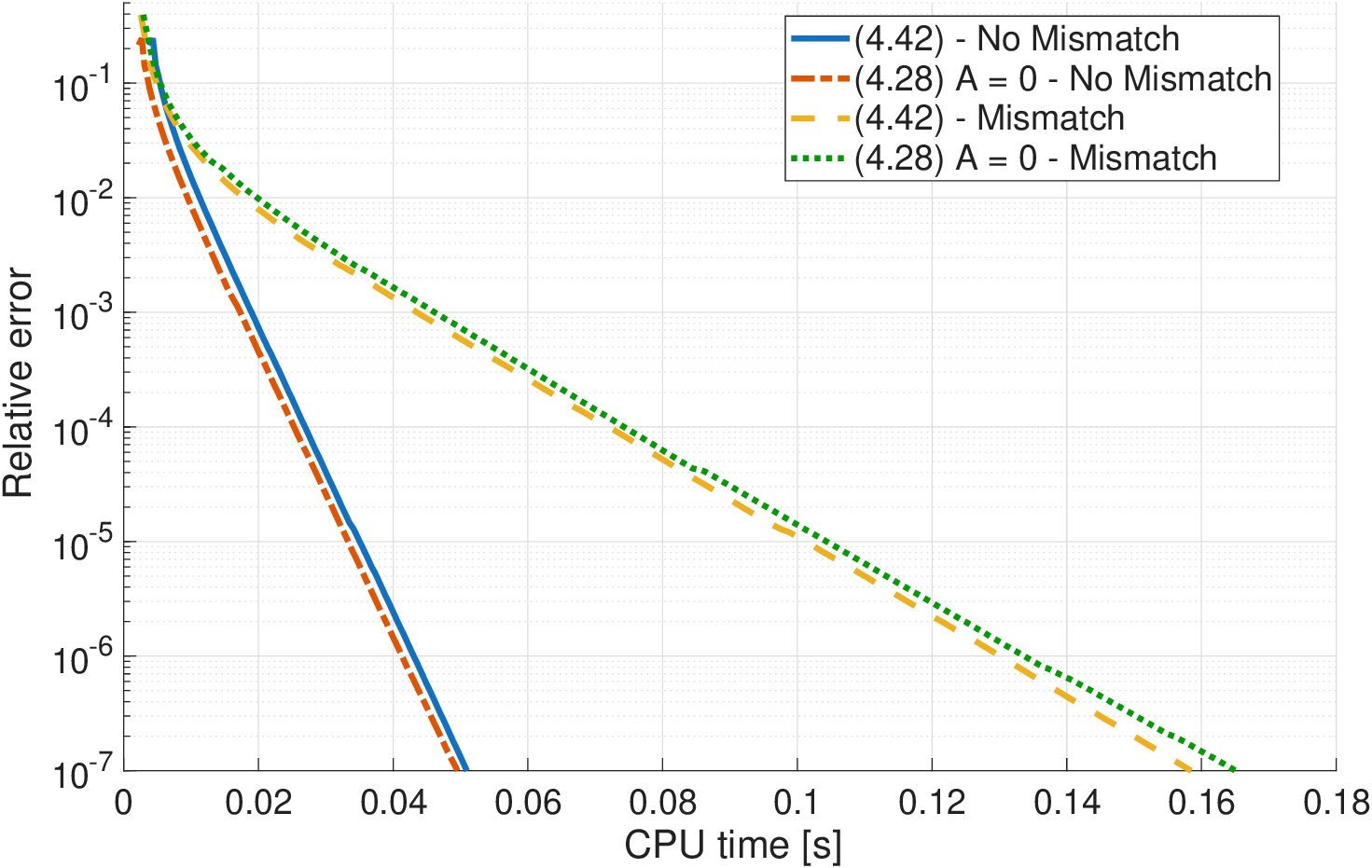}}
	\captionsetup{width=\textwidth} \caption{$\lambda = 0.005$ and  $\theta=1.3$. The PSNR values of all restored signals are approximately 35.23~dB.}
	\label{fig:restheta14}
\end{figure}
\begin{figure}
	\centering
	\subfloat[Original and restored signals]{\label{fig:restheta11p}\includegraphics[width=0.32\textwidth]{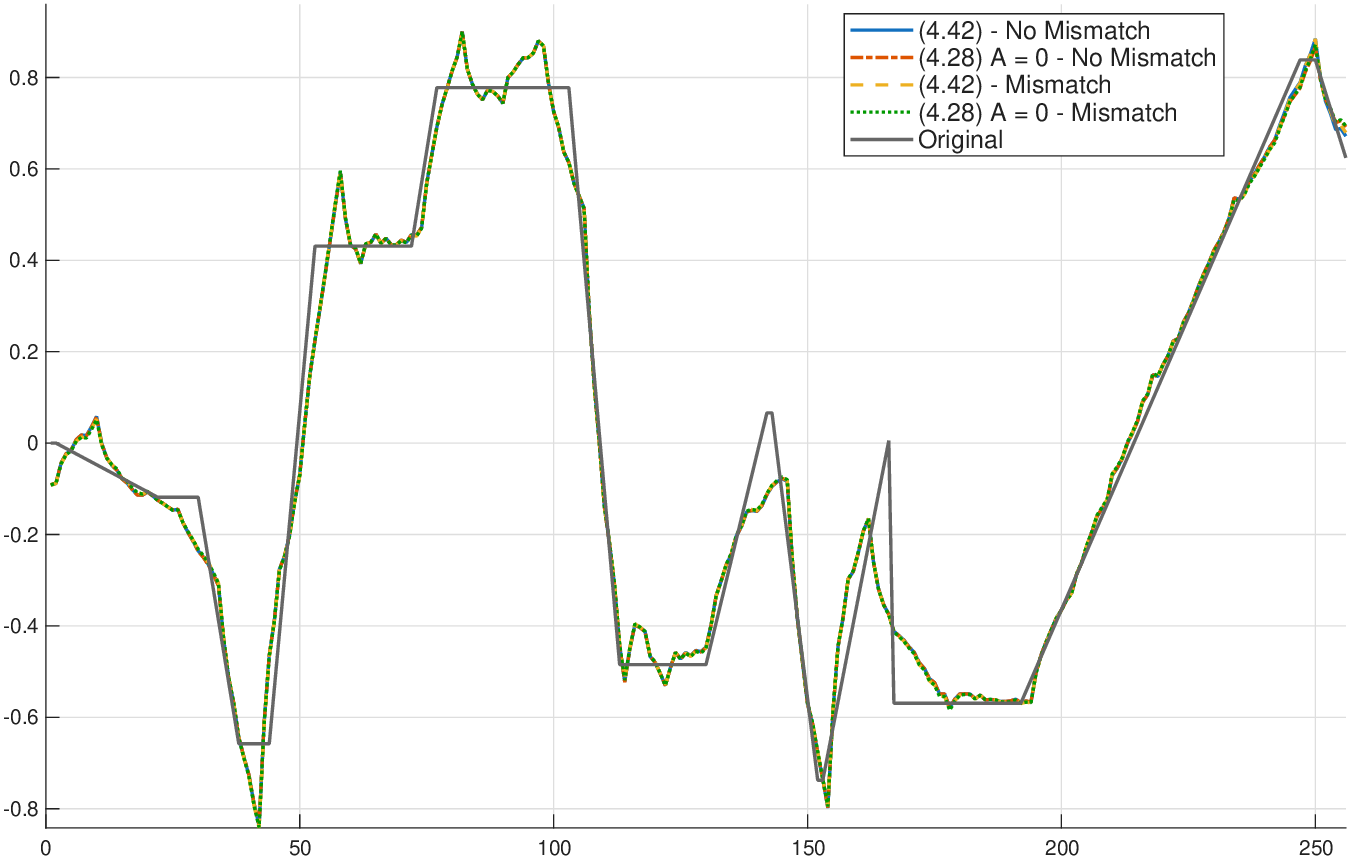}}\,
	\subfloat[Relative error vs iterations]{\label{fig:restheta12p}\includegraphics[width=0.32\textwidth]{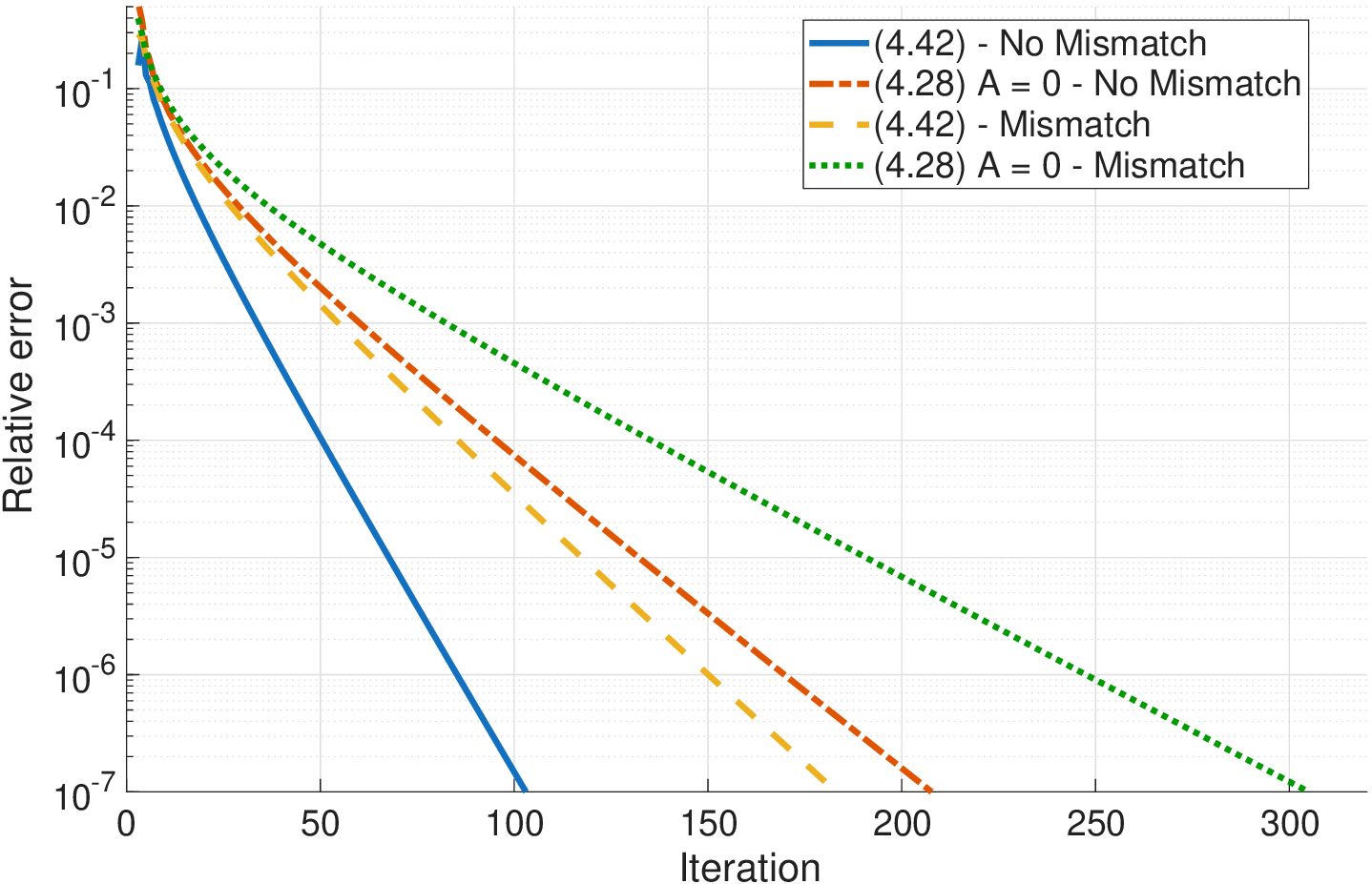}}
	\subfloat[Relative error vs time]{\label{fig:restheta13p}\includegraphics[width=0.32\textwidth]{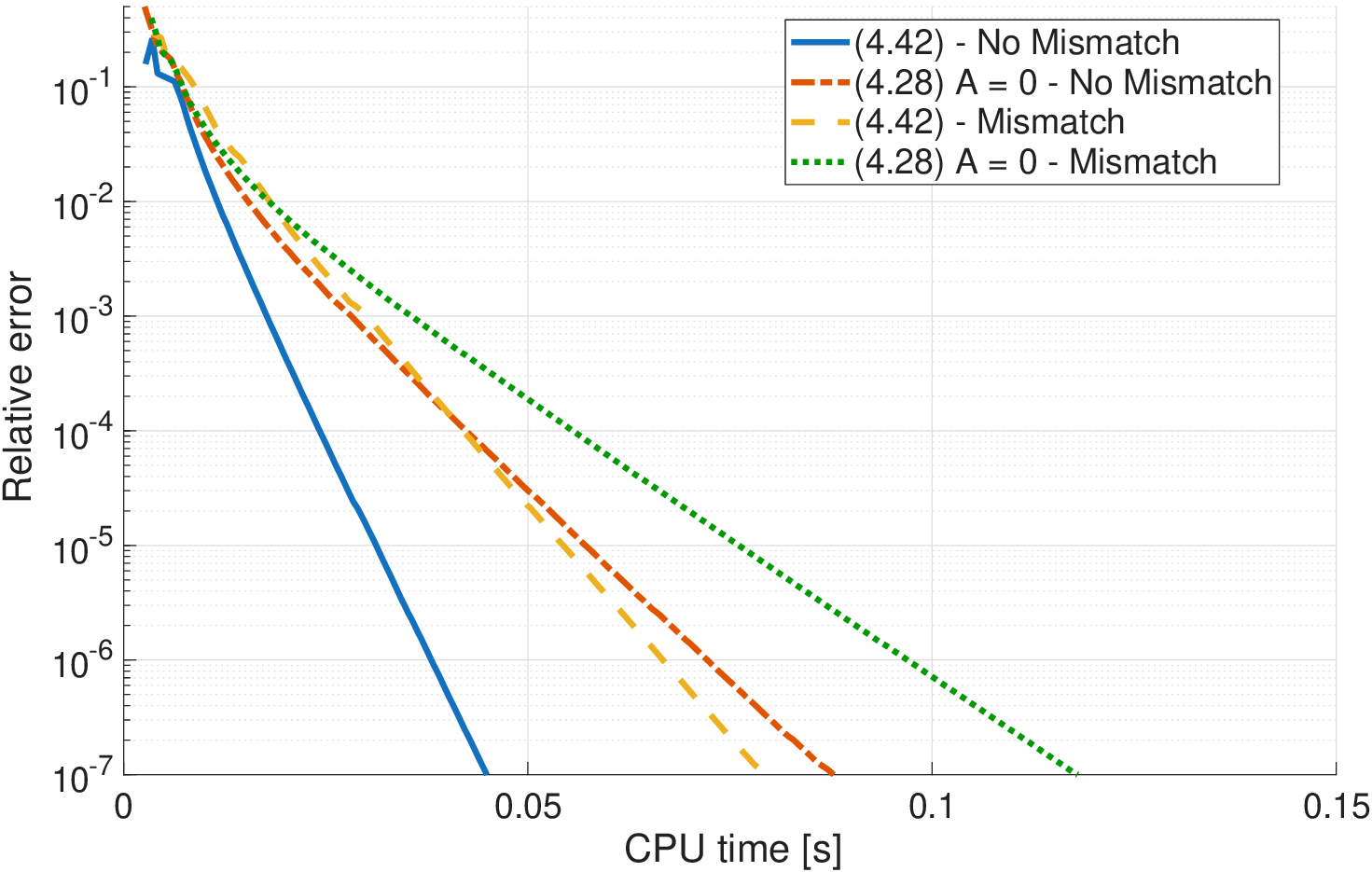}}
	\captionsetup{width=\textwidth} \caption{$\lambda = 0.5$ and $\theta=1$. The PSNR values of all restored signals are approximately 22.92~dB.}
	\label{fig:restheta1p}
\end{figure}
	\begin{figure}
	\centering
	\subfloat[Original and restored signals]{\label{fig:restheta141p}\includegraphics[width=0.32\textwidth]{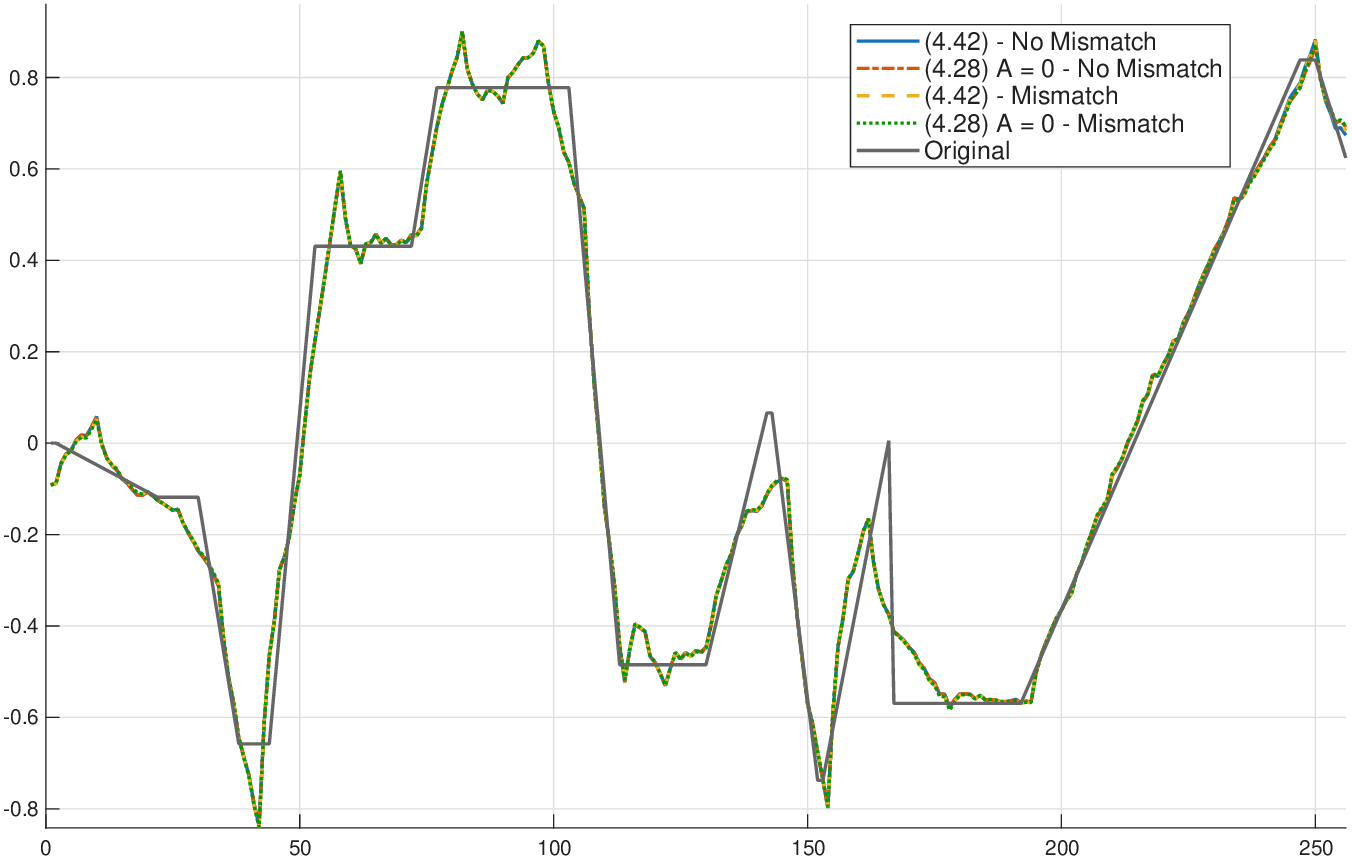}}\,
	\subfloat[Relative error vs iterations]{\label{fig:restheta142p}\includegraphics[width=0.32\textwidth]{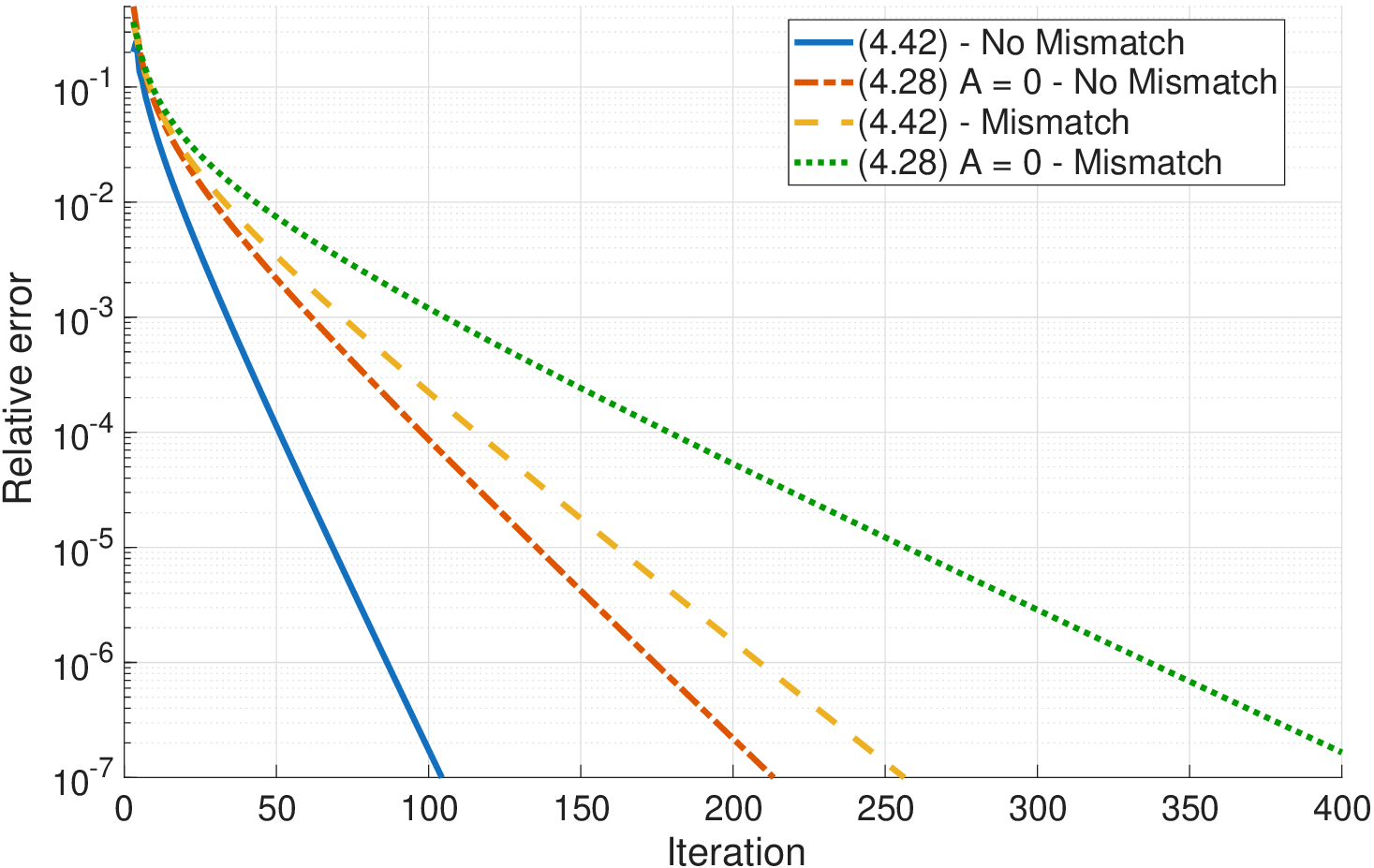}}
	\subfloat[Relative error vs time]{\label{fig:restheta143p}\includegraphics[width=0.32\textwidth]{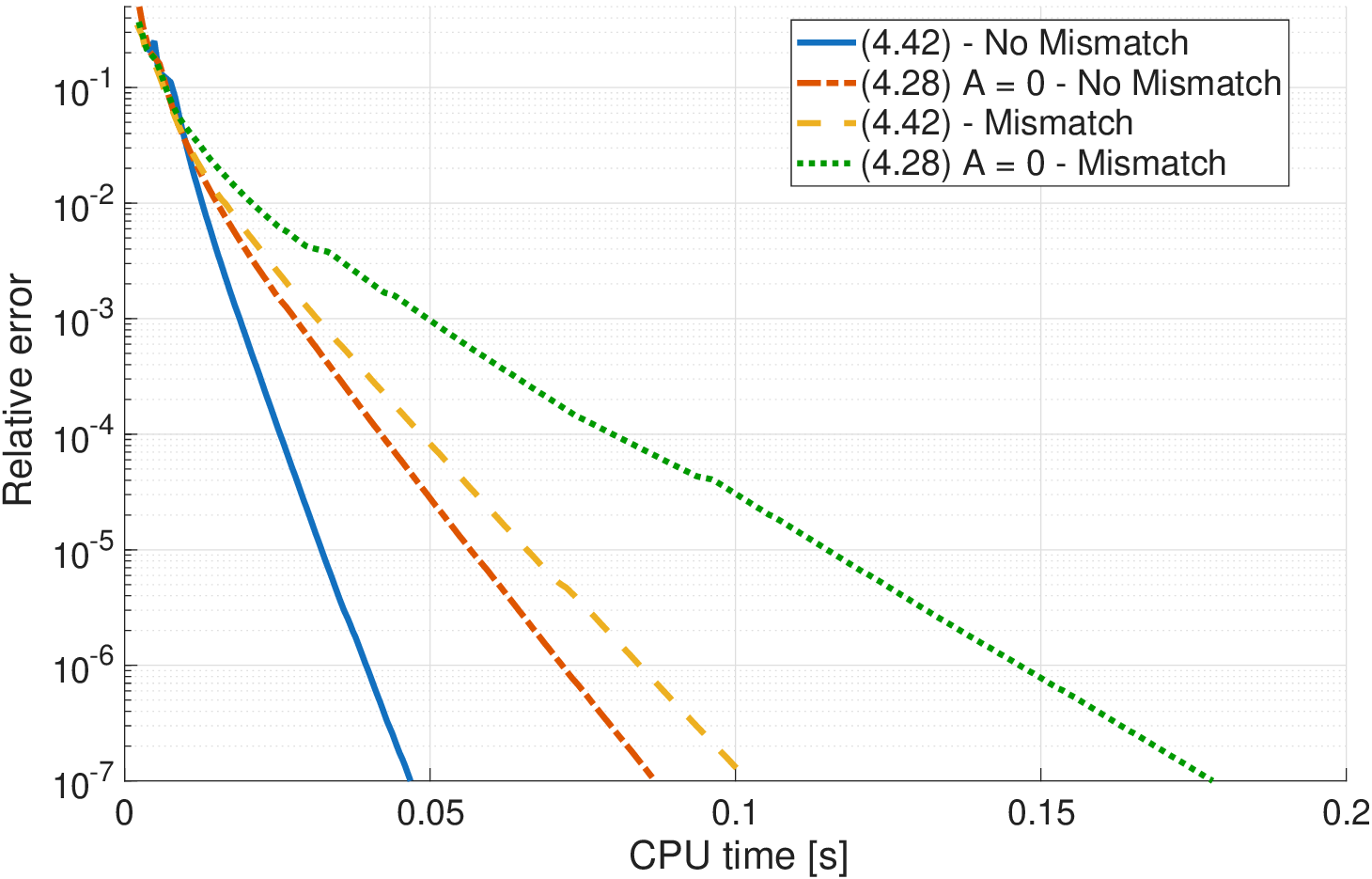}}
	\captionsetup{width=\textwidth} \caption{$\lambda = 0.5$ and $\theta=1.2$.  The PSNR values of all restored signals are approximately 22.92~dB.}	\label{fig:restheta14p}
\end{figure}
% 					\begin{table}[htbp]
% 	\centering
% 	\captionsetup{width=\linewidth}
% 	\renewcommand{\arraystretch}{1.2}
% 	\setlength{\tabcolsep}{6pt}
% 	\begin{tabular}{cccccc}
% 		Algorithm & $\theta$ & $\tau$ & IN & T & PSNR \\
% 		\toprule
% 		\multirow{3}{*}{\eqref{e:algonMT2} No Mismatch} 
% 		& 1 & 0.50 & 625 & 0.72 & 24.08 \\
% 		& 1.2 & 0.40 & 648 & 0.94 & 24.08 \\
% 		& 1.4 & 0.30 & 730 & 1.10 & 24.08 \\
% 		\hline
% 		\multirow{3}{*}{\eqref{e:algonMT} No Mismatch} 
% 		& 1 & 0.48 & 636 & 0.64 & 24.08 \\
% 		& 1.2 & 0.39 & 659 & 0.89 & 24.08 \\
% 		& 1.4 & 0.29 & 743 & 0.77 & 24.08 \\
% 		\hline
% 		\multirow{3}{*}{\eqref{e:algonMT2}} 
% 		& 1 & 0.39 & 761 & 0.34 & 24.08 \\
% 		& 1.2 & 0.29 & 836 & 0.42 & 24.08 \\
% 		& 1.4 & 0.19 & 1098 & 0.55 & 24.08 \\
% 		\hline
% 		\multirow{3}{*}{\eqref{e:algonMT}} 
% 		& 1 & 0.38 & 772 & 0.35 & 24.08 \\
% 		& 1.2 & 0.29 & 849 & 0.39 & 24.08 \\
% 		& 1.4 & 0.18 & 1120 & 0.97 & 24.08 \\
% 		\hline
% 	\end{tabular}
% 	\caption{}\label{Tab:result1}
% \end{table}     
\section{Conclusion}
In this article, we developed a nonlinear Forward-Backward algorithm for solving non-monotone inclusions involving Lipschitz continuous operators. We established the convergence of the method under a comonotonicity assumption and analyzed the conditions on the algorithm parameters in detail. In addition, we provided explicit conditions ensuring the comonotonicity of sums of operators and primal-dual operators. As a consequence, we established convergence guarantees for several methods from the literature for solving comonotone inclusions, including Forward-Backward, Chambolle--Pock, Condat--V\~u, and Forward-Reflected-Backward. Finally, we presented an application to adjoint mismatch problems and illustrated our theoretical results through a signal recovery example.

\section*{Acknowledgment}
The second author was partially supported by ANID through FONDECYT Iniciación Grant 11250164.

	\end{document}